\documentclass[12pt,final]{article}   

\usepackage{amsmath}
\usepackage{amsthm}
\usepackage{amssymb}
\usepackage{amsfonts}
\usepackage{latexsym}               
\usepackage{amsgen}
\usepackage[mathscr]{eucal}
\usepackage{mathrsfs}
\usepackage{bm}
\usepackage{enumerate}
\usepackage{cite}

\usepackage{mathptmx} 
\usepackage{color}
\usepackage[dvips]{graphicx}

\usepackage{showkeys}

\newtheorem{lem}{Lemma}[section]
\newtheorem{thm}[lem]{Theorem}
\newtheorem{pro}[lem]{Proposition}

\newtheorem{defn}[lem]{Definition}
\newtheorem{rem}[lem]{Remark}

\newcommand{\bqn}{\begin{equation}}
\newcommand{\eqn}{\end{equation}}
\newcommand{\beqx}{\begin{equation*}}
\newcommand{\eeqx}{\end{equation*}}
\newcommand{\barr}{\begin{array}}
\newcommand{\earr}{\end{array}}
\newcommand{\beqn}{\begin{eqnarray}}
\newcommand{\eeqn}{\end{eqnarray}}
\newcommand{\beqnx}{\begin{eqnarray*}}
\newcommand{\eeqnx}{\end{eqnarray*}}
\newcommand{\bmt}{\begin{multline}}
\newcommand{\emt}{\end{multline}}

\numberwithin{equation}{section}

\newcommand{\D}{\partial}

\newcommand{\supp}{\operatorname{supp}}

\newcommand{\ve}{\varepsilon}

\newcommand{\vphi}{{\varphi}}

\newcommand{\er}{\eqref}
\newcommand{\lb}{\label}

\newcommand{\pd}{\partial}

\newcommand{{\R}}{\mbox{$\mathbb R$}}

\begin{document}

\abovedisplayskip=8pt plus 2pt minus 4pt
\belowdisplayskip=\abovedisplayskip


\thispagestyle{plain}

\title{Nonlinear Stability and Instability  of  \\Plasma Boundary Layers} 

\author{%
{\large\sc Masahiro Suzuki${}^1$},
\\
{\large\sc Masahiro Takayama${}^2$},
\\
and
{\large\sc Katherine Zhiyuan Zhang${}^3$}
}

\date{ \today \\
\bigskip
\normalsize
${}^1$%
Department of Computer Science and Engineering, 
Nagoya Institute of Technology,
\\
Gokiso-cho, Showa-ku, Nagoya, 466-8555, Japan
\\ [7pt]
${}^2$%
Department of Mathematics, 
Keio University, 
\\ 
Hiyoshi, Kohoku-ku, Yokohama, 223-8522, Japan
\\ [7pt]
${}^3$%
Courant Institute of Mathematical Sciences, New York University, \\
New York, NY 10012, USA
} 

\maketitle

\begin{abstract}

We investigate the formation of a plasma boundary layer (sheath) by considering the Vlasov--Poisson system on a half-line with the completely absorbing boundary condition. In \cite{MM2}, the solvability of the stationary problem is studied. In this paper, we study the nonlinear stability and instability of these stationary solutions of the Vlasov--Poisson system.

\end{abstract}

\begin{description}
\item[{\it Keywords:}]
Vlasov--Poisson system,
Sheath,
Bohm criterion.

\item[{\it 2020 Mathematics Subject Classification:}]
35M13, 
35Q35, 
35B35, 
76X05. 

\end{description}


\newpage

\section{Introduction}

The goal of this paper is to investigate the stability and instability of a plasma boundary layer near the surface of materials immersed in a plasma, called a \emph{sheath}. The sheath appears when a material is surrounded by a plasma and the plasma contacts with its surface. Since electrons have much higher thermal velocities than ions, more electrons tend to hit the material compared with ions. Hence the material gets negatively charged, and this material with a negative potential attracts and accelerates ions toward the surface, and simultaneously repels electrons away from it. Eventually, a non-neutral potential region appears near the surface, where a nontrivial equilibrium of the densities is achieved. This non-neutral region is referred as the sheath. For more discussion on the physics in sheath development, we refer the reader to \cite{D.B.1, F.C.1,I.L.1, LL, K.R.1, K.R.2}.

It is observed by Langmuir \cite{I.L.1} that in order to form a sheath, positive ions must enter the sheath region with a sufficiently large flow velocity. Bohm \cite{D.B.1} proposed the original \emph{Bohm criterion} using the Euler--Poisson system. The Bohm criterion states that the flow velocity of positive ions at the plasma edge must exceed the ion acoustic speed. There have been a large amount of mathematical results investigating the sheath formation by using the Euler--Poisson system ever since. The studies \cite{A.A.,AMR,NOS,M.S.1,M.S.2,MM1} established theories of existence and stability of stationary solutions assuming the Bohm criterion, and verified this criterion in a rigorous mathematical sense. There are also varied studies from other perspectives on the sheath. In \cite{GHR1,GHR2,JKS1,JKS2,JKS3}, the sheath formation was discussed by considering the quasi-neutral limit 
as letting the Debye length in the Euler--Poisson system tend to zero. Moreover, in \cite{FHS1}, the sheath formation was considered by adopting a certain hydrodynamic model which describes the dynamics of an interface between the plasma and sheath, see also \cite{RD1}.

Beside those studies using hydrodynamic models, it is also of important to analyze the sheath formation from a kinetic point of view. Boyd--Thompson \cite{BT1} proposed a kinetic Bohm criterion (see \eqref{Bohm1} below). Then Riemann \cite{K.R.1} pointed out (although without a rigorous proof) that the kinetic Bohm criterion \eqref{Bohm1} is a necessary condition for the boundary value problems of the stationary Vlasov--Poisson system to be solvable. Recently, the first two authors \cite{MM2} analyzed rigorously the solvability of the stationary problem, and clarified in all possible cases whether or not there is a stationary solution. It was concluded that the Bohm criterion is necessary but not sufficient for the solvability. In this paper, we study the nonlinear stability and instability of the stationary solutions of the Vlasov--Poisson system.

The Vlasov--Poisson system is written as
\begin{subequations}\label{VP1}
\begin{gather}
\partial_{t} F+\xi_{1} \partial_{x} F+\partial_{x} \Phi \partial_{\xi_{1}}F = 0, \quad t>0, \  x>0 , \ \xi \in \mathbb R^{3},
\label{boltz1}
\\
\partial_{xx} \Phi = \int_{\mathbb R^{3}} F d\xi - n_{e}(\Phi), \quad t>0, \  x>0,
\label{poisson1}
\end{gather}
where $t>0$, $x>0$, and $\xi = (\xi_{1},\xi_{2},\xi_{3})=(\xi_{1},\xi') \in \mathbb R^{3}$ are the time, spatial, and velocity variables, respectively.
The unknown functions $F = F(t,x,\xi)$ and $-\Phi=-\Phi(t,x)$ stand 
for the velocity distribution of positive ions
and the electrostatic potential, respectively. 
The given function $n_{e}(\cdot) \in C^{2}(\mathbb R)$ denotes the number density of electrons.
We assume that 
\begin{gather}\label{n_e}
n_{e}(0)=1, \quad n_{e}'(0)=-1, \quad n_{e}'(\Phi)<0.
\end{gather}
One of typical functions is the Boltzmann relation $n_{e}(\Phi)=e^{-{\Phi}}$.
We study the initial--boundary value problem of \eqref{boltz1}--\eqref{n_e} with the following initial and boundary data:
\begin{gather}
F (0,x,\xi) = F_{0}(x,\xi), \quad x>0, \ \xi \in \mathbb R^{3},
\label{ice}\\
F (t,0,\xi) = 0, \quad t>0, \ \xi_{1}>0,
\label{bce1} \\
\Phi(t,0)=\Phi_{b}>0, \quad t>0,
\label{bce3} \\
\lim_{x \to\infty} F (t,x,\xi) =  F_{\infty}(\xi),
\quad t>0,  \ \xi \in \mathbb R^{3},
\label{bce2} \\
\lim_{x \to\infty} \Phi (t,x)=0, 
\quad t>0.
\label{bce4} 
\end{gather}
\end{subequations}
The constant $-\Phi_{b} \in \mathbb R$ denotes the voltage on the boundary.
Furthermore, $F_{0}=F_{0}(x,\xi)$ and $F_{\infty}=F_{\infty}(\xi)$ are nonnegative functions.
The boundary condition \eqref{bce1} on $\xi_{1}>0$ does not allow positive ions to flow into the domain from the boundary.
On the other hand, we do not put any boundary condition on $\xi_{1}<0$, and therefore positive ions can flow out through the boundary.
In other words, we consider the {\it completely absorbing} boundary condition for positive ions.

As mentioned above, the studies \cite{K.R.1,MM2} investigated the solvability of the stationary problem associated to the problem \eqref{VP1}. The stationary solution $(F^{s},\Phi^{s}) = (F^{s}(x,\xi),$ $\Phi^{s}(x))$ solves
\begin{subequations}\label{sVP1}
\begin{gather}
\xi_{1} \partial_{x} F^{s}+\partial_{x} \Phi^{s} \partial_{\xi_{1}} F^{s} =0, \quad x>0 , \ \xi \in \mathbb R^{3},
\label{seq1}
\\
\partial_{xx} \Phi^{s} 
= \int_{\mathbb R^{3}} F^{s} d\xi - n_{e}(\Phi^{s}), \quad x>0,
\label{seq2}
\\
F^{s} (0,\xi) = 0, \quad \xi_{1}>0,
\label{sbc1} \\
\Phi^{s}(0)=\Phi_{b}>0,
\label{sbc3} \\
\lim_{x \to\infty} F^{s} (x,\xi) =  F_{\infty}(\xi), 
\quad \xi \in \mathbb R^{3},
\label{sbc2} \\
\lim_{x \to\infty} \Phi^{s} (x) =  0.
\label{sbc4} 
\end{gather}
\end{subequations}
It was concluded by \cite{K.R.1,MM2} that the following Bohm criterion is a necessary condition 
for the solvability of the stationary problem \eqref{sVP1}:
\begin{gather}\label{Bohm1}
\int_{\mathbb R^{3}}\xi_{1}^{-2}F_{\infty}(\xi) d\xi  \leq  1.
\end{gather}
In this paper, we consider the nonlinear stability and instability of these stationary solutions.

We review mathematical results on the Vlasov--Poisson system describing the motion of plasma. 
%
%
The time-global solvability and the regularity of solutions were investigated extensively. Early references \cite{BD, LP, K.P., J.S.} and books \cite{glassey, G.R.} considered the Cauchy problem. Moreover, \cite{Y.G.1, Y.G.2, HV1} studied the initial--boundary value problem in a half space imposing various boundary conditions such as the specular reflecting and the absorbing boundary conditions. 

%
%

The stability and instability of equilibria (spatially homogeneous stationary solutions) have been extensively studied. The articles \cite{Penrose1,HMRW1,HNR} investigated the stability of equilibria on a whole space. The studies \cite{P.D.,BR1,MPN1} discussed the stability with periodic boundary condition on the spatial variable. The nonlinear instability was shown on a 3D cube with the specular reflecting boundary condition in \cite{GS3}. Nowadays, the problem of nonlinear Landau damping has drawn huge attention, see \cite{MV1, BMM1, GNR1, IPWW1, LZ1, LZ2}.

The analysis on stationary solutions (spatially inhomogeneous stationary solutions) is more important in application of plasmas. Rein \cite{G.R.2}, Esent\"urk--Hwang--Strauss \cite{EHS}, and Suzuki--Takayama \cite{MM2} established theories for the existence on various domains. Batt--Morrison--Rein \cite{BMR1} showed the existence and linear stability of stationary solutions on a whole space.  Esent\"urk--Hwang \cite{EH1} discussed those on a finite interval and half line with the inflow and specular reflecting boundary conditions, respectively. The linear instability was also analyzed in \cite{Z.L.1}.
It is of greater interest to investigate the nonlinear stability and instability of stationary solutions.
Rein \cite{G.R.1} proved the nonlinear orbital stability on a whole space by using the energy-Casimir method. 
Guo--Strauss \cite{GS1} and Lin \cite{Z.L.2} investigated the criterions for the nonlinear instability with periodic boundary condition on the spatial variable. See also \cite{GS2,LS1} which studied the Vlasov--Maxwell system.

In this paper, we consider the Vlasov--Poisson system on a half-line with the completely absorbing boundary condition under the assumption that the electron density obeys $n_{e}=n_{e}(\Phi)$, and establish the nonlinear stability and instability of stationary solutions in a weighted $L^2$ energy framework. We make use of the dissipative nature of the problem by exploiting the completely absorbing boundary condition (in particular, the outflow part on $\xi_{1}<0$) as well as observing the transport of the support of the velocity distribution $F$. It is concluded that the $\xi$-support of the initial data $F_{0}$ is a major factor leading to stability/instability. Roughly speaking, the $\xi$-support being constrained to $\{\xi \in \mathbb R^{3} \, | \, \xi_{1}<0 \}$ helps in stability; 
A nonempty intersection of the $\xi$-support and $\{\xi \in \mathbb R^{3} \, | \, \xi_{1}>0 \}$ helps in instability.
To our knowledge, this is the first asymptotic stability result on spatially inhomogeneous stationary solutions of the Vlasov--Poisson system.

This paper is organized as follows.
Section \ref{S2} provides our main theorems (Theorems \ref{mainThm} and \ref{mainThm2}) on the nonlinear stability and instability.
In Section \ref{S3}, we reformulate the problem \eqref{VP1} and reduce the proofs of the main theorems to the key Propositions \ref{local1}, \ref{apriori1}, and \ref{apriori2}. Section \ref{S4} is devoted to preliminaries to be applied in the proofs of the key propositions. Specifically, we discuss estimates on the hydrodynamic variables, estimates on the stationary solution $F^{s}$, and elliptic estimates for the Poisson equation in subsections \ref{SS:HV}, \ref{SS:Fs}, and \ref{ES}, respectively. Then, Section \ref{LS} gives a proof of Proposition \ref{local1} on the time-local solvability of the reformulated problem. In Sections \ref{AE} and \ref{AE2}, we prove Proposition \ref{apriori1} and \ref{apriori2} which lead to the stability and instability theorems, respectively. In addition, in Appendix \ref{SA}, we discuss the selection of the constants so that the condition \eqref{asp1} in Proposition \ref{apriori1} holds by assuming either the condition (i) or condition (ii) in Theorem \ref{mainThm}.  
Appendix \ref{SB} provides a study of the solvability of the linearized problem.

Before closing this section, we give our notation used throughout this paper.

\medskip

\noindent
{\bf Notation.} 
Let  $\pd_{t} := \frac{\pd}{\pd t}$, $\pd_x := \frac{\pd}{\pd x}$, and $\pd_{\xi_{j}} := \frac{\pd}{\pd \xi_{j}}$.
The operators $\nabla_{x,\xi} := (\pd_{x},\pd_{\xi_{1}} ,\pd_{\xi_{2}},\pd_{\xi_{3}})$ and
$\nabla_{\xi} := (\pd_{\xi_{1}} ,\pd_{\xi_{2}},\pd_{\xi_{3}})$ 
denote the gradients with respect to $(x,\xi)$ and $\xi$, respectively.
Furthermore, $\mathbb R_{+}:=\{x>0\}$ is a one-dimensional half space;
$\mathbb R_{+}^{3}:=\{\xi \in \mathbb R^{3} \, | \, \xi_{1}>0 \}$ is a three-dimensional upper half space;
$\mathbb R_{-}^{3}:=\{\xi \in \mathbb R^{3} \, | \, \xi_{1}<0 \}$ is a three-dimensional lower half space.
We use the one-dimensional indicator function $\chi(s)$ of the set $\{s>0\}$.

For a domain $\Omega$ and $p \in [1,\infty]$, $L^p(\Omega)$ is the Lebesgue space
equipped with the norm $\Vert\cdot\Vert_{L^{p}(\Omega)}$.
Let us denote by $(\cdot,\cdot)_{L^{2}(\Omega)}$ the inner product of $L^{2}(\Omega)$.
For an integer $k \geq 0$, $H^k(\Omega)$ means the Sobolev space in the $L^2$ sense.
Furthermore, we use the function space $H^1_{0,\Gamma}:=\{ f \in H^{1}({\mathbb R}_{+}\times{\mathbb R}^{3}) \, | \,  f(0,\xi)=0, \ \xi \in \mathbb R^{3}_{+}  \}$.
We sometimes abbreviate the spaces $L^{2}({\mathbb R}_{+}\times{\mathbb R}^{3})$, 
$L^{2}({\mathbb R}_{+})$,  $L^{2}({\mathbb R}^{3})$, 
$H^{k}({\mathbb R}_{+}\times{\mathbb R}^{3})$, and $H^{k}({\mathbb R}_{+})$
by $L^{2}_{x,\xi}$, $L^{2}_{x}$, $L^{2}_{\xi}$, $H^{k}_{x,\xi}$, and $H^{k}_{x}$, respectively. 
For $T>0$ and $\alpha \geq 0$, a solution space ${\cal X}_{\rm e}(T,\alpha) \times {\cal Y}_{\rm e}(T,\alpha)$  is defined by
\begin{align*}
{\cal X}_{\rm e}(T,\alpha)&:=\left\{ f  \, \left| \, 
\begin{array}{ll}
e^{\alpha x/2} f \in C([0,T];H^{1}({\mathbb R}_{+}\times{\mathbb R}^{3})),
\\
e^{\alpha x/2} (1+|\xi_{1}|)^{-1} f \in C^{1}([0,T];L^{2}({\mathbb R}_{+}\times{\mathbb R}^{3})),
\\
e^{\alpha x /2} |\xi_{1}|^{1/2} \nabla_{x,\xi}^{k} f \in L^{2}(0,T;L^{2}({\mathbb R}_{+}\times{\mathbb R}^{3})) \quad \text{for $k=0,1$}
\end{array}
\right.\right\},
\\
{\cal Y}_{\rm e}(T,\alpha)&:=\{ \phi \, | \, e^{\alpha x /2}\phi \in C([0,T];H^{3}({\mathbb R}_{+}))\}.
\end{align*}
For a nonnegative integer $k$, ${\cal {B}}^k (\mathbb R)$ stands for
the space of functions whose derivatives up to  
$k$-th order are continuous and bounded over $\mathbb R$.

\section{Main Results}\label{S2}
In this section, we state our main results on 
the nonlinear stability and instability of the stationary solutions $(F^{s},\Phi^{s})$ of the Vlasov--Poisson system.
We start from introducing the result in \cite{MM2} studying the existence of stationary solutions.
To do so, we give a definition of stationary solutions.
Note that the potential $-\Phi^{s}$ is observed as an increasing function when the plasma sheath is formed.
\begin{defn}\label{DefS1}
We say that $(F^{s},\Phi^{s})$ is a solution of 
the boundary value problem \eqref{sVP1} if it satisfies the following:
\begin{enumerate}[(i)]
\item $F^{s} \in L^{1}_{loc}(\overline{\mathbb R_{+} \times \mathbb R^{3}}) \cap C(\overline{\mathbb R_{+}};L^{1}(\mathbb R^{3}))$
and $\Phi^{s} \in C^{1}(\overline{\mathbb R_{+}}) \cap C^{2}(\mathbb R_{+})$;
\item $F^{s}(x,\xi)\geq 0$ and $\D_{x}\Phi^{s}(x)<0$;
\item $F^{s}$ solves the following weak form:
\begin{gather*}
(F^{s},\xi_1\D_x\psi
+\D_x{\Phi^{s}}\D_{\xi_1}\psi)_{L^2({\mathbb R}_+\times {\mathbb R}^3)}
=0, \quad \forall \psi \in {\cal X},  
\\
\lim_{x\to \infty} \| F^{s}(x,\cdot)-F_{\infty} \|_{L^{1}(\mathbb R^{3})}=0,
\end{gather*}
where 
${\cal X}:=\{f\in C_0^1(\overline{\mathbb{R}_+ \times \mathbb{R}^3})\, | \,
\hbox{$f(0,\xi)=0$ for $\xi \in {\mathbb R}_-^3$}\}$;
\item $\Phi^{s}$ solves \eqref{seq2} with \eqref{sbc3} and \eqref{sbc4} in the classical sense.
\end{enumerate}
\end{defn}

To solve the Poisson equation \eqref{seq2} with \eqref{sbc4}, we must require the quasi-neutral condition
\begin{gather}\label{netrual1}
\int_{\mathbb R^{3}} F_{\infty}(\xi) d\xi = 1.
\end{gather}  
The following is also a necessary condition (for more details, see Lemma 3.1 in \cite{MM2}):
\begin{gather}
F_{\infty}(\xi)=0, \quad \xi_{1} > 0.
\label{need2} 
\end{gather}
The solvability of \eqref{sVP1} is summarized in Proposition \ref{existence1} below. 
Here the set $B$ is defined by
\begin{align*}
B:=\{\varphi>0 \, | \, V(\phi)>0 \ \text{for} \ \phi \in (0,\varphi] \},
\end{align*}
where
\begin{align*}
V(\phi):=\int_{0}^{\phi} \left( \int_{\mathbb R^{3}} F_{\infty}(\xi)\frac{-\xi_{1}}{\sqrt{\xi_{1}^{2}+2\varphi}} d\xi - n_{e}(\varphi) \right)  d\vphi.
\end{align*}

\begin{pro}[\hspace{-0.1mm}\cite{MM2}]\label{existence1}
Let $\Phi_{b}>0$ and $F_{\infty} \in L^{1}(\mathbb R^{3})$ 
satisfy $F_{\infty} \geq 0$, \eqref{netrual1}, \eqref{need2}, and
\begin{gather}
\int_{\mathbb R^{3}}\xi_{1}^{-2}F_{\infty}(\xi) d\xi <1.
\label{Bohm2}
\end{gather}
Then the set $B$ is not empty. 
Furthermore, if and only if $\Phi_{b}<\sup B$ holds, the problem \eqref{sVP1} has a unique solution $(F^{s},\Phi^{s})$.
There also hold that 
\begin{gather}
F^{s}(x,\xi)=F_{\infty}(-\sqrt{\xi_{1}^{2}-2\Phi^{s}(x)},\xi')\chi(\xi_{1}^{2}-2\Phi^{s}(x))\chi(-\xi_{1}),
\label{fform} \\
|\D_{x}^{l} \Phi^{s}(x)| \leq C e^{- c x} 
\quad \text{for $l=0,1,2$},
\label{decay1}
\end{gather}
where $\chi(s)$ is the one-dimensional indicator function of the set $\{s>0\}$, and 
$c$ and $C$ are positive constants independent of $x$. 
In addition, if $F_{\infty} \in C^{2}(\mathbb R^{3})$, 
then the solution $(F^{s},\Phi^{s}) \in C^{1}(\overline{\mathbb R_{+} \times \mathbb R^{3}}) \times C^{2}(\overline{\mathbb R_{+}})$ is a classical solution.

%
%
%
\end{pro}

Owing to $\partial_{x} \Phi^{s}(x)<0$, \eqref{sbc3}, and \eqref{sbc4}, there holds that
\begin{subequations}\label{decay0}
\begin{gather}
0 \leq \Phi^{s}(x) \leq \Phi_{b}.
\end{gather}
Further, in the case $\Phi_{b} \ll 1$, the inequality \eqref{decay1} can be improved slightly as
\begin{gather}
|\D_{x}^{l} \Phi^{s}(x)| \leq C \Phi_{b} e^{- c x} 
\quad \text{for $l=0,1,2$},
\end{gather}
\end{subequations}
where $c$ and $C$ are positive constants independent of $x$ and $\Phi_{b}$.
We also remark that the study \cite{MM2} analyzed the case $\int_{\mathbb R^{3}}\xi_{1}^{-2}F_{\infty}(\xi) d\xi \geq 1$, 
and the following general boundary condition imposed for $F^{s}$ as well:
\begin{gather*}
F^{s} (0,\xi) = F_{b}(\xi)+\alpha  F^{s}(0,-\xi_{1},\xi'), \quad \xi_{1}>0.
\end{gather*}

We study the stability and instability for some special end states $F_{\infty}$. 
From a physical point of view, 
it seems natural that the end state $F^{\infty}$ is given 
by the following Maxwellian $M_{\infty}=M_{\infty}(\xi)$ with constants $\rho_{\infty}>0$, 
$(u_{\infty},0,0) \in \mathbb R^{3}_{-}$, and $\theta_{\infty}>0$:
\begin{equation*}
M_{\infty}(\xi):=\frac{\rho_{\infty}}{(2\pi \theta_{\infty})^{3/2}}\exp\left(-\frac{|\xi-(u_{\infty},0,0)|^{2}}{2\theta_{\infty}} \right).
\end{equation*}
However, due to the necessary condition \eqref{need2}, 
we have no stationary solution for the case $F_{\infty}=M_{\infty}$.
To avoid this issue, we employ  a cut-off function $\psi \in C^{\infty}({\mathbb R}^{3})$ 
which is nonnegative and satisfies
\begin{gather}\label{cutoff0}
\psi(\xi)=\psi(\xi_{1})=
\left\{
\begin{array}{ll}
0 &  \text{if $\xi_{1} \geq -r$},
\\
1 & \text{if $\xi_{1} \leq - r - \sigma$},
\end{array}
\right.
\end{gather}
where $0<r<|u_{\infty}|$ and $0<\sigma<(|u_{\infty}|-r)/2$, and then set
\begin{gather}\label{cutoff1}
F_{\infty}=M_{\infty} \psi.
\end{gather}
We notice that $M_{\infty} \psi$ is a good approximation of $M_{\infty}$ in the case $\theta_{\infty} \ll 1$ and
\begin{gather*}
{\rm supp} F_{\infty}={\rm supp} \psi \subset \{\xi \in {\mathbb R}^{3} \, | \, \xi_{1}\leq -r \}.
\end{gather*}
Throughout this paper, we first fix a certain cut-off function $\psi$ and then \footnote{Note that $|u_{\infty}| > r + 2\sigma$.}choose $u_{\infty}$. 

From Theorem \ref{existence1}, it is seen that 
the problem \eqref{sVP1} with $F_{\infty}=M_{\infty} \psi $ has a unique classical stationary solution $(F^{s},\Phi^{s})$
for parameters $\rho_{\infty}$, $u_{\infty}$, and $\theta_{\infty}$ so that \eqref{netrual1} and \eqref{Bohm2} hold.
From \eqref{fform} and $\Phi^{s}(x) >0$, we also see that
\begin{gather}\label{r2}
{\rm supp} F^{s} \subset \{ (x,\xi) \in \overline{\mathbb R_{+} \times {\mathbb R}^{3}} \, | \, \xi_{1}\leq -r \}.
\end{gather}

We are now in a position to state our main theorems.
The first one is the stability theorem for the initial perturbation $F_{0}-F^{s}$ whose support is located in $\overline{\mathbb R_{+}} \times \mathbb R_{-}^{3}$.
\begin{thm}[Nonlinear stability]\label{mainThm}
Let $\ve \in (0,r/2)$ for $r>0$ being in \eqref{cutoff0}, and $F_{\infty}=M_{\infty}\psi$ satisfy \eqref{netrual1} and \eqref{Bohm2}. 
Suppose that either of the following conditions (i)--(ii) holds:
\begin{enumerate}[(i)]
\item $r-2\ve\geq 1$, $|u_{\infty}|>1$, and $\theta_{\infty} \leq \theta_{*}$ for some positive constant $\theta_{*}=\theta_{*}(r,\ve,u_{\infty})$;
\item $|u_{\infty}| \geq  u_{*}$ for some positive constant $u_{*}=u_{*}(r,\ve,\theta_{\infty})$.
\end{enumerate}
Then there exist constants $\beta \in (0,1)$ and $\delta_{0}>0$
such that if 
\begin{gather*}
{\rm supp} F_{0} \subset \{ (x,\xi) \in \overline{{\mathbb R}_{+}\times{\mathbb R}^{3}} \, | \,  \xi_{1} \leq - r + \ve \},
\\
\Phi_{b} + \|e^{\beta x/2} M_{\infty}^{-1/2} (F_{0}-F^{s})\|_{H^{1}({\mathbb R}_{+}\times{\mathbb R}^{3})} \leq \delta_{0},
\end{gather*}
the initial--boundary value problem \eqref{VP1} has
a unique solution $(F,\Phi)$ that satisfies 
$(M_{\infty}^{-1/2} $ $(F-F^{s}),\Phi-\Phi^{s}) \in {\cal X}_{\rm e}(\infty,\beta) \times {\cal Y}_{\rm e}(\infty,\beta)$.
Furthermore, there hold that
\begin{gather*}
{\rm supp} F \subset \{ (t,x,\xi) \in \overline{{\mathbb R}_{+}\times{\mathbb R}_{+}\times{\mathbb R}^{3}} \, | \,  \xi_{1} \leq - r + 2\ve \},
\\
\| e^{\beta x/2}M_{\infty}^{-1/2} (F-F^{s}) \|_{H^{1}({\mathbb R}_{+}\times{\mathbb R}^{3})}
+ \| e^{\beta x /2} (\Phi-\Phi^{s}) \|_{H^{3}({\mathbb R}_{+})} 
\leq C e^{-\gamma t /2} \quad \text{for $t>0$},
\end{gather*}
where $\gamma$ and $C$ are some positive constants independent of $t$.
\end{thm}

\begin{rem}\label{rem1}
{\rm
We remark that the Vlasov--Poisson system \eqref{boltz1}--\eqref{poisson1} itself have no dissipation, but a dissipative effect occurs on the boundary $\partial ({\mathbb R}_{+} \times \mathbb R^{3})=\{0\} \times \mathbb R^{3}$ thanks to the completely absorbing boundary condition, in particular, the outflow part on $\{0\} \times \mathbb R^{3}_{-}$.
A weight function $e^{\beta x/2}$ is employed to exploit the dissipative nature of the problem in the whole domain ${\mathbb R}_{+} \times \mathbb R^{3}$.

The cold plasma case $\theta_{\infty} \ll 1$ corresponding the condition (i) is a regime of major interest in plasma physics.
On the mathematical side, we need delicate analysis to treat this case, since the rightmost term in the equation \eqref{ree1} below for new unknown functions diverges as $\theta_{\infty} \to 0$. 
To resolve this issue, we make use of another weight function $M_{\infty}^{-1/2}$ and some hydrodynamic variables. We will explain more details of the difficulties in Section \ref{S3}.

Theorem \ref{mainThm} allows the situation that ${\rm supp} F^{s} \subset {\rm supp} F_{0}$ owing to \eqref{r2}. Furthermore, in condition (i), it is reasonable to assume $|u_{\infty}|> 1$. Indeed, $F^{\infty}=M_{\infty}\psi$ with the quasi-neutral constraint \eqref{netrual1} converges to a delta function $\delta(\xi-(u_{\infty},0,0))$ as $\theta_{\infty} \to 0$. This fact together with \eqref{Bohm2} requires the condition $|u_{\infty}|>1$. On the other hand, the condition $r-2\ve \geq 1$ gives essentially no restriction, since we can choose $r>1$ and $\ve \ll 1$.
We can show the asymptotic stability under a weaker assumption \eqref{asp1} than the conditions (i)--(ii), although the expression of the weaker assumption is much more complicated. For more details, see Theorem \ref{stability} below.
}
\end{rem}

For the initial perturbation $F_{0}-F^{s}$ whose support is located in $\overline{\mathbb R_{+}} \times \mathbb R_{+}^{3}$, we have the following instability theorem:

\begin{thm}[Nonlinear instability]\label{mainThm2}
Let $\beta \in (0,1)$ and $F_{\infty}=M_{\infty}\psi$ satisfy \eqref{netrual1} and \eqref{Bohm2}. 
Suppose that $g_{0}$ satisfies
\begin{subequations}\label{g0}
\begin{gather}
\supp g_{0} \subset \{ (x,\xi) \in \mathbb R_{+} \times \mathbb R^{3} \, | \, R_{1} \leq \xi_{1}  \leq R_{2} \},
 \\
\|e^{\beta x/2} g_{0}\|_{H^{1}({\mathbb R}_{+}\times{\mathbb R}^{3})}=1
\end{gather}
\end{subequations}
for some constants $R_{1}>0$ and $R_{2}>2R_{1}$.
Then there exists a constant $\ve>0$ such that for any sufficiently small $\delta>0$,
the initial--boundary value problem \eqref{VP1} with $(F_{0},\Phi_{b})=(\delta  M_{\infty}^{1/2} g_{0} +F^{s},\delta)$ has a solution $(F,\Phi)$ that satisfies the following:
\begin{gather*}
(M_{\infty}^{-1/2} (F-F^{s}),\Phi-\Phi^{s}) \in {\cal X}_{\rm e}(T,\beta) \times {\cal Y}_{\rm e}(T,\beta),
\\
\|e^{\beta x/2} M_{\infty}^{-1/2} (F-F^{s})(T)\|_{H^{1}({\mathbb R}_{+}\times{\mathbb R}^{3})} \geq \ve
\end{gather*}
 for some $T>0$.
\end{thm}

\begin{rem} {\rm
We employ the same solution space ${\cal X}_{\rm e}(T,\beta) \times {\cal Y}_{\rm e}(T,\beta)$
in studying both the stability and instability stated in Theorems \ref{mainThm} and \ref{mainThm2}. 
For $g_{0}$ satisfying
\begin{gather*}
\supp g_{0} \subset \{ (x,\xi) \in \mathbb R_{+} \times \mathbb R^{3} \, | \, R_{1} \leq \xi_{1}  \leq R_{2} \} \cup \{ (x,\xi) \in \overline{\mathbb R_{+} \times \mathbb R^{3}} \, | \, \xi_{1} \leq -r  \},
\\
\supp g_{0} \cap \{ (x,\xi) \in \mathbb R_{+} \times \mathbb R^{3} \, | \, R_{1} \leq \xi_{1}  \leq R_{2} \} \neq \emptyset,
 \\
\|e^{\beta x/2} g_{0}\|_{H^{1}({\mathbb R}_{+}\times{\mathbb R}^{3})}=1,
\end{gather*}
we can show an instability theorem in the same way as in the proof of Theorem \ref{mainThm2}.
}
\end{rem}

\section{Reformulation and Reduction}\label{S3}

First we reformulate the problem \eqref{VP1}.  Let us introduce new unknown functions $(f,\phi)$ as
\begin{equation}\label{pert1}
f=f(t,x,\xi):=M_{\infty}^{-1/2}(\xi)(F(t,x,\xi)-F^{s}(x,\xi)), \quad 
\phi=\phi(t,x):=\Phi(t,x)-\Phi^{s}(x). 
\end{equation}
These solve the following initial--boundary value problem:
\begin{subequations}\label{ree0}
\begin{gather}
\begin{aligned}
& \pd_{t} f + \xi_{1}\partial_{x} f + (\partial_{x} \Phi^{s} + \partial_{x}\phi) \partial_{\xi_{1}}f + \frac{\pd_{\xi_{1}}F^{s}}{M_{\infty}^{1/2}} \partial_{x} \phi
\\
&= \frac{\xi_{1}-u_{\infty}}{2\theta_{\infty}}  (\partial_{x} \Phi^{s} + \partial_{x}\phi) f, \quad t>0, \  x>0 , \   \xi \in \mathbb R^{3},
\end{aligned}
\label{ree1}\\
\partial_{xx} \phi =  \int_{\mathbb R^{3}}  M_{\infty}^{1/2} f d\xi - (n_{e}(\Phi^{s}+\phi)-n_{e}(\Phi^{s})), \quad t>0, \  x>0
\label{ree2}
\end{gather}
with
\begin{gather}
f(0,x,\xi) = f_{0}(x,\xi) :=  M_{\infty}^{-1/2}(\xi) (F_{0}(x,\xi) - F^{s}(x,\xi)), \quad x>0, \   \xi \in \mathbb R^{3},
\label{reic1}\\
f(t,0,\xi) = 0, \quad t>0, \ \xi_{1} >0,
\label{rebc1}\\
\phi(t,0)=0, \quad t>0, 
\label{rebc2}\\
\lim_{x \to\infty} f(t,x,\xi) = 0,  \quad t>0,  \ \xi \in \mathbb R^{3},
\label{rebc3}\\
\lim_{x \to\infty} \phi(t,x) = 0,  \quad t>0,
\label{rebc4}
\end{gather}
\end{subequations}
where we need to assume the compatibility condition, i.e.
\begin{gather*}
f_{0}(0,\xi) = 0, \quad \xi_{1} >0.
\end{gather*}

As mentioned in Remark \ref{rem1},
the first difficulty in the proof of asymptotic stability is that the Vlasov--Poisson system \eqref{boltz1}--\eqref{poisson1} has no dissipation. However,  since we impose the completely absorbing boundary condition, 
we can exploit some dissipative effect on the outflow part $\{0\}\times \mathbb R^{3}_{-}$ of the boundary $\partial ({\mathbb R}_{+} \times \mathbb R^{3})=\{0\} \times \mathbb R^{3}$ by employing the weight function $e^{\beta x/2}$. Indeed, the weight function generates dissipative terms in the energy estimates with integration by parts for the situation that  the $\xi$-support of $f$ is constrained to the lower half space $\mathbb R^{3}_{-}$.
The $\xi$-support of $f$ can be tracked over time by using the a priori estimate on energy functionals. 
We make use of this dissipative effect in the proof of asymptotic stability. 
Similar approach is also used in the proof of instability, 
since the weight function $e^{\beta x/2}$ can generate a growing term if the $\xi$-support of $f$ has a nonempty intersection with the upper half space $\mathbb R^{3}_{+}$.

When we analyze the stability for the cold plasma case $\theta_{\infty} \ll 1$, another difficulty lies in the rightmost of \eqref{ree1}, which diverges as $\theta_{\infty} \to 0$. For this, we introduce the following hydrodynamic variables $(n,m)=(n,m)(t,x)$ to reform the corresponding terms in the energy estimate and generates good contribution:
\begin{gather}\label{nm0}
n=n[f]:=\int_{\mathbb R^{3}}  M_{\infty}^{1/2} f d\xi, \quad 
m=m[f]:=\int_{\mathbb R^{3}}  \xi_{1} M_{\infty}^{1/2} f d\xi.
\end{gather}
Here $n$ and $m$ are the perturbations of the density and flux.
An important thing to note is that another weight function $M_{\infty}^{-1/2}$ in \eqref{pert1} plays an essential role when we find good contribution mentioned above.

Concerning the time-local solvability of \eqref{ree0}, the following difficulties arise. There is a characteristic boundary condition at $(x,\xi_{1},\xi')=(0,0,\xi')$, where a loss of regularity of solutions may occur. 
Another issue comes from the $\xi_{1}$ factor on the rightmost of \eqref{ree1}, which causes problem when estimating the related terms. 
To avoid these issues, we construct the solution $f$ so that $\supp f$ is alway from the point $(x,\xi_{1},\xi')=(0,0,\xi')$, and bounded from above in the $\xi_{1}$-direction 
by assuming the same conditions for the initial support $\supp f_{0}$.
But we need some more techniques when solving the linearized problem.
For more details, see Appendix \ref{SB}.


In order to prove Theorem \ref{mainThm},
it suffices to show the following theorem, where
the constant $\mu_{\infty}=\mu_{\infty}(\rho_{\infty},u_{\infty},\theta_{\infty})$ and the function $\eta=\eta(\beta)$ are defined by
\begin{gather}\label{mu_eta}
\mu_{\infty}(\rho_{\infty},u_{\infty},\theta_{\infty}):=\left\| \frac{\pd_{\xi_{1}}(M_{\infty}\psi-M_{\infty})}{M_{\infty}^{1/2}} \chi(-\xi_{1}) \right\|_{L^{2}_{\xi}},
\quad
\eta(\beta):=\frac{\beta^{2}}{2-\beta^{2}}.
\end{gather}

\begin{thm}[Nonlinear stability]\label{stability}
Let $\ve \in (0,r/2)$ for $r>0$ being in \eqref{cutoff0}.
Suppose that $F_{\infty}=M_{\infty} \psi$ satisfies \eqref{netrual1} and \eqref{Bohm2} as well as
there exist constants 
$\beta \in (0,1)$, $\rho_{\infty}>0$, $u_{\infty} < 0$,  $\theta_{\infty}>0$, and $\mu_{\infty}>0$ with
\begin{gather}
\frac{|u_{\infty}|}{\theta_{\infty}}- 
\left( \frac{\rho_{\infty}^{1/2}}{\theta_{\infty}^{1/2}}(1+\eta(\beta))^{1/2} + \frac{\mu_{\infty}}{\beta}  \right)^{2} \frac{1+\eta(\beta)}{r - 2\ve}
-  \frac{\{1+\beta^{2}(1+\eta(\beta))^{2} \} \mu_{\infty}^{2}}{\beta^{2}{(r - 2\ve) }}>0.
\label{asp1}
\end{gather}
There exists a constant $\delta_{0}>0$
such that if 
\begin{subequations}
\begin{gather}
{\rm supp} f_{0} \subset \{ (x,\xi) \in \overline{{\mathbb R}_{+} \times{\mathbb R}^{3}} \, | \,  \xi_{1} \leq - r + \ve \},
\label{supp0}\\
\Phi_{b} + \|e^{\beta x/2} f_{0}\|_{H^{1}({\mathbb R}_{+}\times{\mathbb R}^{3})} \leq \delta_{0}, 
\end{gather}
\end{subequations}
the initial--boundary value problem \eqref{ree0} has
a unique solution $(f,\phi) \in {\cal X}_{\rm e}(\infty,\beta) \times {\cal Y}_{\rm e}(\infty,\beta)$. 
Furthermore, there holds that
\begin{gather*}
{\rm supp} f \subset \{ (t,x,\xi) \in \overline{{\mathbb R}_{+} \times {\mathbb R}_{+} \times{\mathbb R}^{3}} \, | \,  \xi_{1} \leq - r + 2\ve \},
\\
\| e^{\beta x/2} f (t)\|_{H^{1}({\mathbb R}_{+}\times{\mathbb R}^{3})} + \| e^{\beta x /2} \phi (t)\|_{H^{3}({\mathbb R}_{+})} 
\leq C e^{-\gamma t /2} \quad \text{for $t>0$},
\end{gather*}
where $\gamma$ and $C$ are some positive constants independent of $t$.
\end{thm}

\begin{rem} 
{\rm
We can find constants  $\beta \in (0,1)$, $\rho_{\infty}>0$, $u_{\infty} < 0$,  $\theta_{\infty}>0$, and $\mu_{\infty}>0$ 
with  \eqref{asp1} if {\it either} the condition (i) {\it or} condition (ii) in Theorem \ref{mainThm} holds.
This fact is shown in Appendix \ref{SA}. 
}
\end{rem}

We show this theorem by combining the time-local solvability stated in Proposition \ref{local1} and the a priori estimate stated in Proposition \ref{apriori1} below. For the notational convenience, we define $N_{\alpha,\gamma}(T)$ by
\begin{gather}\label{N0}
N_{\alpha,\gamma}(T) := \sup_{t \in [0,T]} 
e^{\gamma t/2} \| e^{\alpha x/2} f(t) \|_{H^{1}({\mathbb R}_{+}\times{\mathbb R}^{3})} 
\quad \text{for $\alpha \geq 0$, $\gamma \geq 0$}.
\end{gather}

\begin{pro}\label{local1}
Let $\alpha \in (0,1)$, $s \in (0,r] $,  and $R \geq 1$ for $r>0$ being in \eqref{cutoff0}. Suppose that $F_{\infty}=M_{\infty} \psi$ satisfies \eqref{netrual1} and \eqref{Bohm2}.
Then there exist constants $\delta_{*}=\delta_{*}(\alpha,s,R)>0$, $T_{*}=T_{*}(\alpha,s,R) \in (0,1]$, and $C_{*}=C_{*}(\alpha,s,R)\geq 1$
such that if the initial data $f_{0}$ satisfy $e^{\alpha x /2} f_{0} \in H^1_{0,\Gamma}$,
\begin{gather*}
{\rm supp} f_{0} \subset \{ (x,\xi) \in \overline{{\mathbb R}_{+} \times{\mathbb R}^{3}} \, | \, |x|^{2}+|\xi_{1}|^{2} \geq s^{2}, \ \xi_{1} \leq R \}, 
\quad
\Phi_{b} + \|e^{\alpha x/2} f_{0}\|_{H^{1}_{x,\xi}} \leq \delta_{*},
\end{gather*}
the initial--boundary value problem \eqref{ree0} has
a unique solution $(f,\phi) \in {\cal X}_{\rm e}(T_{*},\alpha) \times {\cal Y}_{\rm e}(T_{*},\alpha)$.
Furthermore, there holds that
\begin{gather}\label{locales1}
\sup_{t \in [0,T_{*}]} (\Vert e^{\alpha x/2}   f (t) \Vert_{H^{1}_{x,\xi}} 
+ \Vert e^{\alpha x/2}   \phi (t) \Vert_{H^{3}_{x}} ) \leq C_{*} (\Phi_{b} + \|e^{\alpha x/2} f_{0}\|_{H^{1}_{x,\xi}} ).
\end{gather}
\end{pro}

\begin{pro}\label{apriori1}
Let  $\ve \in (0,r/2)$ for $r>0$ being in \eqref{cutoff0}.
Suppose that $F_{\infty}=M_{\infty} \psi$ satisfies \eqref{netrual1} and \eqref{Bohm2}, 
and further \eqref{asp1} holds.
Assume that $(f,\phi) \in {\cal X}_{\rm e}(T,\beta) \times {\cal Y}_{\rm e}(T,\beta)$ 
is a solution of the problem \eqref{ree0}
with the initial data $f_{0}$ satisfying \eqref{supp0}
for some constant $T>0$.
Then there exist constants $\gamma>0$ and $\delta_{0}>0$  
such that if $\Phi_{b}+N_{\beta,\gamma}(T)  \leq \delta_{0}$, the following hold:
\begin{gather}
{\rm supp} f \subset \{ (t,x,\xi) \in \overline{[0,T] \times {\mathbb R}_{+} \times{\mathbb R}^{3}} \, | \,  \xi_{1} \leq - r + 2\ve \},
\label{aes0}\\
N_{\beta,\gamma}(T) \leq  \tilde{C} \Vert e^{\beta x/2}   f_{0}  \Vert_{H^{1}_{x,\xi}}, 
\label{aes1}\\
 \|e^{\beta x/2} \phi (t) \|_{H^{3}_{x}} \leq  \tilde{C}  \Vert e^{\alpha x/2}   f (t) \Vert_{H^{1}_{x,\xi}}  \quad \text{for $t \in [0,T]$},
\label{ellaes1}
\end{gather}
where $\tilde{C}>1$ is a constant independent of $\Phi_b$, $f_{0}$, $t$, and $T$.
\end{pro}

Theorem \ref{stability}  can be proved by using these propositions as follows.

\begin{proof}[Proof of Theorem \ref{stability}]
First we fix $\alpha=\beta$, $s=r-2\ve$, and $R=1$ to apply Proposition \ref{local1},
where $\beta$, $r$, and $\ve$ are the same constants being in \eqref{asp1} and Proposition \ref{apriori1}.
Let us take the boundary data $\Phi_{b}$ and initial data $f_{0}$ so small that 
the right hand sides of \eqref{locales1} and \eqref{aes1} are estimated 
by $\min\{\delta_{*},\delta_{0}\}/(4 e^{\gamma/2})$,
where $\delta_{*}$, $\delta_{0}$, and $\gamma$ are the same constants being in Propositions \ref{local1} and \ref{apriori1}.
Note that $\Phi_{b} + \|e^{\alpha x/2} f_{0}\|_{H^{1}_{x,\xi}} \leq \min\{\delta_{*},\delta_{0}\}/4$ holds, 
and also the initial data $f_{0}$ satisfies \eqref{supp0}.

Now applying Proposition \ref{local1} with the same constants $\alpha$, $s$, and $R$ as above,
we obtain the time-local solution $(f,\phi) \in {\cal X}_{\rm e}(T_{*},\beta) \times {\cal Y}_{\rm e}(T_{*},\beta)$, where $T_{*} \leq 1$, and see from \eqref{locales1} with the small data $\Phi_{b}$ and $f_{0}$ above that
$N_{\beta,\gamma}(T_{*}) \leq \min\{\delta_{*},\delta_{0}\}/4$ holds.
Moreover, using Proposition \ref{apriori1}, we see that the solution satisfies \eqref{aes0} with $T=T_{*}$.
Set
\begin{gather*}
T_{sup}:=\sup\left\{ \tau>0 \left|
\begin{array}{ll}
\text{the solution $(f,\phi)$ exists until $t=\tau$ and satisfies the following:}
\\
\quad \quad \text{$N_{\beta,\gamma}(\tau) <  \min\{\delta_{*},\delta_{0}\}/2$ and \eqref{aes0} with $T=\tau$ }
\end{array}
\right.\right\}.
\end{gather*}
Note that $T_{sup} \geq T_{*}>0$.

It suffices to show that $T_{sup}=\infty$ thanks to \eqref{ellaes1}.
We argue by contradiction. Suppose that $T_{sup}<\infty$.
By regarding $T_{sup}-T_{*}/2$ as an initial time and applying Proposition \ref{local1} 
with the same constants $\alpha$, $s$, and $R$ as above,
we see that the solution $(f,\phi)$ exists beyond $T_{sup}$.
Therefore, either $N_{\beta,\gamma}(T_{sup}) = \min\{\delta_{*},\delta_{0}\}/2$ or the negation of \eqref{aes0} with $T=T_{sup}$ holds.
However, the negation does not hold if  $N_{\beta,\gamma}(T_{sup}) \leq \min\{\delta_{*},\delta_{0}\}/2$, due to Proposition \ref{apriori1}.
Thus we have $N_{\beta,\gamma}(T_{sup})  = \min\{\delta_{*},\delta_{0}\}/2$.
Now we can apply Proposition \ref{apriori1}, and hence see from \eqref{aes1} with the small data $f_{0}$ above that
$N_{\beta,\gamma}(T_{sup}) \leq  \min\{\delta_{*},\delta_{0}\}/4$.
This yields a contradiction, i.e. $0<\min\{\delta_{*},\delta_{0}\}/2 = N_{\beta,\gamma}(T_{sup})  \leq \min\{\delta_{*},\delta_{0}\}/4$.
The proof is complete.
\end{proof}


On the other hand, in order to prove Theorem \ref{mainThm2},
it suffices to show the following theorem. 

\begin{thm}[Nonlinear instability]\label{instability}
Let $\beta \in (0,1)$.
Suppose that $F_{\infty}=M_{\infty} \psi$ satisfies \eqref{netrual1} and \eqref{Bohm2}
as well as $g_{0}$ satisfies \eqref{g0}.
Then there exists a constant $\ve>0$ such that for any sufficiently small $\delta>0$,
the initial--boundary value problem \eqref{ree0} with $(f_{0},\Phi_{b})=(\delta g_{0},\delta)$ has a solution $(f,\phi) \in {\cal X}_{\rm e}(T,\beta) \times {\cal Y}_{\rm e}(T,\beta)$ with $\|e^{\beta x/2}f(T)\|_{H^{1}({\mathbb R}_{+}\times{\mathbb R}^{3})} \geq \ve$ for some $T>0$. 
\end{thm}

To show this, we use the next proposition.

\begin{pro}\label{apriori2}
Let $\beta \in (0,1)$.
Suppose that $F_{\infty}=M_{\infty} \psi$ satisfies \eqref{netrual1} and \eqref{Bohm2} as well as $g_{0}$ satisfies \eqref{g0}.
Then there exist constants $\delta_{0}>0$ and $\ve_{0}>0$ such that if the solution $(f,\phi) \in {\cal X}_{\rm e}(T,\beta) \times {\cal Y}_{\rm e}(T,\beta)$ of the problem \eqref{ree0} with $(f_{0},\Phi_{b})=(\delta g_{0},\delta)$ for $\delta \in  (0,\delta_{0}]$ 
satisfies $N_{\beta,0}(T) \leq \ve_{0}$,
then the following hold:
\begin{gather}
{\rm supp} f \cap \{ (t,x,\xi) \in \overline{[0,T] \times {\mathbb R}_{+} \times{\mathbb R}^{3}} \, | \,  -r/2 < \xi_{1} < R_{1}/2  \} = \emptyset,
\label{aesin0} \\
{\rm supp} f \cap \{ (t,x,\xi) \in \overline{[0,T] \times {\mathbb R}_{+} \times{\mathbb R}^{3}} \, | \,  \xi_{1} > 2R_{2}  \} = \emptyset, 
\label{aesin2} \\
\Vert e^{\beta x/2}   f (t) \Vert_{H^{1}_{x,\xi}} 
\geq e^{\gamma t/2} \Vert e^{\beta x/2}   f_{0} \Vert_{L^{2}_{x,\xi}} >0 \quad \text{for $t \in [0,T]$,}
\label{aesin1}
\end{gather}
where $\gamma>0$ is a constant independent of $\Phi_{b}$, $f_{0}$, $t$, and $T$.
\end{pro}

We can prove Theorem \ref{instability} by using Propositions \ref{local1} and \ref{apriori2} as follows.

\begin{proof}[Proof of Theorem \ref{instability}]
First we fix $\alpha=\beta$, $s=\min\{r/2,R_{1}/2\}$, and $R=\max\{1,2R_{2}\}$ to apply Proposition \ref{local1},
where $r$, $\beta$,  $R_{1}$, and $R_{2}$ are the same constants being in \eqref{cutoff0}, Proposition \ref{apriori2}, and \eqref{g0}.
There exists a constant $\delta_{1} \in (0,\delta_{0}]$ such that if $\delta<\delta_{1}$, then
the right hand sides of \eqref{locales1} with $(f_{0},\Phi_{b})=(\delta g_{0}, \delta)$
is estimated by $\min\{\delta_{*},\ve_{0}\}/4$,
where $\delta_{*}$, $\delta_{0}$, and  $\ve_{0}$ are the same constants being in Propositions \ref{local1} and \ref{apriori2}.
Note that $\delta_{1} \leq \min\{\delta_{*}/8,\delta_{0}\}$ holds.

Now applying Proposition \ref{local1} with the same constants $\alpha$, $s$, and $R$ as above,
we have the time-local solution $(f,\phi) \in {\cal X}_{\rm e}(T_{*},\beta) \times {\cal Y}_{\rm e}(T_{*},\beta)$ with $(f_{0},\Phi_{b})=(\delta g_{0}, \delta)$ for $\delta \leq \delta_{1}/2$, 
and see from \eqref{locales1}  that $N_{\beta,0}(T_{*}) \leq \min\{\delta_{*},\ve_{0}\}/4$ holds. Moreover, using Proposition \ref{apriori2}, we see that the solution satisfies \eqref{aesin0} and \eqref{aesin2} with $T=T_{*}$.
Let 
\begin{gather*}
\ve=\frac{1}{2}\min\{\delta_{*},\ve_{0}\}, \quad
\delta \leq \frac{1}{2} \delta_{1}, 
\\
T:=\sup\left\{ \tau>0 \left|
\begin{array}{ll}
\text{the solution $(f,\phi)$ exists until $t=\tau$ and satisfies the following:}
\\
\qquad \qquad  \text{$N_{\beta,0}(\tau) <  \ve$ and \eqref{aesin0}--\eqref{aesin2} with $T=\tau$ }
\end{array}
\right.\right\}.
\end{gather*}
Note that $T \geq T_{*}>0$. In what follows, we show that $\ve$ and $T$ are the desired constants.

Let us show $T<\infty$ by contradiction. Suppose that $T=\infty$. 
We see that $N_{\beta,0}(S) \leq \ve$ holds for any $S>0$. 
This fact with Proposition \ref{apriori2} gives \eqref{aesin1} with $t=S$.
Using this, we can find $S_{*}>0$ so that
$N_{\beta,0}(S_{*}) \geq \Vert e^{\beta x/2}   f (S_{*}) \Vert_{H^{1}_{x,\xi}} \geq 2\ve$. This contradicts to $T=\infty$.
Thus $T<\infty$ holds. 

We complete the proof by showing $\|e^{\beta x/2}f(T)\|_{H^{1}_{x,\xi}} = \ve$.
It is seen that the solution $(f,\phi)$ exists beyond $T$
by regarding $T-T_{*}/2$ as an initial time and applying Proposition \ref{local1} 
with the same constants $\alpha$, $s$, and $R$ as above.
Therefore, either $N_{\beta,0}(T) =\|e^{\beta x/2}f(T)\|_{H^{1}_{x,\xi}}= \ve$ or the negation of \eqref{aesin0}--\eqref{aesin2} holds.
However, the negation does not hold if $N_{\beta,0}(T) \leq \ve (< \ve_{0})$, due to Propositions \ref{apriori2}.
Therefore, $\|e^{\beta x/2}f(T)\|_{H^{1}_{x,\xi}} = \ve$ holds. The proof is complete.
\end{proof}

Consequently, for the completion of the proofs of Theorems \ref{mainThm} and \ref{mainThm2},
it remains only to show the above propositions. 
Propositions \ref{local1}, \ref{apriori1}, and \ref{apriori2} are shown in Sections \ref{LS}, \ref{AE}, and \ref{AE2}, respectively.

\section{Preliminary}\label{S4}
\subsection{Hydrodynamic variables}  \label{SS:HV}
This subsection is devoted to the study of the hydrodynamic variables $(n,m)$ defined in \eqref{nm0}, i.e.
\begin{gather*}
n=n[f]:=\int_{\mathbb R^{3}}  M_{\infty}^{1/2} f d\xi, \quad 
m=m[f]:=\int_{\mathbb R^{3}}  \xi_{1} M_{\infty}^{1/2} f d\xi.
\end{gather*}
Subtracting \eqref{seq1} from \eqref{boltz1} and integrating the resulting equality with respect to $\xi$, 
we have the equation of continuity:
\begin{gather}\label{mass1}
\pd_{t} n  + \pd_{x} m = 0, \quad t>0, \ x>0.
\end{gather}
We also show the following estimates on $n$ and $m$:
\begin{lem}\label{nmlem}
Let $\alpha \geq 0$ and $k=0,1$.
For any function $f$ satisfying $e^{\alpha x/2} f \in H^{k}_{x,\xi}$, there hold that
\begin{gather}
\| e^{\alpha x/2} (\pd_{x}^{k} m[f] - u_{\infty}\pd_{x}^{k} n[f]) \|_{L^{2}_{x}}^{2} \leq \rho_{\infty}\theta_{\infty} \| e^{\alpha x/2} \pd_{x}^{k} f\|_{L^{2}_{x,\xi}}^{2},
\label{nm1}\\
\| e^{\alpha x/2} \pd_{x}^{k} n[f]\|_{L_{x}^{2}}^{2} \leq \rho_{\infty}  \| e^{\alpha x/2} \pd_{x}^{k} f \|_{L^{2}_{x,\xi}}^{2}.
\label{nm3}
\end{gather}
\end{lem}
\begin{proof}
We prove only \eqref{nm1} with $k=0$, since the others can be shown in the same way.
Applying H\"older's inequality, we observe that
\begin{align}
\| e^{\alpha x/2} (m[f] - u_{\infty}n[f]) \|_{L^{2}_{x}}^{2}
& = \int_{{\mathbb R}_{+}} e^{\alpha x} \left( \int_{{\mathbb R}^{3}} (\xi_{1}-u_{\infty}) M_{\infty}^{1/2} f d\xi\right)^{2}dx
\notag \\
& \leq \int_{{\mathbb R}_{+}} e^{\alpha x} \int_{{\mathbb R}^{3}} (\xi_{1}-u_{\infty})^{2} M_{\infty} d\xi  \int_{{\mathbb R}^{3}} |f|^{2} d\xi dx
\notag \\
& = \int_{{\mathbb R}^{3}} |\xi_{1}-u_{\infty}|^{2} M_{\infty} d\xi \int_{{\mathbb R}_{+}} \int_{{\mathbb R}^{3}} e^{\alpha x} |f|^{2} dx  d\xi.
\label{T000}
\end{align}
It is straightforward to see that
\begin{gather}\label{T00}
\int_{{\mathbb R}^{3}} | \xi_{1}-u_{\infty}|^{2}  M_{\infty} d\xi = \rho_{\infty} \theta_{\infty}.
\end{gather}
Substituting \eqref{T00} into \eqref{T000}, we complete the proof.
\end{proof}

\subsection{Estimates on $F^{s}$}  \label{SS:Fs}
In this subsection, we discuss some estimates on $F^{s}$.
First we observe that
\begin{gather}
\begin{aligned}
\frac{\pd_{\xi_{1}}F^{s}}{M_{\infty}^{1/2}} 
&= -  \frac{\xi_{1}-u_{\infty}}{\theta_{\infty}} M_{\infty}^{1/2} + \frac{\pd_{\xi_{1}}(M_{\infty}\psi-M_{\infty})}{M_{\infty}^{1/2}} 
+ \frac{\pd_{\xi_{1}}(F^{s}-M_{\infty}\psi)}{M_{\infty}^{1/2}}. \quad
\end{aligned}
\label{M0}
\end{gather}
We study the estimates for the terms $\frac{\pd_{\xi_{1}}F^{s}}{M_{\infty}^{1/2}} $ and $\frac{\pd_{\xi_{1}}(F^{s}-M_{\infty}\psi)}{M_{\infty}^{1/2}}$ in \eqref{M0}.

\begin{lem}
There exists a constant $\delta_{0}>0$ such that if $\Phi_{b} \leq \delta_{0}$, the following hold:
\begin{align}
\left\|\frac{{\pd_{\xi_{1}}(F^{s}-M_{\infty}\psi)} }{M_{\infty}^{1/2}} (x,\cdot) \right\|_{L^{2}_{\xi}}
&\leq C \Phi^{s}(x),
\label{M2*}\\
\left\| \pd_{x} \left \{\frac{\pd_{\xi_{1}}(F^{s}-M_{\infty}\psi)}{M_{\infty}^{1/2}}  \right\} (x,\cdot) \right\|_{L^{2}_{\xi}}
&\leq C |\pd_{x}\Phi^{s}(x)|,
\label{M3*}\\
\sup_{x \in \mathbb R_{+}}\left\|\nabla_{\xi}\left(\frac{\pd_{\xi_{1}}F^{s}}{M_{\infty}^{1/2}}\right)(x,\cdot) \right\|_{L^{2}_{\xi}}
&\leq C,
\label{Fs1*}
\end{align}
where $C>0$ is a constant independent of $x$ and $\Phi_{b}$. 
\end{lem}
\begin{proof}
Set $\zeta_{1}=\zeta_{1}(x,\xi_{1}):=-\sqrt{\xi_{1}^{2}-2\Phi^{s}(x)}$.
First we show \eqref{M2*}. To this end, it suffices to show that
\begin{gather}
\frac{\left|{\pd_{\xi_{1}}(F^{s}-M_{\infty}\psi)(x,\xi)} \right|}{M_{\infty}^{1/2}(\xi)}
\leq \left\{
\begin{array}{ll}
0 & \text{for $\xi_{1}>-r$},
\\
C \Phi^{s}(x)(1+|\xi_{1}|^{2}) M_{\infty}^{1/2}(\xi) & \text{for $\xi_{1}\leq -r$}.
\end{array}
\right.
\label{M2}
\end{gather}
We see that \eqref{M2} holds for $\xi_{1}>-r$, 
since the left hand side of \eqref{M2} is zero owing to \eqref{cutoff0} and \eqref{r2}.
Let us consider another case $\xi_{1} \leq -r$ with the assumption $\Phi_{b} \ll 1$.
We note that $\xi_{1}^{2} - 2\Phi^{s}(x)>0$ holds.
Using \eqref{fform}, \eqref{decay0}, and $\Phi_{b} \ll 1$, we observe that
\begin{align}
&\left|{\pd_{\xi_{1}}(F^{s}-M_{\infty}\psi)}(x,\xi) \right| 
 = \left| \pd_{\xi_{1}}(M_{\infty}\psi)(\zeta_1,\xi') \frac{|\xi_{1}|}{\sqrt{\xi_{1}^{2}-2\Phi^{s}(x)}}
 - \pd_{\xi_{1}}(M_{\infty}\psi)(\xi) \right|
\notag\\
&\leq \left| \pd_{\xi_{1}}(M_{\infty}\psi)(\zeta_1,\xi') \right| \left|\frac{|\xi_{1}|}{\sqrt{\xi_{1}^{2}-2\Phi^{s}(x)}}-1 \right|
+ \left| \pd_{\xi_{1}}(M_{\infty}\psi)(\zeta_1,\xi') - \pd_{\xi_{1}}(M_{\infty}\psi) (\xi) \right|
\notag\\
&\leq C | \pd_{\xi_{1}}(M_{\infty}\psi)(\zeta_1,\xi') | \Phi^{s}(x)
+ | \pd_{\xi_{1}}(M_{\infty}\psi)(\zeta_1,\xi') - \pd_{\xi_{1}}(M_{\infty}\psi) (\xi) |
\notag\\
&\leq C\{| \pd_{\xi_{1}}M_{\infty} (\zeta_1,\xi') | + M_{\infty}(\zeta_1,\xi') \}\Phi^{s}(x) 
\notag\\
& \quad + | \pd_{\xi_{1}}M_{\infty}(\zeta_1,\xi')||\psi(\zeta_1,\xi') - \psi (\xi)|
+ | \pd_{\xi_{1}}M_{\infty}(\zeta_1,\xi')- \pd_{\xi_{1}}M_{\infty}(\xi)||\psi (\xi) |
\notag\\
& \quad +  M_{\infty}(\zeta_1,\xi')|\pd_{\xi_{1}}\psi(\zeta_1,\xi') -\pd_{\xi_{1}}\psi(\xi)| 
+ | M_{\infty}(\zeta_1,\xi') - M_{\infty} (\xi) | |\pd_{\xi_{1}}\psi (\xi) |.
\label{Minf0}
\end{align}

We estimate $M_{\infty}(\zeta_1,\xi')$, $|\pd_{\xi_{1}}M_{\infty}(\zeta_1,\xi')|$,
$| M_{\infty}(\zeta_1,\xi') - M_{\infty} (\xi) |$, and $| \pd_{\xi_{1}}M_{\infty}(\zeta_1,\xi')- \pd_{\xi_{1}}M_{\infty}(\xi)|$
in the right hand side of \eqref{Minf0}.
There holds that
\begin{align}
M_{\infty}(\zeta_1,\xi')
& \leq C \exp\left(-\frac{|\zeta_{1}-u_{\infty}|^{2}+|\xi'|^{2}}{2\theta_{\infty}} \right)
\notag\\
&=  C \exp\left(-\frac{|\xi_{1}-u_{\infty}|^{2}+|\xi'|^{2}}{2\theta_{\infty}} \right)
\exp\left(\frac{2(\zeta_{1}-\xi_{1})u_{\infty}}{2\theta_{\infty}} \right)
\exp\left(\frac{2\Phi^{s}(x)}{2\theta_{\infty}} \right)
\notag\\
&\leq C M_{\infty}(\xi),
\label{Minf1}
\end{align}
where we have used $(\zeta_{1}-\xi_{1})u_{\infty} \leq 0$ in deriving the last inequality.
It also holds that
\begin{align}
|\pd_{\xi_{1}}M_{\infty}(\zeta_1,\xi')| 
\leq C |\zeta_{1}-u_{\infty}| M_{\infty}(\zeta_1,\xi')
\leq C (1 + |\xi_{1}|) M_{\infty}(\xi).
\label{Minf2}
\end{align}
Furthermore, we observe using the mean value theorem that
\begin{align}
| M_{\infty}(\zeta_1,\xi') - M_{\infty} (\xi) | 
& \leq   \int_{0}^{1} \left| \pd_{\xi_{1}} M_{\infty} (\theta \zeta_{1}+(1-\theta)\xi_{1},\xi') \right| d\theta  |\zeta_1 - \xi_{1}| 
\notag\\
& \leq  C\int_{0}^{1} (1 + |\xi_{1}|) M_{\infty} (\theta \zeta_{1}+(1-\theta)\xi_{1},\xi') d\theta  |\zeta_1 - \xi_{1}| 
\notag\\
&  \leq C  \int_{0}^{1} (1 + |\xi_{1}|) M_{\infty}(\xi) d\theta  |\zeta_1 - \xi_{1}|
\notag\\
& \leq C (1 + |\xi_{1}|) M_{\infty}(\xi)\Phi^{s}(x),
\label{Minf4}
\end{align}
where we have used the equality $|\theta \zeta_{1}+(1-\theta)\xi_{1}-u_{\infty}|^{2}=-|\xi_{1}-u_{\infty}|^{2}+2\theta(\zeta_{1}-\xi_{1})u_{\infty}+2\theta(1-\theta)(\zeta_{1}-\xi_{1})|\xi_{1}|+2\theta^{2}\Phi^{s}$ as well as the inequalities $(\zeta_{1}-\xi_{1})u_{\infty} \leq 0$ and $(\zeta_{1}-\xi_{1})|\xi_{1}| \leq 2\Phi^{s}$ in deriving the third inequality;
we have used $\xi_{1} \leq -r$ to estimate $|\zeta_1 - \xi_{1}|$ in deriving the last inequality.
Similarly, it follows that
\begin{align}
| \pd_{\xi_{1}}M_{\infty}(\zeta_1,\xi')- \pd_{\xi_{1}}M_{\infty}(\xi)|
& \leq   \int_{0}^{1} \left| \pd_{\xi_{1}\xi_{1}} M_{\infty}(\theta \zeta_{1}+(1-\theta)\xi_{1},\xi') \right| d\theta  |\zeta_1 - \xi_{1}| 
\notag \\
& \leq C  \int_{0}^{1} (1 +  |\xi_{1}|^{2}) M_{\infty}(\xi) d\theta  |\zeta_1 - \xi_{1}| 
\notag\\
& \leq C (1 + |\xi_{1}|^{2}) M_{\infty}(\xi)\Phi^{s}(x).
\label{Minf5}
\end{align}
On the other hand, it is straightforward to see from \eqref{cutoff0} and $\xi_{1} \leq -r$ that
\begin{gather}
|\psi(\zeta_1,\xi') - \psi (\xi)| + |\pd_{\xi_{1}}\psi(\zeta_1,\xi') -\pd_{\xi_{1}}\psi(\xi)| \leq  C \Phi^{s}(x).
\label{Minf6}
\end{gather}
Plugging \eqref{Minf1}--\eqref{Minf6} into \eqref{Minf0} and dividing the resultant inequality by $M_{\infty}^{1/2}(\xi)$, we conclude that \eqref{M2} holds for $\xi_{1} \leq -r$. 

Next we derive \eqref{M3*}. It is sufficient to show that
\begin{gather}
\left| \pd_{x} \left\{\! \frac{\pd_{\xi_{1}}(F^{s}-M_{\infty}\psi)}{M_{\infty}^{1/2}} \! \right\} \!\! (x,\xi)\right| 
\leq  \left\{\begin{array}{ll}
\!\!0 & \!\text{for $\xi_{1}>-r$},
\\
\!\!C |\pd_{x}\Phi^{s}(x)|(1+|\xi_{1}|^{2}) M_{\infty}^{1/2}(\xi) & \!\text{for $\xi_{1} \leq -r$}.
\end{array}
\right.
\label{M3}
\end{gather}
We see that \eqref{M3} holds for $\xi_{1}>-r$ in much the same way as \eqref{M2}.
Let us consider another case $\xi_{1} \leq -r$ with $\Phi_{b} \ll 1$.
In a similar way to \eqref{Minf0}, we observe that
\begin{align}
& \left| \pd_{x} {\pd_{\xi_{1}}(F^{s}-M_{\infty}\psi)}(\xi_{1},\xi')\right| 
\notag\\
& = \left| \pd_{\xi_{1}\xi_{1}}(M_{\infty}\psi)(\zeta_1,\xi') \frac{\xi_{1}\pd_{x}\Phi^{s}(x)}{\xi_{1}^{2}-2\Phi^{s}(x)} 
+ \pd_{\xi_{1}}(M_{\infty}\psi)(\zeta_1,\xi') \frac{\xi_{1}\pd_{x}\Phi^{s}(x)}{(\xi_{1}^{2}-2\Phi^{s}(x))^{3/2}} \right|
\notag\\
& \leq C ( | \pd_{\xi_{1}\xi_{1}}(M_{\infty}\psi)(\zeta_1,\xi') | + | \pd_{\xi_{1}}(M_{\infty}\psi)(\zeta_1,\xi') | ) |\pd_{x}\Phi^{s}(x)|.
\label{Minf7}
\end{align}
Furthermore, similarly as \eqref{Minf2}, it follows that 
\begin{align}
|\pd_{\xi_{1}\xi_{1}}M_{\infty}(\zeta_1,\xi')| 
 \leq C (1+|\zeta_{1}-u_{\infty}|^{2}) M_{\infty}(\zeta_1,\xi')
\leq C (1 + |\xi_{1}|^{2}) M_{\infty}(\xi).
\label{Minf3}
\end{align}
Estimating the rightmost of \eqref{Minf7} by \eqref{Minf1}, \eqref{Minf2}, and \eqref{Minf3}, 
we arrive at \eqref{M3} for $\xi_{1} \leq -r$. 

Finally, let us derive \eqref{Fs1*}. It is sufficient to show that
\begin{gather}
\left|\nabla_{\xi}\left(\frac{\pd_{\xi_{1}}F^{s}}{M_{\infty}^{1/2}}\right)(x,\xi)\right| 
\leq  \left\{\begin{array}{ll}
0 & \text{for $\xi_{1}>-r$},
\\
C(1+|\xi|^{2}) M_{\infty}^{1/2}(\xi) & \text{for $\xi_{1} \leq -r$}.
\end{array}
\right.
\label{Fs1}
\end{gather}
It is clear that \eqref{Fs1} holds for $\xi_{1}>-r$, and the following hold for $\xi_{1} \leq -r$:
\begin{gather*}
\pd_{\xi_{1}}\left(\frac{\pd_{\xi_{1}}F^{s}}{M_{\infty}^{1/2}}\right)
= \frac{\pd_{\xi_{1}\xi_{1}}F^{s}}{M_{\infty}^{1/2}} +\frac{\pd_{\xi_{1}}F^{s}}{M_{\infty}^{1/2}} \frac{(\xi_{1}-u_{\infty})}{2\theta_{\infty}},
\quad 
\pd_{\xi_{j}}\left(\frac{\pd_{\xi_{1}}F^{s}}{M_{\infty}^{1/2}}\right)
= \frac{\pd_{\xi_{j}\xi_{1}}F^{s}}{M_{\infty}^{1/2}} +\frac{\pd_{\xi_{1}}F^{s}}{M_{\infty}^{1/2}} \frac{\xi_{j}}{2\theta_{\infty}},
\end{gather*}
where $j=2,3$.
By using \eqref{fform}, \eqref{decay0}, \eqref{cutoff0}, \eqref{cutoff1}, \eqref{Minf1}, \eqref{Minf2}, \eqref{Minf3}, and $\Phi_{b} \ll 1$ similarly as above, we observe that for $\xi_{1} \leq -r$,
\begin{align*}
\left| \pd_{\xi_{1}}F^{s} \right| 
&= \left| \pd_{\xi_{1}}(M_{\infty}\psi)(\zeta_1,\xi') \frac{\xi_{1}}{\sqrt{\xi_{1}^{2}-2\Phi^{s}}} \right|
\leq C (1 + |\xi_{1}|) M_{\infty}(\xi),
\\
\left| \pd_{\xi_{1}\xi_{1}}F^{s} \right| 
&=\left|  \pd_{\xi_{1}\xi_{1}}(M_{\infty}\psi)(\zeta_1,\xi') \frac{\xi_{1}^{2}}{{\xi_{1}^{2}-2\Phi^{s}}} 
 - \pd_{\xi_{1}}(M_{\infty}\psi)(\zeta_1,\xi')\frac{\xi_{1}^{2}}{(\xi_{1}^{2}-2\Phi^{s})^{3/2}} \right|
 \\
& \leq C (1 + |\xi_{1}|^{2}) M_{\infty}(\xi),
\\
\left| \pd_{\xi_{j}\xi_{1}}F^{s} \right| 
&=\left|  \pd_{\xi_{j}\xi_{1}}(M_{\infty}\psi)(\zeta_1,\xi') \frac{\xi_{1}}{\sqrt{\xi_{1}^{2}-2\Phi^{s}}} \right|
\leq C (1 + |\xi|^{2}) M_{\infty}(\xi).
\end{align*}
These implies \eqref{Fs1} for $\xi_{1} \leq -r$. The proof is complete.
\end{proof}

\subsection{Elliptic estimates}\label{ES}
This subsection provides estimates of $\phi=\phi(x)$ which solves the elliptic equation \er{ree2} with \eqref{rebc2} and \eqref{rebc4}, i.e.
\begin{gather*}
\partial_{xx} \phi =  n - (n_{e}(\Phi^{s}+\phi)-n_{e}(\Phi^{s})), \quad x>0,
\\
\phi(0)=0, \quad \lim_{x\to\infty}\phi(x)=0,
\end{gather*}
where $n=n[f]=\int_{\mathbb R^{3}}  M_{\infty}^{1/2} f d\xi$ is supposed to be a given function.

Let us first show the lower and upper bounds of $\phi$.
From the above equation and \eqref{seq2}, it is seen that $\Phi=\Phi^{s}+\phi$ satisfies
\begin{subequations}
\begin{gather}
\partial_{xx} \Phi = \int_{\mathbb R^{3}}  F^{s} d\xi + n - n_{e}(\Phi), \quad x>0,
\label{elleq1}\\
\Phi(0)=\Phi_{b}>0, \quad \lim_{x\to\infty}\Phi(x)=0.
\label{ellbc1}
\end{gather}
\end{subequations}


\begin{lem}\lb{ell0}
Let $n \in H_{x}^{1}$.
Suppose that $\phi \in H^{3}_{x}$ is a solution of the elliptic equation \er{ree2} with \eqref{rebc2} and \eqref{rebc4}.
Then there exists a constant $\delta_{0}>0$ such that if $\Phi_{b}+\|n\|_{H^{1}_{x}}  \leq \delta_{0}$, the following hold:
\begin{gather}
\sup_{x\in\mathbb R_{+}}(\phi+\Phi^{s})(x) \leq M_1:= \max\left\{\Phi_{b}, \ n_{e}^{-1}\!\! \left( \inf_{x\in \mathbb R_{+}} \!\!\left(\int_{{\mathbb R}^{3}} F^{s}(x,\xi) d\xi+n (x)\right) \!\!\right)\!\! \right\},
\lb{bounds1}\\
\inf_{x\in\mathbb R_{+}}(\phi+\Phi^{s})(x) \geq  -M_2:= -\max\left\{\Phi_{b}, \  -n_{e}^{-1}\!\! \left( \sup_{x\in \mathbb R_{+}} \!\!\left(\int_{{\mathbb R}^{3}} F^{s}(x,\xi) d\xi+n (x)\right) \!\!
\right)\!\! \right\}.
\lb{bounds2}
\end{gather}
\end{lem}
\begin{proof}
First it is seen from \eqref{netrual1} and \eqref{cutoff1} that 
\begin{align*}
\left| \int_{{\mathbb R}^{3}} F^{s}(x,\xi) d\xi+n (x) -1 \right| 
& = \left| \int_{{\mathbb R}^{3}} F^{s}(x,\xi) d\xi+n (x) - \int_{{\mathbb R}^{3}} M_{\infty}\psi(\xi) d\xi \right|
\\
&\leq C(\Phi_{b} + \|n\|_{H^{1}_{x}}),
\end{align*}
where we have used \eqref{fform} and \eqref{decay0} in deriving the inequality. 
This together with the assumption \eqref{n_e} ensures that the constants $M_{1}$ and $M_{2}$ are well-defined and small if $\Phi_{b}+\|n\|_{H^{1}_{x}} \ll 1$.
To show \er{bounds1}, we set $\Psi:=\Phi-M_{1}$.
It is straightforward to see from \er{elleq1} that
\begin{equation*}
\pd_{xx} \Psi
= \int_{{\mathbb R}^{3}} F^{s} d\xi + n - n_{e}(\Psi+M_{1}).
\end{equation*}
Multiply this by $\Psi^+:= \max\{\Psi, 0\}$, integrate it over ${\mathbb R}_{+}$, and use the boundary condition \eqref{ellbc1}.
Then using $\Psi^+ \geq 0$ and the assumption \eqref{n_e},
we have
\begin{align*}
\int_{{\mathbb R}_{+}} |\pd_{x}\Psi^+|^2 dx
&=\int_{{\mathbb R}_{+}} \left( n_{e}(\Psi^{+}+M_{1}) - \int_{{\mathbb R}^{3}} F^{s}(x,\xi) d\xi - n  \right)\Psi^+ dx
\\
&\leq \int_{{\mathbb R}_{+}}\left\{ n_{e}(M_{1}) - \inf_{x\in \mathbb R_{+}} \!\!\left(\int_{{\mathbb R}^{3}} F^{s}(x,\xi) d\xi +n (x)\right) \right\} \Psi^+ dx \leq 0,
\end{align*}
which gives \er{bounds1} with the aid of $\Psi^{+}(0)=0$. Similarly, one can have the lower bound \er{bounds2}.
The proof is complete.
\end{proof}

Next we show some estimates of $\phi$ in the Sobolev space $H^{k}_{x}$.
We rewrite the elliptic equation \eqref{ree2} as follows:
\begin{gather}\label{pe3}
\partial_{xx} \phi - \phi -n
= {\cal N}_{1}:=
-\left(1 + \int_{0}^{1} n_{e}'(\Phi^{s}+\theta \phi) d\theta\right)\phi, \quad \ x>0.
\end{gather}

\begin{lem}\lb{ell1} 
Let $\beta \in (0,1)$ and $e^{\beta x /2} n \in H_{x}^{1}$.
Suppose that $\phi \in H^{3}_{x}$ is a solution of the elliptic equation \er{ree2} with \eqref{rebc2} and \eqref{rebc4},
and also satisfies $e^{\beta x /2} \phi \in H^{3}_{x}$.
Then there exists a constant $\delta_{0}>0$ independent of $\beta$
such that if $\Phi_{b}+\|n\|_{H^{1}_{x}}  \leq \delta_{0}$, the following hold:
\begin{gather}
\|\phi\|_{H^{3}_{x}} \leq C \|n\|_{H^{1}_{x}}, 
\lb{ellineq0}\\
 \|e^{\beta x/2}\partial_{x}^{k} \phi \|_{H^{2}_{x}}^{2} \leq  C \sum_{i=0}^{k} \|e^{\beta x/2} \partial_{x}^{i} n\|_{L^{2}_{x}}^2 \quad \text{for $k=0,1$,}
\lb{ellineq3}\\
\begin{aligned}
&\|e^{\beta x/2} \phi \|_{L^{2}_{x}}^2 +2(1+\eta(\beta))\|e^{\beta x/2} \pd_{x}\phi \|_{L^{2}_{x}}^2
\\
&\leq (1+\eta(\beta))^{2}\|e^{\beta x/2} n \|_{L^{2}_{x}}^2
+C(\Phi_{b} + \|n\|_{H^{1}_{x}})\|e^{\beta x/2} n \|_{L^{2}_{x}}^2,
\end{aligned}
\lb{ellineq1}\\
\|e^{\beta x/2} \pd_{xx} \phi \|_{L^{2}_{x}}^2
\leq  \{ 1+\beta^{2}(1+\eta(\beta))^{2}\} \| e^{\beta x/2} n \|_{L^{2}_{x}}^2
+C(\Phi_{b} + \|n\|_{H^{1}_{x}})\|e^{\beta x/2} n \|_{L^{2}_{x}}^2,
\lb{ellineq2}
\end{gather}
where $C>0$ is a constant independent of $\beta$ and $\Phi_b$, 
and $\eta(\beta)$ is defined in \eqref{mu_eta}.
\end{lem}
\begin{proof}
We first show \er{ellineq0}. 
Multiply \er{ree2} by $\phi$, integrate it by parts over ${\mathbb R}_{+}$,
and use \er{rebc2} to get
\begin{align*}
\int_{{\mathbb R}_{+}} |\pd_{x} \phi |^2 dx
-\int_{{\mathbb R}_{+}} \left(n_{e}(\Phi^{s}+\phi) - n_{e}(\Phi^{s}) \right)\phi dx
=\int_{{\mathbb R}_{+}} n \phi dx
\leq \mu \|\phi\|^2
+C\mu^{-1}\|n\|_{L^{2}_{x}}^{2},
\end{align*}
where $\mu$ is a positive constant to be determined later.
Let us find a good contribution of the second term on the left hand side.
By \er{n_e}, \er{decay0}, \er{bounds1}, \er{bounds2}, and the mean value theorem,
the second term on the left hand side is estimated from below as follows:
\[
-\int_{{\mathbb R}_{+}} \left(n_{e}(\Phi^{s}+\phi) - n_{e}(\Phi^{s}) \right)\phi dx \geq c \|\phi\|^2.
\]
These two inequalities with sufficiently small $\mu>0$ lead to
$\|\phi\|_{H^{1}_{x}} \leq C\|n\|_{L^{2}_{x}}$.
From this and \eqref{ree2}, we conclude \eqref{ellineq0}.

Next let us show \er{ellineq3}. 
By \eqref{n_e}, \er{decay0}, and \er{ellineq0}, one can estimate ${\cal N}_1$ in \eqref{pe3} as
\begin{gather}
\|e^{\beta x/2}{\cal N}_1\|_{L^{2}_{x}} \leq C(\Phi_{b} + \|n\|_{H^{1}_{x}})\|e^{\beta x/2}\phi\|_{L^{2}_{x}}.
\lb{N1}
\end{gather}
Multiply \er{pe3} by $e^{\beta x} \phi$, integrate it by parts over ${\mathbb R}_{+}$, 
and use \er{N1} and Schwarz's inequality to get
\begin{align*}
&\int_{{\mathbb R}_{+}} e^{\beta x} |\pd_{x} \phi|^2
+\left(1-\frac{\beta^2}{2}\right)e^{\beta x}|\phi|^2 dx
\\
&=-\int_{{\mathbb R}_{+}} e^{\beta x} \left(n+{\cal N} _1\right)\phi dx
\\
&\leq \frac{1}{4} \|e^{\beta x/2}\phi\|_{L^{2}_{x}}^2 +   \|e^{\beta x/2} n\|_{L^{2}_{x}}^2
+C(\Phi_{b} + \|n\|_{H^{1}_{x}}) \|e^{\beta x/2}\phi\|_{L^{2}_{x}}^2.
\end{align*}
Owing to $\beta < 1$, letting $\Phi_{b} + \|n\|_{H^{1}_{x}}$ be small enough,
we have $ \|e^{\beta x/2}\phi \|_{H^{1}_{x}}^{2} \leq  C  \|e^{\beta x/2} n\|_{L^{2}_{x}}^2$.
From this and \er{pe3}, we arrive at \eqref{ellineq3} for both the cases $k=0,1$.

Once again, multiply \er{ree2} by $e^{\beta x} \phi$,
integrate it by parts over ${\mathbb R}_{+}$, 
and estimate the result in a different way as above by using \er{ellineq3}.
Then we have
\begin{align*}
&\int_{{\mathbb R}_{+}} e^{\beta x} |\pd_{x} \phi|^2
+\left(1-\frac{\beta^2}{2}\right)e^{\beta x}|\phi|^2 dx
\\
&=-\int_{{\mathbb R}_{+}} e^{\beta x} \left(n+{\cal N} _1\right)\phi dx
\\
&\leq \frac{1}{2}\left(1-\frac{\beta^2}{2}\right) \|e^{\beta x/2}\phi\|_{L^{2}_{x}}^2 
+ \frac{1}{2}\left(1-\frac{\beta^2}{2}\right)^{\!\!-1}  \!\!\|e^{\beta x/2} n\|_{L^{2}_{x}}^2
+C(\Phi_{b} + \|n\|_{H^{1}_{x}}) \|e^{\beta x/2}\phi\|_{L^{2}_{x}}^2.
\end{align*}
This together with \eqref{ellineq3} with $k=0$ immediately gives \er{ellineq1}.

Let us show \eqref{ellineq2}.
Multiply \er{ree2} by $e^{\beta x} \pd_{xx} \phi$, integrate it by parts over ${\mathbb R}_{+}$,
and use Schwarz's inequality to obtain
\begin{align*}
{}&
\int_{{\mathbb R}_{+}} e^{\beta x} |\pd_{xx} \phi|^2+e^{\beta x}|\pd_{x}\phi|^2 dx
\\
&=
\int_{{\mathbb R}_{+}} 
\frac{\beta^2}{2} e^{\beta x}|\phi|^2 + e^{\beta x} n \pd_{xx}\phi dx +e^{\beta x} {\cal N}_{1} \pd_{xx}\phi dx
\\
&\leq \frac{\beta^{2}}{2} \|e^{\beta x/2} \phi \|_{L^{2}_{x}}^2
+\frac{1}{2} \|e^{\beta x/2} \pd_{xx}\phi \|_{L^{2}_{x}}^2
+ \frac{1}{2} \|e^{\beta x/2} n \|_{L^{2}_{x}}^2 
+\|e^{\beta x/2} {\cal N}_{1} \|_{L^{2}_{x}}\|e^{\beta x/2} \partial_{xx} \phi \|_{L^{2}_{x}}
\\
&\leq \frac{1}{2}\beta^{2}(1+\eta(\beta))^{2}\|e^{\beta x/2} n \|_{L^{2}_{x}}^2
+\frac{1}{2} \|e^{\beta x/2} \pd_{xx}\phi \|_{L^{2}_{x}}^2
+ \frac{1}{2} \|e^{\beta x/2} n \|_{L^{2}_{x}}^2 
\\
&\quad +C(\Phi_{b} + \|n\|_{H^{1}_{x}})\|e^{\beta x/2} n \|_{L^{2}_{x}}^2,
\end{align*}
where we have also used  \eqref{ellineq3} with $k=0$, \eqref{ellineq1}, and \eqref{N1} in deriving the last inequality.
This immediately gives \er{ellineq2}.
The proof is complete.
\end{proof}

\section{Time-local Solvability}\label{LS}
In this section, we prove Proposition \ref{local1} on the time-local solvability of the problem \eqref{ree0}.
We start from studying the following linearized and semi-linearized problems:
\begin{subequations}\label{l1}
\begin{gather}
{\cal L} f := \pd_{t} f + \xi_{1}\partial_{x} f + (\partial_{x} \Phi^{s} + \partial_{x}\hat{\phi}) \partial_{\xi_{1}}f 
-  \frac{\xi_{1}-u_{\infty}}{2\theta_{\infty}}  (\partial_{x} \Phi^{s} + \partial_{x}\hat{\phi})   f = \hat{F},
\label{leq1}\\
f(0,x,\xi) = f_{0}(x,\xi), \quad
f(t,0,\xi) = 0, \ \xi_{1}>0, \quad
\lim_{x \to\infty} f(t,x,\xi) = 0
\label{libc1}
\end{gather}
\end{subequations}
and
\begin{subequations}\label{l2}
\begin{gather}
\partial_{xx} \phi = \hat{n} - (n_{e}(\Phi^{s}+\phi)-n_{e}(\Phi^{s})), 
\label{leq2} \\
\phi(t,0)=0, \quad
\lim_{x \to\infty} \phi(t,x) = 0.
\label{libc2}
\end{gather}
\end{subequations}
The solvability of the problems \eqref{l1} and \eqref{l2} are summarized 
in Lemmas \ref{sovl1} and \ref{sovl2} below, respectively.
\begin{lem}\label{sovl1}
Let $\alpha>0$, $s>0$, and $R \geq 0$. 
Suppose that 
\begin{subequations}\label{lasp1}
\begin{gather}
\hat{\phi} \in C([0,\hat{T}];L_{x}^{2}) \cap L^{\infty}(0,\hat{T};H_{x}^{3}), 
\label{lphi1}\\
e^{\alpha x /2} \hat{F} \in C([0,\hat{T}];H^{1}_{0,\Gamma}), \ \ {\rm supp} \hat{F} \! \subset \! \{ (t, x,\xi) \in \overline{[0,\hat{T}] \! \times \! {\mathbb R}_{+} \! \times \! {\mathbb R}^{3}} \, | \, |x|^{2}\!+\!|\xi_{1}|^{2}\!\geq\! s^{2}, \ \xi_{1} \leq R \},
\label{lF1}\\
e^{\alpha x /2} f_{0} \in H^{1}_{0,\Gamma},\ \
{\rm supp} f_{0} \subset  \{ (x,\xi) \in \overline{{\mathbb R}_{+} \times{\mathbb R}^{3}} \, | \, |x|^{2}+|\xi_{1}|^{2}\geq s^{2}, \ \xi_{1} \leq R \}
\label{lini1}
\end{gather}
\end{subequations}
for some constant $\hat{T}>0$. 
Then there exist constants $T_{0} \in (0,\hat{T}]$ and $\delta_{0} \in (0,1]$ such that if $\Phi_{b} \leq \delta_{0}$ and $\sup_{t \in [0,\hat{T}]} \| \hat{\phi} (t) \|_{H^{3}_{x}} \leq \delta_{0}$, the initial--boundary value problem \eqref{l1} has a unique solution $f$ that satisfies
\begin{gather}
e^{\alpha x /2} f \in C([0,T_{0}];H^{1}_{x,\xi}), \quad   e^{\alpha x /2} (1+|\xi_{1}|)^{-1} f \in C^{1}([0,T_{0}];L_{x,\xi}^{2}). 
\label{Class_solution}
\end{gather}
Furthermore, the solution $f$ satisfies the following:
\begin{subequations}\label{localpro1}
\begin{gather}
f \in {\cal X}_{e}(T_{0},\alpha), 
\label{localpro1a}\\
{\rm supp} f \subset \{ (t, x,\xi) \in \overline{[0,T_{0}]  \times \! {\mathbb R}_{+}  \times  {\mathbb R}^{3}} \, | \,  \xi_{1} \leq R +C_{0} t \}, 
\label{localpro1b}\\
\begin{aligned}
&\| e^{\alpha x /2} f(t)\|_{H_{x,\xi}^1}^2 
+\frac{\alpha}{2} \sum_{k=0}^{1} \int_0^te^{C_0(t-\tau)} \left\||\xi_1|^{1/2} \nabla_{x,\xi}^{k} (e^{\alpha x /2} f)(\tau) \chi(-\xi_{1})   \right\|_{L_{x,\xi}^2}^2 \,d\tau
\\
& \leq e^{C_0 t}\| e^{\alpha x /2} f_{0}\|_{H_{x,\xi}^1}^2 
+\int_0^t e^{C_0(t-\tau)}\|e^{\alpha x /2} \hat{F}(\tau)\|_{H_{x,\xi}^1}^2\,d\tau \quad \text{for $t \in [0,T_{0}]$},
\label{localpro1c}
\end{aligned}
\end{gather}
\end{subequations}
where $C_0$ is a positive constant. 
\end{lem}
\begin{proof} 
We show this lemma in Appendix \ref{SB}, since the proof is long.
\end{proof}

\begin{lem}\label{sovl2}
Let $\alpha \in [0,1]$. Suppose that $e^{\alpha x /2} \hat{n} \in C([0,\hat{T}];H_{x}^{1})$ for $\hat{T}>0$.
Then there exists a constant $\delta_{1}>0$ such that if $\Phi_{b} \leq \delta_{1}$ and $\sup_{t \in [0,\hat{T}]}\|e^{\alpha x /2} \hat{n}(t)\|_{H_{x}^{1}} \leq \delta_{1}$, 
the boundary value problem \eqref{l2} has a unique solution $\phi$ that satisfies $e^{\alpha x /2} \phi \in C([0,\hat{T}];H_{x}^{3})$.
\end{lem}
\begin{proof}
This lemma can be shown in much the same way as in subsection 3.1 in \cite{M.S.1}. 
\end{proof}

We define the set ${\cal X}_{e}(T,\alpha,K) \times {\cal Y}_{e}(T,\alpha,L)$ to be the collection of the elements $(\hat{f},\hat{\phi})$ satisfying
\begin{gather*}
(\hat{f},\hat{\phi}) \in  {\cal X}_{e}(T,\alpha) \times {\cal Y}_{e}(T,\alpha),
\\
\sup_{t \in [0,T]} \|e^{\alpha x/ 2} \hat{f} (t) \|_{H^{1}_{x,\xi}}^{2} 
+ \frac{\alpha}{2} \int_{0}^{T} \left\| e^{\alpha x/ 2} |\xi_{1}|^{1/2}  \hat{f}(\tau) \chi(-\xi_{1})   \right\|_{L^{2}_{x,\xi}}^{2} \,d\tau \leq K^{2},
\\
\sup_{t \in [0,T]} \| e^{\alpha x/ 2} \hat{\phi} (t) \|_{H^{3}_{x}}^{2} \leq L^{2}.
\end{gather*}
Let us show that there exist positive constants $T$, $K$, and $L$ such that
${\cal X}_{e}(T,\alpha,K) \times {\cal Y}_{e}(T,\alpha,L)$ is invariant under the mapping 
$A:(\hat{f},\hat{\phi}) \mapsto (f,\phi)$ defined by solving \eqref{l1} and \eqref{l2} 
with $\hat{F}=-(\pd_{\xi_{1}}F^{s}) M_{\infty}^{-1/2} \partial_{x} \hat{\phi}$ and $\hat{n}= \int_{\mathbb R^{3}}  M_{\infty}^{1/2} \hat{f} d\xi$.

\begin{lem}\label{invariant}
Let $\alpha \in (0,1)$, $s \in (0,r] $, and $R \geq 1$ for $r>0$ being in \eqref{cutoff0}.
There exist constants $\delta_{\sharp} \in (0,1]$, $K_{\sharp} \geq 1$, $L_{\sharp} \geq 1$, and $T_{\sharp} \in (0,1]$
such that for any $\delta \in (0,\delta_{\sharp}] $ and $T \in (0,T_{\sharp}]$,
if the initial data $f_{0}$ satisfies \eqref{lini1}, $\Phi_{b}+\|e^{\alpha x /2} f_{0}\|_{H_{x,\xi}^{1}} \leq \delta$ holds,
and $(\hat{f},\hat{\phi})$ belongs to ${\cal X}_{e}(T,\alpha,K_{\sharp}\delta) \times {\cal Y}_{e}(T,\alpha,L_{\sharp}\delta)$,
then the problem \eqref{l1} with $\hat{F}=-(\pd_{\xi_{1}}F^{s}) M_{\infty}^{-1/2} \partial_{x} \hat{\phi}$ has a unique solution $f \in {\cal X}_{e}(T,\alpha,K_{\sharp}\delta)$, and the problem \eqref{l2} with $\hat{n}= \int_{\mathbb R^{3}}  M_{\infty}^{1/2} \hat{f} d\xi$ has a unique solution $\phi \in {\cal Y}_{e}(T,\alpha,L_{\sharp}\delta)$.
\end{lem}
\begin{proof}
This lemma can be shown by the energy method. We omit the proof, since it is simpler than those of Propositions \ref{apriori1} and \ref{apriori2}.
\end{proof}

In order to show Proposition \ref{local1}, it is sufficient to show Lemma \ref{local2} below. 

\begin{lem}\label{local2}
Let $\alpha \in (0,1)$, $s \in (0,r] $, and $R \geq 1$ for $r>0$ being in \eqref{cutoff0}.
There exist constants $\delta_{*} \in (0,1]$, $K_{*} \geq 1$, $L_{*} \geq 1$, and $T_{*} \in (0,1]$
such that if the initial data $f_{0}$ satisfies \eqref{lini1} and $d:=\Phi_{b}+\|e^{\alpha x /2} f_{0}\|_{H_{x,\xi}^{1}} \leq \delta_{*}$ holds,
then the nonlinear problem \eqref{ree0} has a unique solution $(f,\phi) \in {\cal X}_{\rm e}(T_{*},\alpha) \times {\cal Y}_{\rm e}(T_{*},\alpha)$. Furthermore, $(f,\phi) \in {\cal X}_{\rm e}(T_{*},\alpha,K_{*}d) \times {\cal Y}_{\rm e}(T_{*},\alpha,L_{*}d)$ holds.
\end{lem}
\begin{proof}
Let $\delta_{\sharp} \in (0,1]$, $K_{\sharp} \geq 1$, $L_{\sharp} \geq 1$, and $T_{\sharp} \in (0,1]$ 
be the same constants being in Lemma~\ref{invariant}.
We define the successive approximation sequence $\{(f^{k},\phi^{k})\}_{k=0}^\infty$ by $(f^{0},\phi^{0})=(0,0)$ for $k=0$, and for $k \geq 1$, 
\begin{gather*}
\pd_{t} f^{k} + \xi_{1}\partial_{x} f^{k} + (\partial_{x} \Phi^{s} + \partial_{x}\phi^{k-1}) \partial_{\xi_{1}}f^{k} 
-  \frac{\xi_{1}-u_{\infty}}{2\theta_{\infty}}  (\partial_{x} \Phi^{s} + \partial_{x}\phi^{k-1})   f^{k} = \frac{\pd_{\xi_{1}}F^{s}}{ M_{\infty}^{1/2}} \partial_{x} \phi^{k-1},
\\
f^{k}(0,x,\xi) = f_{0}(x,\xi), \quad
f^{k}(t,0,\xi) = 0, \ \xi_{1}>0, \quad
\lim_{x \to\infty} f^{k}(t,x,\xi) = 0
\end{gather*}
and
\begin{gather*}
\partial_{xx} \phi^{k} =  \int_{\mathbb R^{3}}  M_{\infty}^{1/2} f^{k-1} d\xi - (n_{e}(\Phi^{s}+\phi^{k})-n_{e}(\Phi^{s})), 
\\
\phi^{k}(t,0)=0, \quad
\lim_{x \to\infty} \phi^{k}(t,x) = 0.
\end{gather*}
Lemma \ref{invariant} ensures that for all $k \geq 1$, $\delta \in (0,\delta_{\sharp}]$, and $T \in (0,T_{\sharp}]$,
the approximation $(f^{k},\phi^{k})$ is well-defined and belongs to ${\cal X}_{\rm e}(T,\alpha,K_{\sharp}\delta)$ $\times {\cal Y}_{\rm e}(T,\alpha,L_{\sharp}\delta)$ if $d \leq \delta$ holds.

By the standard energy method, it is seen by taking $\delta_{\sharp}$ and $T_{\sharp}$ appropriately small if necessary that
$\{(e^{\alpha x /2 }f^{k},e^{\alpha x /2 }\phi^{k})\}$ is a Cauchy sequence in $C([0,T_{\sharp}];L_{x,\xi}^{2}) \! \times \! C([0,T_{\sharp}];H_{x}^{2})$. 
Hence, there exists $(f,\phi)$ with
$(e^{\alpha x /2 }f,e^{\alpha x /2 }\phi) \in C([0,T_{\sharp}];L_{x,\xi}^{2}) \times C([0,T_{\sharp}];H_{x}^{2})$ such that
as $k \to \infty$,
\begin{gather*}
(e^{\alpha x /2 }f^{k},e^{\alpha x /2 }\phi^{k}) \to (e^{\alpha x /2 }f,e^{\alpha x /2 }\phi) \quad
\text{in $C([0,T_{\sharp}];L_{x,\xi}^{2}) \times C([0,T_{\sharp}];H_{x}^{2})$.} 
\end{gather*}
Using this, the Banach--Alaoglu theorem, and the fact that $\{(f^{k},\phi^{k})\} \subset {\cal X}_{\rm e}(T_{\sharp},\alpha,K_{\sharp}\delta) \times {\cal Y}_{\rm e}(T_{\sharp},\alpha,L_{\sharp}\delta)$, we conclude that
\begin{gather*}
e^{\alpha x/ 2} f \in L^{\infty}(0,T_{\sharp};H_{x,\xi}^{1}), \quad
\sup_{t \in [0,T_{\sharp}]} \|e^{\alpha x/ 2} f (t) \|_{H^{1}_{x,\xi}}  \leq K_{\sharp} \delta,
\\
e^{\alpha x/ 2} \phi \in L^{\infty}(0,T_{\sharp};H_{x}^{3}), \quad
\sup_{t \in [0,T_{\sharp}]} \| e^{\alpha x/ 2} \phi(t)  \|_{H^{3}_{x}} \leq L_{\sharp} \delta.
\end{gather*}
It is straightforward to see that $(f,\phi)$ is a solution of the nonlinear problem \eqref{ree0}.

Let us investigate the regularity of $(f,\phi)$. 
It is clear that $f$ solves the linear problem \eqref{l1} with $\hat{\phi}=\phi$ and $\hat{F}=-(\pd_{\xi_{1}}F^{s}) M_{\infty}^{-1/2}\partial_{x} \phi$.
We set 
\begin{gather*}
\delta_{*}:= \min\left\{\delta_{\sharp}, \frac{\delta_{0}}{L_{\sharp}}, \frac{\delta_{1}}{\sqrt{\rho_{\infty}}K_{\sharp}} \right\}, \quad T_{*}:= \min\left\{T_{\sharp}, T_{0}\right\},
\end{gather*}
where $\rho_{\infty}$, $\delta_{0}$, $\delta_{1}$, and $T_{0}$ are the same constants being in \eqref{nm3} and Lemmas \ref{sovl1}--\ref{sovl2}.
Now taking $\delta=d \leq \delta_{*}$ and applying Lemma \ref{sovl1} to the problem \eqref{l1} with $\hat{\phi}=\phi$ and $\hat{F}=-(\pd_{\xi_{1}}F^{s}) M_{\infty}^{-1/2}\partial_{x} \phi$,
we conclude that $f$ satisfies \eqref{localpro1} with $T_{0}=T_{*}$. 
Hence, $f \in {\cal X}_{\rm e}(T_{*},\alpha,K_{*}d)$ holds by letting $K_{*}=K_{\sharp}$. 
Similarly, we set $L_{*}=L_{\sharp}$ and then see from Lemma \ref{sovl2} that $\phi \in {\cal Y}_{\rm e}(T_{*},\alpha,L_{*}d)$.

It is straightforward to show the uniqueness of the solution $(f,\phi)$ of the problem \eqref{ree0}.
The proof is complete.
\end{proof}

\section{A Priori Estimate on the Stability}\label{AE}
This section is devoted to the proof of Proposition \ref{apriori1} which gives the stability theorem.
We will choose $\gamma>0$ at the end of this section, 
and all constants $c$ and $C$ in this section are independent of $\gamma$.
We first study the location of the support of $f$.
To this end, we use the following equality:
\begin{align}
& \pd_{t} |f|^{2} + \partial_{x} (\xi_{1} |f|^{2}) 
+  \partial_{\xi_{1}} \{ (\pd_{x}\Phi^{s} +\pd_{x} \phi) \vert f \vert^2\}
\notag \\
&= - 2\frac{\pd_{\xi_{1}}F^{s}}{M_{\infty}^{1/2}} (\partial_{x} \phi) f
+  \frac{\xi_{1}-u_{\infty}}{\theta_{\infty}} (\partial_{x} \Phi^{s} + \partial_{x}\phi)|f|^{2}.
\label{suppes1}
\end{align}
This immediately follows from multiplying $2f$ to \eqref{ree1}. 
We note that
\begin{gather}\label{N2}
N_{0,0}(T) \leq  N_{\beta,0}(T) \leq N_{\beta,\gamma}(T),
\end{gather}
where $N_{\beta,\gamma}(T)$ is defined in \eqref{N0}.

\begin{lem}\label{alem1}
Suppose that the same assumption as in Proposition \ref{apriori1} holds.
For any $\gamma>0$, there exists a constant $\delta_{1}=\delta_{1}(\gamma)>0$ such that 
if $\Phi_{b}+N_{\beta,\gamma}(T)  \leq \delta_{1}$, then \eqref{aes0} holds.
\end{lem}
\begin{proof}
First it is seen from \eqref{decay0}, \eqref{nm3}, \eqref{ellineq0}, \eqref{N2}, and $\Phi_{b}+N_{\beta,\gamma}(T)  \leq \delta_{1} \ll 1$ that 
\begin{gather}\label{phimax}
\sup_{(t,x) \in [0,T] \times {\mathbb R}_{+} } |\pd_{x}\Phi^{s} (x) +\pd_{x}\phi (t,x)| \leq C \delta_{1} \leq 1, \quad 
\sup_{x \in {\mathbb R}_{+} } |\pd_{x}\phi (t,x)| \leq C \delta_{1} e^{-\gamma t/2},
\end{gather}
where $C>0$ is a constant independent of $\beta$, $\gamma$, $\delta_{1}$, $t$, and $T$.
For any $L>T$ and $t \in [0,T]$, we set
\begin{gather*}
\Omega_{L}(t):=\{ (\tau,x,\xi)  \, | \,  \tau \in [0,t], \ x \in \overline{{\mathbb R}_{+}}, \  \xi \in D_{L}(\tau) \},
\\
D_{L}(\tau):=\{ \xi \, | \,  \Xi^{-}_{L}(\tau) \leq \xi_{1} \leq \Xi^{+}_{L}(\tau), \ \xi' \in \mathbb R^{2}\},
\\
\Xi^{-}_{L}(\tau):=-(r-\ve) +  \ve\frac{\gamma}{2} \int_{0}^{\tau} e^{-\gamma s/2} ds, \quad
\Xi^{+}_{L}(\tau):=2L - \tau, 
\\
\Gamma_{L}^{\pm}(t):=\left\{ (\tau,x,\xi) \, | \, \tau \in [0,t], \ x \in \overline{{\mathbb R}_{+}}, \  \xi_{1} = \Xi^{\pm}_{L}(\tau), \ \xi' \in \mathbb R^{2} \right\}, 
\\
N^{\pm}_{L}(\tau):= \sqrt{1+|\partial_{\tau}  \Xi^{\pm}_{L}(\tau)|^{2}},
\end{gather*}
where $r$ and $\ve$ are the same constants being in the assumption of Proposition \ref{apriori1}.
We note that 
\begin{gather}\label{Xi0}
 -(r-\ve) \leq \Xi^{-}_{L}(\tau) \leq -(r-2\ve), \quad  L \leq \Xi^{+}_{L} (\tau) \leq 2L.
\end{gather}

It is clear that $F^{s}=0$ on $\Omega_{L}(t)$ and $f_{0}=0$ on $\overline{\mathbb R_{+}} \times D_{L}(0)$ owing to \eqref{r2} and \eqref{supp0}.
Integrating \eqref{suppes1} over $\Omega_{L}(t)$ and applying the divergence theorem with \eqref{rebc1}, we see that
\begin{align}
&\int_{\mathbb R_{+}}\!\!\int_{D_{L}(t)} |f|^{2}(t) dxd\xi 
+ \left. \int_{0}^{t}\int_{D_{L}(t)} \vert \xi_1\vert \vert f \vert^2 \chi(-\xi_{1})  d\tau d\xi \right|_{x=0}
\notag \\
&\quad + \int\!\! \int_{\Gamma_{L}^{+}(t)} \frac{1}{N^{+}_{L}} \left\{(\pd_{x}\Phi^{s} +\pd_{x} \phi) - \partial_{\tau}  \Xi^{+}_{L}\right\} |f|^{2}  dS
\notag \\
&\quad + \int\!\! \int_{\Gamma_{L}^{-}(t)} \frac{1}{N^{-}_{L}}  \left\{\partial_{\tau}  \Xi^{-}_{L}- (\pd_{x}\Phi^{s} +\pd_{x} \phi) \right\} |f|^{2}  dS
 \notag \\
& = \int_{0}^{t}\!\! \int_{\mathbb R_{+}}\!\!\int_{D_{L}(\tau)} \frac{\xi_{1}-u_{\infty}}{\theta_{\infty}} (\partial_{x} \Phi^{s} + \partial_{x}\phi) |f|^{2}d\tau dxd\xi
\notag \\
& \leq C(L) \int_{0}^{t}\!\! \int_{\mathbb R_{+}}\!\!\int_{D_{L}(\tau)} |f|^{2}(\tau) d\tau dxd\xi,
\label{suppes2}
\end{align}
where  $C(L)>0$ is a constant depending on $L$, and we have used \eqref{phimax} in deriving the inequality.
It is obvious that the second term on the left hand side is nonnegative.
We also see from \eqref{phimax} and $\pd_{x} \Phi^{s}(x)<0$
that the integrands of the third and fourth terms on the left hand side are nonnegative for suitably small $\delta_{1}$
according to $\ve$ and $\gamma$ as follows:
\begin{align*}
&\frac{1}{N^{+}_{L}}  \left\{(\pd_{x}\Phi^{s} +\pd_{x} \phi) - \partial_{\tau}  \Xi^{+}_{L}\right\} |f|^{2}
\\
&= \frac{1}{N^{+}_{L}}  \left\{(\pd_{x}\Phi^{s} +\pd_{x} \phi) +1 \right\} |f|^{2} 
\geq \frac{1}{N^{+}_{L}}  \left(1 - |\pd_{x}\Phi^{s}+\pd_{x}\phi| \right) |f|^{2} 
\geq 0
\end{align*}
and
\begin{align*}
&\frac{1}{N^{-}_{L}}  \left\{\partial_{\tau}  \Xi^{-}_{L}- (\pd_{x}\Phi^{s} +\pd_{x} \phi) \right\} |f|^{2}
\\
&= \frac{1}{N^{-}_{L}}  \left\{\ve\frac{\gamma}{2} e^{-\gamma \tau/2} - (\pd_{x}\Phi^{s} +\pd_{x} \phi) \right\} |f|^{2}
\geq  \frac{1}{N^{-}_{L}}  \left\{\ve\frac{\gamma}{2} e^{-\gamma \tau/2} - C \delta_{1} e^{-\gamma t/2} \right\} |f|^{2}
\geq 0.
\end{align*}
Now applying Gronwall's inequality to \eqref{suppes2}, we conclude that
\begin{gather*}
\int_{\mathbb R_{+}}\!\!\int_{D_{L}(t)} |f|^{2}(t) dxd\xi =0 \quad \text{for $t \in [0,T]$}.
\end{gather*}
On the other hand, it is seen from \eqref{Xi0} that 
\begin{gather*}
\tilde{D}_{L}:=\{ \xi \in {\mathbb R}^{3} \, | \, -(r-2\ve)  \leq \xi_{1}  \leq L  \} \subset D_{L}(t).
\end{gather*}
Using this and the arbitraryness of $L>0$, we arrive at \eqref{aes0}.
\end{proof}

We estimate the $L^{2}_{x,\xi}$-norm of $f$ and $\pd_{x} f$ by assuming \eqref{aes0}.

\begin{lem}\label{alem2}
Suppose that the same assumption as in Proposition \ref{apriori1} and also \eqref{aes0} hold. 
There exists a constant $\delta_{2}>0$ such that if $\Phi_{b} + N_{0,0}(T)  \leq \delta_{2}$, then the following holds:
\begin{align}
&e^{\gamma t} \Vert e^{\beta x/2}  f (t)\Vert^2_{L^2_{x,\xi}} 
+ e^{\gamma t} \Vert  e^{\beta x/2} \pd_{x} f (t) \Vert^2_{L^2_{x,\xi}}
+ \frac{e^{\gamma t} }{\theta_{\infty}}\Vert e^{\beta x/2}   n (t) \Vert^2_{L^2_{x}}
\notag \\
&\quad + c_{1} \int_{0}^{t} 
e^{\gamma \tau} \left\Vert e^{\beta x/2} |\xi_{1}|^{1/2} f (\tau)\right\Vert^2_{L^2_{x,\xi}} 
\!\!+ e^{\gamma \tau} \left\Vert  e^{\beta x/2} |\xi_{1}|^{1/2}  \pd_{x} f (\tau) \right\Vert^2_{L^2_{x,\xi}}
\!\!+  e^{\gamma \tau} \|e^{\beta x/2} n (\tau) \|_{L^{2}_{x}}^{2} 
d\tau
\notag \\
& \leq \Vert e^{\beta x/2}  f_{0}\Vert^2_{L^2_{x,\xi}} 
+ \Vert  e^{\beta x/2} \pd_{x} f_{0} \Vert^2_{L^2_{x,\xi}}
+ \frac{1}{\theta_{\infty}}\Vert e^{\beta x/2}   n (0) \Vert^2_{L^2_{x}}
\notag \\
&\quad + \gamma \int_{0}^{t} 
e^{\gamma \tau} \Vert e^{\beta x/2}  f (\tau) \Vert^2_{L^2_{x,\xi}} 
+ e^{\gamma \tau} \Vert  e^{\beta x/2} \pd_{x} f (\tau) \Vert^2_{L^2_{x,\xi}}
+ \frac{e^{\gamma \tau}}{\theta_{\infty}} \|e^{\beta x/2} n (\tau) \|_{L^{2}_{x}}^{2} 
d\tau
\notag \\
& \quad + C_{1} \sqrt{\Phi_b + N_{0,0}(T)}  \sum_{k=0}^{1}
\int_{0}^{t} e^{\gamma \tau} 
\left\| e^{\beta x/2} |\xi_{1}|^{1/2} \nabla_{x,\xi}^{k} f (\tau) \right\|_{L^{2}_{x,\xi}}^{2}
d\tau
\label{aes2}
\end{align}
for any $\gamma>0$, where $c_{1}$ and $C_{1}$ are positive constants independent of $\gamma$, $\Phi_b$, $f_{0}$, $t$, and $T$.
\end{lem}
\begin{proof}
We start from deriving some equalities.
Multiplying $2e^{\beta x} f$ to \eqref{ree1} and using \eqref{M0} yields
\begin{align*}
&\partial_{t} (e^{\beta x}\vert  f \vert^2 )
+ \partial_{x} (\xi_{1} e^{\beta x}  \vert f \vert^2)
-\beta \xi_1 e^{\beta x}  \vert f \vert^2
+\partial_{\xi_{1}} \{e^{\beta x} (\pd_{x}\Phi^{s} +\pd_{x} \phi) \vert f \vert^2\}
\\
&- 2\pd_{x} \left( e^{\beta x} \phi \frac{\xi_{1}-u_{\infty}}{\theta_{\infty}} M_{\infty}^{1/2} f \right)
+  2e^{\beta x} \phi \frac{\xi_{1}-u_{\infty}}{\theta_{\infty}} M_{\infty}^{1/2} \pd_{x} f
+ 2\beta e^{\beta x} \phi \frac{\xi_{1}-u_{\infty}}{\theta_{\infty}} M_{\infty}^{1/2} f 
\\
&+2e^{\beta x} (\partial_{x} \phi) \frac{\pd_{\xi_{1}}(M_{\infty}\psi-M_{\infty})}{M_{\infty}^{1/2}} f
= {\cal R}_{1},
\end{align*}
where
\begin{gather*}
 {\cal R}_{1} :=- 2e^{\beta x} (\pd_{x} \phi ) \frac{\pd_{\xi_{1}}(F^{s}-M_{\infty}\psi)}{M_{\infty}^{1/2}}  f
 + e^{\beta x} \frac{\xi_{1}-u_{\infty}}{\theta_{\infty}} (\partial_{x} \Phi^{s} + \partial_{x}\phi)  \vert f \vert^2.
\end{gather*}
Integrate this over $\mathbb R_{+} \times \mathbb{R}^3$ and use \eqref{rebc1}--\eqref{rebc4} and \eqref{aes0} to obtain
\begin{align}
&\dfrac{d}{dt} \Vert e^{\beta x/2} f (t) \Vert^2_{L^2_{x,\xi}} 
+\left. \int_{\mathbb{R}^3_-} \vert \xi_1\vert \vert  f \vert^2 d\xi \right|_{x=0}
+ \beta  \int_{\mathbb{R}_{+}} \! \int_{\mathbb{R}^3_{-}}  e^{\beta x}  | \xi_1 | |f|^2  dxd\xi 
\notag \\
& + \frac{2}{\theta_{\infty}}  \int_{\mathbb{R}_{+}}  e^{\beta x} \phi \pd_{x} m dx
- \frac{2u_{\infty}}{\theta_{\infty}}  \int_{\mathbb{R}_{+}}  e^{\beta x} \phi \pd_{x} n dx
 +\frac{2\beta}{\theta_{\infty}}  \int_{\mathbb{R}_{+}} e^{\beta x} \phi m dx 
-  \frac{2u_{\infty}\beta}{\theta_{\infty}}  \int_{\mathbb{R}_{+}} e^{\beta x}  \phi n dx 
\notag \\ 
&  + 2 \int_{\mathbb{R}_{+}} \! \int_{\mathbb{R}^3} e^{\beta x} (\partial_{x} \phi)  \frac{\pd_{\xi_{1}}(M_{\infty}\psi-M_{\infty})}{M_{\infty}^{1/2}} f d\xi dx
=\int_{\mathbb{R}_{+}} \! \int_{\mathbb{R}^3} {\cal R}_{1} dxd\xi,
\label{bes1}
\end{align}
where $n$ and $m$ are defined in \eqref{nm0}.
Then, we differentiate \eqref{ree1} with respect to $x$,
multiply $2e^{\beta x} \pd_{x} f$ to the resultant equation and use \eqref{M0}, and arrive at
\begin{align*}
&\partial_{t} (e^{\beta x}\vert  \pd_{x} f \vert^2 )
+ \partial_{x} (\xi_{1} e^{\beta x}  \vert \pd_{x} f \vert^2)
-\beta \xi_1 e^{\beta x}  \vert \pd_{x} f\vert^2
+\partial_{\xi_{1}} \{e^{\beta x} (\pd_{x}\Phi^{s} +\pd_{x} \phi) \vert \pd_{x} f \vert^2\}
\\
&-   2(\pd_{xx}\phi)  e^{\beta x}  \frac{\xi_{1}-u_{\infty}}{\theta_{\infty}} M_{\infty}^{1/2} \pd_{x} f
+2e^{\beta x} (\partial_{xx} \phi) \frac{\pd_{\xi_{1}}(M_{\infty}\psi-M_{\infty})}{M_{\infty}^{1/2}} \pd_{x} f
=  {\cal R}_{2},
\end{align*}
where 
\begin{align*}
{\cal R}_{2}:= & - 2e^{\beta x} (\pd_{xx}\Phi^{s} +\pd_{xx} \phi) \pd_{\xi_{1}} f \pd_{x} f 
 - 2e^{\beta x} \pd_{x} \left \{\frac{\pd_{\xi_{1}}(F^{s}-M_{\infty}\psi)}{M_{\infty}^{1/2}}  \pd_{x} \phi \right\}\pd_{x} f 
 \\
& +e^{\beta x} \frac{\xi_{1}-u_{\infty}}{\theta_{\infty}} \pd_{x} \left \{  (\partial_{x} \Phi^{s} + \partial_{x}\phi) f \right\} \pd_{x} f .
\end{align*}
Integrate this over $\mathbb R_{+} \times \mathbb{R}^3$, 
and use \eqref{rebc1}--\eqref{rebc4} and \eqref{aes0}  to obtain
\begin{align}
&\dfrac{d}{dt} \Vert e^{\beta x/2}   \pd_{x} f (t) \Vert^2_{L^2_{x,\xi}} 
+\left.  \int_{\mathbb{R}^3_-} \vert \xi_1\vert \vert \pd_{x} f \vert^2 d\xi \right|_{x=0}
+ \beta  \int_{\mathbb{R}_{+}} \! \int_{\mathbb{R}^3_{-}} e^{\beta x} |\xi_1| |\pd_{x}f|^2  dxd\xi 
\notag \\
& - \frac{2}{\theta_{\infty}}  \int_{\mathbb{R}_{+}}  e^{\beta x} \pd_{xx}\phi \pd_{x} m dx
+ \frac{2u_{\infty}}{\theta_{\infty}}  \int_{\mathbb{R}_{+}}  e^{\beta x} \pd_{xx}\phi \pd_{x} n dx
\notag \\
&  + 2\int_{\mathbb{R}_{+}} \! \int_{\mathbb{R}^3} e^{\beta x} (\partial_{xx} \phi)  \frac{\pd_{\xi_{1}}(M_{\infty}\psi-M_{\infty})}{M_{\infty}^{1/2}} \pd_{x} f d\xi dx
=\int_{\mathbb{R}_{+}} \! \int_{\mathbb{R}^3} {\cal R}_{2} dxd\xi.
\label{bes2}
\end{align}
Adding up \eqref{bes1} and \eqref{bes2} yields
\begin{align}
&\dfrac{d}{dt} (\Vert e^{\beta x/2}   f (t) \Vert^2_{L^2_{x,\xi}} 
+\Vert e^{\beta x/2}   \pd_{x} f (t) \Vert^2_{L^2_{x,\xi}})
+\left. \int_{\mathbb{R}^3_-} \vert \xi_1\vert \vert f \vert^2 +\vert \xi_1\vert \vert \pd_{x} f \vert^2  d\xi \right|_{x=0}
\notag\\
&\quad  + \beta  \int_{\mathbb{R}_{+}} \! \int_{\mathbb{R}^3_{-}}  e^{\beta x}  |\xi_1| ( |f|^2+|\pd_{x}f|^2)  dxd\xi 
\notag\\
&\quad - \frac{2}{\theta_{\infty}}  \int_{\mathbb{R}_{+}}  e^{\beta x} (\pd_{xx}\phi - \phi)\pd_{x} m dx
+ \frac{2u_{\infty}}{\theta_{\infty}}  \int_{\mathbb{R}_{+}}  e^{\beta x} (\pd_{xx}\phi - \phi) \pd_{x} n dx
\notag \\
&=- \frac{2\beta}{\theta_{\infty}}  \int_{\mathbb{R}_{+}} e^{\beta x} \phi (m-u_{\infty}n) dx  
- 2\int_{\mathbb{R}_{+}} \! \int_{\mathbb{R}^3}  e^{\beta x} (\partial_{x} \phi)  \frac{\pd_{\xi_{1}}(M_{\infty}\psi-M_{\infty})}{M_{\infty}^{1/2}} f d\xi dx
\notag \\
&\quad 
- 2\int_{\mathbb{R}_{+}} \! \int_{\mathbb{R}^3}  e^{\beta x} (\partial_{xx} \phi)  \frac{\pd_{\xi_{1}}(M_{\infty}\psi-M_{\infty})}{M_{\infty}^{1/2}} \pd_{x} f d\xi dx
 +\int_{\mathbb{R}_{+}} \! \int_{\mathbb{R}^3} {\cal R}_{1} + {\cal R}_{2} dxd\xi.
\label{bes3}
\end{align}


Let us find \footnote{In \eqref{bes3} we have the dissipative term, i.e. the third term on the left hand side, which is arisen by employing the weight function $e^{\beta x/ 2}$. But the terms having $\theta_{\infty}^{-1}$ in \eqref{bes3} are still problematic to study the case $\theta_{\infty} \ll 1$.}some good contributions from the forth and fifth terms on the left hand side of \eqref{bes3}.
First we rewrite the forth term by using \eqref{mass1} and \eqref{pe3} as 
\begin{align*}
& -\frac{2}{\theta_{\infty}}  \int_{\mathbb{R}_{+}}  e^{\beta x} (\pd_{xx}\phi - \phi)\pd_{x} m dx
=  - \frac{2}{\theta_{\infty}}  \int_{\mathbb{R}_{+}}  e^{\beta x} (n + {\cal N}_{1}) \pd_{x} m dx
\\
& = \frac{1}{\theta_{\infty}}\dfrac{d}{dt} \Vert e^{\beta x/2}   n (t) \Vert^2_{L^2_{x}} 
-\frac{2}{\theta_{\infty}}  \int_{\mathbb{R}_{+}}  e^{\beta x} {\cal N}_{1} \pd_{x} m dx.
\end{align*}
The fifth term can be rewritten by using \eqref{pe3} as 
\begin{align*}
&\frac{2u_{\infty}}{\theta_{\infty}}  \int_{\mathbb{R}_{+}}  e^{\beta x} (\pd_{xx}\phi - \phi) \pd_{x} n dx
=\frac{2u_{\infty}}{\theta_{\infty}}  \int_{\mathbb{R}_{+}}  e^{\beta x} (n+{\cal N}_{1}) \pd_{x} n dx
\\
&=\frac{|u_{\infty}|}{\theta_{\infty}}  |n|^{2}(t,0) + \frac{|u_{\infty}|}{\theta_{\infty}}\beta\Vert e^{\beta x/2} n \Vert^2_{L^2_{x}} 
+\frac{2u_{\infty}}{\theta_{\infty}}\int_{\mathbb{R}_{+}} e^{\beta x}   {\cal N}_{1} \pd_{x} n dx.
\end{align*}
%
%
Plugging these two into \eqref{bes3}, we have the \footnote{In the derivation of \eqref{bes4}, the weight function $M^{-1/2}_{\infty}$ in \eqref{pert1} plays an essential role. Indeed, without this weight function, 
we cannot write some integrals in \eqref{bes1} and \eqref{bes2} by using hydrodynamic variables $n$ and $m$, and therefore we cannot find the dissipative term, i.e. the last term on the left hand side of \eqref{bes4}, which is bigger for lower temperature $\theta_{\infty} \ll 1$.}following equality:
\begin{align}
&\dfrac{d}{dt} \left(\Vert e^{\beta x/2}   f (t) \Vert^2_{L^2_{x,\xi}} 
+\Vert e^{\beta x/2}   \pd_{x} f (t) \Vert^2_{L^2_{x,\xi}} 
+ \frac{1}{\theta_{\infty}}  \Vert e^{\beta x/2}   n (t) \Vert^2_{L^2_{x}} \right)
\notag \\
&\quad +\left. \left(\int_{\mathbb{R}^3_-} \vert \xi_1\vert \vert f \vert^2 
+\vert \xi_1\vert \vert \pd_{x} f \vert^2  d\xi 
+\frac{|u_{\infty}|}{\theta_{\infty}}|n|^{2}\right)\right|_{x=0}
\notag \\
&\quad  + \beta  \int_{\mathbb{R}_{+}} \! \int_{\mathbb{R}^3_{-}}  e^{\beta x} |\xi_1| ( |f|^2+|\pd_{x}f|^2)  dxd\xi 
+\frac{|u_{\infty}|}{\theta_{\infty}} \beta \| e^{\beta x/2} n \|_{L^{2}_{x}}^{2}
\notag\\
&= - \frac{2\beta}{\theta_{\infty}}  \int_{\mathbb{R}_{+}} e^{\beta x} \phi (m-u_{\infty}n) dx   
 - 2\int_{\mathbb{R}_{+}} \! \int_{\mathbb{R}^3} e^{\beta x} (\partial_{x} \phi)  \frac{\pd_{\xi_{1}}(M_{\infty}\psi-M_{\infty})}{M_{\infty}^{1/2}} f d\xi dx
\notag\\
&\quad - 2\int_{\mathbb{R}_{+}} \! \int_{\mathbb{R}^3} e^{\beta x} (\partial_{xx} \phi) \frac{\pd_{\xi_{1}}(M_{\infty}\psi-M_{\infty})}{M_{\infty}^{1/2}} \pd_{x} f d\xi dx
+ \int_{\mathbb{R}_{+}} \! \int_{\mathbb{R}^3} {\cal R}_{1} + {\cal R}_{2} dxd\xi
\notag\\
&\quad 
+ \frac{2}{\theta_{\infty}}  \int_{\mathbb{R}_{+}}  e^{\beta x} {\cal N}_{1} (\pd_{x} m -u_{\infty}\pd_{x} n)  dx.
\label{bes4}
\end{align}

Let us estimate the right hand side of \eqref{bes4} term by term. Set $ \delta:= \Phi_b + N_{0,0}(T)$.
The first term is estimated by using  \eqref{nm1} with $(\alpha,k)=(\beta,0)$, \eqref{ellineq1}, and \eqref{nm3} with $(\alpha,k)=(\beta,0)$ as 
\begin{align*}
&\left| 2\frac{\beta}{\theta_{\infty}}  \int_{\mathbb{R}_{+}} e^{\beta x} \phi (m-u_{\infty}n) dx \right|
\notag \\
& \leq 2\frac{\beta\rho_{\infty}^{1/2}}{\theta_{\infty}^{1/2}} \|e^{\beta x/2} \phi \|_{L^{2}_{x}} \| e^{\beta x/2} f \|_{L^{2}_{x,\xi}}
\notag \\
& \leq 2\frac{\beta\rho_{\infty}^{1/2}}{\theta_{\infty}^{1/2}} (1+\eta(\beta)) \|e^{\beta x/2} n \|_{L^{2}_{x}} \| e^{\beta x/2} f \|_{L^{2}_{x,\xi}}
+ C \delta^{1/2} \|e^{\beta x/2} n \|_{L^{2}_{x}}\| e^{\beta x/2} f \|_{L^{2}_{x,\xi}}
\notag \\
& \leq 2\frac{\beta\rho_{\infty}^{1/2}}{\theta_{\infty}^{1/2}} (1+\eta(\beta)) \|e^{\beta x/2} n \|_{L^{2}_{x}} \| e^{\beta x/2} f \|_{L^{2}_{x,\xi}}
+ C \delta^{1/2} \| e^{\beta x/2} f \|_{L^{2}_{x,\xi}}^{2}
\notag \\
& \leq 2\frac{\beta\rho_{\infty}^{1/2}}{\theta_{\infty}^{1/2}} \frac{1+\eta(\beta)}{(r - 2\ve)^{1/2}} \|e^{\beta x/2} n \|_{L^{2}_{x}} \left\| e^{\beta x/2} |\xi_{1}|^{1/2}  f \right\|_{L^{2}_{x,\xi}}
+ \frac{C \delta^{1/2} }{r - 2\ve} \left\| e^{\beta x/2} |\xi_{1}|^{1/2} f \right\|_{L^{2}_{x,\xi}}^{2},
\end{align*}
where we have also used \eqref{aes0} and the fact that $|\xi_{1}| > r-2\ve$ holds on the support of $f$
in deriving  the last inequality.
Recall the definition of $\mu_{\infty}$ in \eqref{mu_eta}.
We estimate the second and third terms by \eqref{ellineq1}, \eqref{nm3} with $(\alpha,k)=(\beta,0)$, and \eqref{aes0} as follows:
\begin{align*}
&\left| 2 \int_{\mathbb{R}_{+}} \! \int_{\mathbb{R}^3}  e^{\beta x} (\partial_{x} \phi) \frac{\pd_{\xi_{1}}(M_{\infty}\psi-M_{\infty})}{M_{\infty}^{1/2}} f d\xi dx \right|
\\
& \leq 2 \mu_{\infty} \| e^{\beta x/2} \pd_{x} \phi \|_{L^{2}_{x}} \| e^{\beta x/2} f \|_{L^{2}_{x,\xi}}
\\
& \leq 2 \mu_{\infty} (1+\eta(\beta))^{1/2}  \| e^{\beta x/2} n \|_{L^{2}_{x}} \| e^{\beta x/2} f \|_{L^{2}_{x,\xi}}
+ C \mu_{\infty} \delta^{1/2} \|e^{\beta x/2} n \|_{L^{2}_{x}} \| e^{\beta x/2} f \|_{L^{2}_{x,\xi}}
\\
& \leq 2 \mu_{\infty} \frac{(1+\eta(\beta))^{1/2}}{(r - 2\ve)^{1/2}} \| e^{\beta x/2} n \|_{L^{2}_{x}}
\left\| e^{\beta x/2} |\xi_{1}|^{1/2}  f \right\|_{L^{2}_{x,\xi}}
+\frac{C \mu_{\infty} \delta^{1/2}}{r - 2\ve}\left\| e^{\beta x/2} |\xi_{1}|^{1/2}  f \right\|_{L^{2}_{x,\xi}}^{2}.
\end{align*}
Similarly as above, using \eqref{aes0}, \eqref{nm3}, and \eqref{ellineq2}, we observe that
\begin{align*}
&\left| 2 \int_{\mathbb{R}_{+}} \! \int_{\mathbb{R}^3} e^{\beta x} (\partial_{xx} \phi) \frac{\pd_{\xi_{1}}(M_{\infty}\psi-M_{\infty})}{M_{\infty}^{1/2}} \pd_{x} f d\xi dx \right|
\\
& \leq 2 \mu_{\infty} \| e^{\beta x/2} \phi_{xx} \|_{L^{2}_{x}}  \| e^{\beta x/2} \pd_{x} f \|_{L^{2}_{x,\xi}}
\\
& \leq 2 \{1+\beta^{2}(1+\eta(\beta))^{2} \}^{1/2} \mu_{\infty} \| e^{\beta x/2} n \|_{L^{2}_{x}} \| e^{\beta x/2} \pd_{x} f \|_{L^{2}_{x,\xi}}
\\
& \quad +  C \mu_{\infty} \delta^{1/2}  \|e^{\beta x/2} n \|_{L^{2}_{x}} \| e^{\beta x/2} \pd_{x}f \|_{L^{2}_{x,\xi}}
\\
& \leq \frac{2 \{1+\beta^{2}(1+\eta(\beta))^{2} \}^{1/2} \mu_{\infty}}{(r - 2\ve)^{1/2}}  \| e^{\beta x/2} n \|_{L^{2}_{x}}
\left\| e^{\beta x/2} |\xi_{1}|^{1/2} \pd_{x} f \right\|_{L^{2}_{x,\xi}}
\\
& \quad +\frac{C \mu_{\infty} \delta^{1/2}}{r - 2\ve} \sum_{k=0}^{1} \left\| e^{\beta x/2} |\xi_{1}|^{1/2} \nabla_{x,\xi}^{k} f \right\|_{L^{2}_{x,\xi}}^{2}.
\end{align*}
It is straightforward to see from \eqref{decay0}, \eqref{aes0}, \eqref{M2*}, \eqref{M3*}, \eqref{ellineq0}, and \eqref{ellineq3} that 
\begin{align*}
\left|\int_{\mathbb{R}_{+}} \! \int_{\mathbb{R}^3} {\cal R}_{1} + {\cal R}_{2} dxd\xi\right|
&\leq  C\delta\sum_{k=0}^{1} 
\left\| e^{\beta x/2} (1+|\xi_{1}|^{1/2}) \nabla_{x,\xi}^{k} f \right\|_{L^{2}_{x,\xi}}^{2}
\\
&\leq   \frac{C \delta}{r - 2\ve}
\sum_{k=0}^{1} \left\| e^{\beta x/2} |\xi_{1}|^{1/2} \nabla_{x,\xi}^{k} f \right\|_{L^{2}_{x,\xi}}^{2}.
\end{align*}
Furthermore, by using \eqref{nm1} and \eqref{N1}, it is seen that
\begin{align*}
 \left| \frac{2}{\theta_{\infty}}  \int_{\mathbb{R}}  e^{\beta x} {\cal N}_{1} (\pd_{x} m -u_{\infty}\pd_{x} n)  dx \right|
& \leq  C\delta \left(\| e^{\beta x/2} f \|_{L^{2}_{x,\xi}}^{2}
+ \| e^{\beta x/2} \pd_{x}f  \|_{L^{2}_{x,\xi}}^{2} \right)
\\
&\leq   \frac{C\delta}{r - 2\ve}
\sum_{k=0}^{1} \left\| e^{\beta x/2} |\xi_{1}|^{1/2} \nabla_{x,\xi}^{k} f \right\|_{L^{2}_{x,\xi}}^{2}.
\end{align*}
Substituting these estimates into \eqref{bes4} and recalling \eqref{aes0}, we arrive at
\begin{align}
&\dfrac{d}{dt} \left(\Vert e^{\beta x/2}   f (t) \Vert^2_{L^2_{x,\xi}} 
+\Vert e^{\beta x/2}   \pd_{x} f (t) \Vert^2_{L^2_{x,\xi}} 
+ \frac{1}{\theta_{\infty}}\Vert e^{\beta x/2}   n (t) \Vert^2_{L^2_{x}} \right)
\notag \\
&\quad +\left.   \left(\int_{\mathbb{R}^3_-} \vert \xi_1\vert \vert f \vert^2 
+\vert \xi_1\vert \vert \pd_{x} f \vert^2  d\xi 
+\frac{|u_{\infty}|}{\theta_{\infty}}|n|^{2}\right)\right|_{x=0} +\beta{\cal D}
\notag \\
&\leq C \delta^{1/2} 
\sum_{k=0}^{1} \left\| e^{\beta x/2} |\xi_{1}|^{1/2} \nabla_{x,\xi}^{k} f \right\|_{L^{2}_{x,\xi}}^{2},
\label{bes5}
\end{align}
where 
\begin{align*}
{\cal D}&:= \left\| e^{\beta x/2} |\xi_{1}|^{1/2}  f \right\|_{L^{2}_{x,\xi}}^{2}
+\left\| e^{\beta x/2} |\xi_{1}|^{1/2}  \pd_{x} f \right\|_{L^{2}_{x,\xi}}^{2}
+ \frac{|u_{\infty}|}{\theta_{\infty}} \| e^{\beta x/2} n \|_{L^{2}_{x}}^{2}
\\
& \quad - 2\left( \frac{\rho_{\infty}^{1/2}}{\theta_{\infty}^{1/2}}(1+\eta(\beta))^{1/2} + \frac{\mu_{\infty}}{\beta}  \right) \frac{(1+\eta(\beta))^{1/2}}{(r - 2\ve)^{1/2}} \|e^{\beta x/2} n \|_{L^{2}_{x}} \left\| e^{\beta x/2} |\xi_{1}|^{1/2}  f \right\|_{L^{2}_{x,\xi}}
\\
& \quad - \frac{2 \{1+\beta^{2}(1+\eta(\beta))^{2} \}^{1/2}\mu_{\infty}}{\beta(r - 2\ve)^{1/2}}\| e^{\beta x/2} n \|_{L^{2}_{x}} \left\| e^{\beta x/2} |\xi_{1}|^{1/2}  \pd_{x} f \right\|_{L^{2}_{x,\xi}}.
\end{align*}
Let us write ${\cal D}$ in the quadratic form ${\cal D} = {}^{t}w D w$,
where
\begin{align*}
w&:=\begin{bmatrix}
\left\| e^{\beta x/2} |\xi_{1}|^{1/2} f \right\|_{L^{2}_{x,\xi}}
\\
\left\| e^{\beta x/2} |\xi_{1}|^{1/2}  \pd_{x} f \right\|_{L^{2}_{x,\xi}}
\\ 
\|e^{\beta x/2} n \|_{L^{2}_{x}}
\end{bmatrix},
\quad 
D:=\begin{bmatrix}
1 & 0 &  -d_{1}
\\
0 & 1 &  -d_{2}
\\
-d_{1} & -d_{2} & \frac{|u_{\infty}|}{\theta_{\infty}}
\end{bmatrix},
\\
d_{1}&:=\left( \frac{\rho_{\infty}^{1/2}}{\theta_{\infty}^{1/2}}(1+\eta(\beta))^{1/2} + \frac{\mu_{\infty}}{\beta}  \right) \frac{(1+\eta(\beta))^{1/2}}{(r - 2\ve)^{1/2}},
\quad
d_{2}:= \frac{\{1+\beta^{2}(1+\eta(\beta))^{2} \}^{1/2}\mu_{\infty}}{\beta(r - 2\ve)^{1/2}}.
\end{align*}
The symmetric matrix $D$ is positive definite if and only if the assumption \eqref{asp1} holds.
Therefore, multiplying \eqref{bes5} by $e^{\gamma t}$ and integrating it over $[0,t]$,
we conclude the desired estimate \eqref{aes2}. 
\end{proof}

We also need to derive the estimate of $\nabla_{\xi}f$.

\begin{lem}\label{alem3}
Suppose that the same assumption as in Proposition \ref{apriori1} and also \eqref{aes0} hold. 
There exists a constant $\delta_{3}>0$ such that if $\Phi_{b}+N_{0,0}(T)  \leq \delta_{3}$, then the following holds:
\begin{align}\label{aes3}
& e^{\gamma t} \Vert e^{\beta x/2} \nabla_{\xi} f (t) \Vert^2_{L^2_{x,\xi}} 
+ \frac{\beta}{2}  \int_{0}^{t} e^{\gamma \tau} \left\| e^{\beta x/2} |\xi_{1}|^{1/2} \nabla_{\xi} f (\tau) \right\|_{L^{2}_{x,\xi}}^{2} d\tau
\notag \\
&\leq  \Vert e^{\beta x/2} \nabla_{\xi} f_{0} \Vert^2_{L^2_{x,\xi}} 
+\gamma \int_{0}^{t} e^{\gamma \tau} \Vert e^{\beta x/2} \nabla_{\xi} f (\tau) \Vert^2_{L^2_{x,\xi}} d\tau
\notag\\
& \quad + C_{2}  \int_{0}^{t} e^{\gamma \tau} \left\Vert e^{\beta x/2} |\xi_{1}|^{1/2} f (\tau)\right\Vert^2_{L^2_{x,\xi}} 
\!\!+ e^{\gamma \tau} \left\Vert  e^{\beta x/2} |\xi_{1}|^{1/2}  \pd_{x} f (\tau) \right\Vert^2_{L^2_{x,\xi}} d\tau
\end{align}
for any $\gamma>0$, where $C_{2}$ is a positive constant independent of $\gamma$, $\Phi_b$, $f_{0}$, $t$, and $T$.
\end{lem}
\begin{proof}
Differentiate \eqref{ree1} with respect to $\xi_{i}$,
multiply $2 e^{\beta x} \pd_{\xi_{i}} f$ to the resultant equation, and sum them up for $i=1,2,3$.
Then we have
\begin{align*}
&\partial_{t} (e^{\beta x}\vert  \nabla_{\xi} f \vert^2 )
+ \partial_{x} (\xi_{1} e^{\beta x}  \vert \nabla_{\xi} f \vert^2)
-\beta \xi_1 e^{\beta x}  \vert \nabla_{\xi}  f\vert^2
+\partial_{\xi_{1}} \{e^{\beta x} (\pd_{x}\Phi^{s} +\pd_{x} \phi) \vert \nabla_{\xi} f \vert^2\}
={\cal R}_{3},
\end{align*}
where 
\begin{align*}
{\cal R}_{3}&:=  - 2e^{\beta x}  \pd_{x}f  \pd_{\xi_{1}} f 
- 2e^{\beta x}  \pd_{x}\phi  \nabla_{\xi}\left(\frac{\pd_{\xi_{1}}F^{s}}{M_{\infty}^{1/2}}\right) \cdot  \nabla_{\xi} f 
\\
& \quad + e^{\beta x} (\partial_{x} \Phi^{s} + \partial_{x}\phi)   \nabla_{\xi} \left( \frac{\xi_{1}-u_{\infty}}{\theta_{\infty}} f\right) \cdot \nabla_{\xi} f.
\end{align*}
Integrate this over $\mathbb R_{+} \times \mathbb{R}^3$ 
and use \eqref{rebc1}--\eqref{rebc4} and \eqref{aes0}  to obtain
\begin{align}
&\dfrac{d}{dt} \Vert e^{\beta x/2}   \nabla_{\xi} f (t) \Vert^2_{L^2_{x,\xi}} 
+\left.  \int_{\mathbb{R}^3_-} \vert \xi_1\vert \vert  \nabla_{\xi} f \vert^2 d\xi \right|_{x=0}
+ \beta  \int_{\mathbb{R}_{+}} \! \int_{\mathbb{R}^3_{-}}  e^{\beta x} |\xi_1| | \nabla_{\xi} f|^2  dxd\xi 
\notag \\
&=\int_{\mathbb{R}_{+}} \! \int_{\mathbb{R}^3} {\cal R}_{3}  dxd\xi.
\label{hes1}
\end{align}

Set $ \delta:= \Phi_b + N_{0,0}(T)$. It is easy to see from \eqref{nm3}, \eqref{Fs1*}, \eqref{ellineq0}, and \eqref{ellineq3}  with $k=0$ that
\begin{align*}
&\left|\int_{\mathbb{R}_{+}} \! \int_{\mathbb{R}^3} {\cal R}_{3}  dxd\xi\right|
\\
&\leq  ( \lambda +  C\delta ) \left\| e^{\beta x/2} (1+|\xi_{1}|^{1/2})  \nabla_{\xi} f  \right\|_{L^{2}_{x,\xi}}^{2} 
 + C\lambda^{-1} \left( \Vert e^{\beta x/2} f \Vert^2_{L^2_{x,\xi}} + \Vert e^{\beta x/2}   \pd_{x} f  \Vert^2_{L^2_{x,\xi}} \right)
\\
&\leq \frac{ \lambda +  C\delta }{r - 2\ve} \left\| e^{\beta x/2} |\xi_{1}|^{1/2} \nabla_{\xi} f  \right\|_{L^{2}_{x,\xi}}^{2} 
\!\!\! + \frac{C\lambda^{-1}}{r - 2\ve} \left( \!
\left\Vert e^{\beta x/2}  |\xi_{1}|^{1/2} f \right\Vert^2_{L^2_{x,\xi}} 
\!\!\! +  \left\Vert  e^{\beta x/2} |\xi_{1}|^{1/2}  \pd_{x} f \right\Vert^2_{L^2_{x,\xi}} \! \right)\!,
\end{align*}
where $\lambda>1$ is a constant to be determined later, and we have also used \eqref{aes0} and the fact that $|\xi_{1}| > r-2\ve$ holds on the support of $f$ in deriving  the last inequality. 
Substitute this into \eqref{hes1}, multiply the resultant inequality by $e^{\gamma t}$, and integrate it over $[0,t]$.
Then taking $\lambda$ and $\delta$ small enough, we arrive at \eqref{aes3}. The proof is complete.
\end{proof}

We are now in a position to prove Proposition \ref{apriori1}.

\begin{proof}[Proof of Proposition \ref{apriori1}]
First we set 
\begin{align*}
\gamma&:=\frac{1}{2}\min\{1, r-2\ve\} \min\left\{ \frac{c_{1}}{4}, \ c_{1}\theta_{\infty}, \frac{\beta}{4} \right\}>0,
\\
\delta_{0}&:=\min\left\{ \delta_{1}(\gamma), \ \delta_{2}, \ \delta_{3},  \left(\frac{c_{1}}{4C_{1}}\right)^{2}, \left(\frac{c_{1}\beta}{8C_{2}}\right)^{2}  \right\}>0,
\end{align*}
where $\delta_{1}=\delta_{1}(\gamma)$, $\delta_{2}$, $\delta_{3}$, $c_{1}$, $C_{1}$, and $C_{2}$ are the same constants being in Lemmas \ref{alem1}--\ref{alem3}.
Suppose that $\Phi_{b}+N_{\beta,\gamma}(T)  \leq \delta_{0}$ and recall \eqref{N2}.
Applying Lemma \ref{alem1} with the above $\gamma>0$,  we have \eqref{aes0}.
Now we can use Lemmas \ref{alem2} and \ref{alem3}.
Multiplying $\eqref{aes3}$ by  $c_{1}/(2C_{2})$ and then adding it to \eqref{aes2}, we obtain
\begin{align*}
&e^{\gamma t} E(t) +  \frac{c_{1}}{4}  \int_{0}^{t} e^{\gamma \tau} \left\Vert e^{\beta x/2} |\xi_{1}|^{1/2} f (\tau)\right\Vert^2_{L^2_{x,\xi}} 
\!\!+ e^{\gamma \tau} \left\Vert  e^{\beta x/2} |\xi_{1}|^{1/2}  \pd_{x} f (\tau) \right\Vert^2_{L^2_{x,\xi}} d\tau
\notag \\
&\quad + c_{1} \int_{0}^{t}   e^{\gamma \tau} \|e^{\beta x/2} n (\tau) \|_{L^{2}_{x}}^{2} d\tau
+ \frac{c_{1}\beta}{8C_{2}}  \int_{0}^{t} e^{\gamma \tau} \left\| e^{\beta x/2} |\xi_{1}|^{1/2}  \nabla_{\xi} f  (\tau)\right\|_{L^{2}_{x,\xi}}^{2} d\tau
\notag \\
& \leq E(0) + \gamma \int_{0}^{t} e^{\gamma t} E(\tau) d\tau,
\end{align*}
where 
\begin{gather*}
E(t):=\Vert e^{\beta x/2} f (t) \Vert^2_{L^2_{x,\xi}} 
+\Vert e^{\beta x/2}   \pd_{x} f (t) \Vert^2_{L^2_{x,\xi}} 
+ \frac{1}{\theta_{\infty}}\Vert e^{\beta x/2}   n (t) \Vert^2_{L^2_{x}} 
+ \frac{c_{1}}{2C_{2}} \Vert e^{\beta x/2} \nabla_{\xi} f(t) \Vert^2_{L^2_{x,\xi}}.
\end{gather*}
Using \eqref{aes0} and the fact that $|\xi_{1}| > r-2\ve$ holds on the support of $f$, we see that
\begin{align*}
e^{\gamma t} E(t) 
+ \left( \min\{1, r-2\ve\} \min\left\{ \frac{c_{1}}{4}, \ c_{1}\theta_{\infty}, \frac{\beta}{4} \right\} -\gamma \right)\int_{0}^{t} e^{\gamma t} E(\tau) d\tau
\leq E(0),
\end{align*}
which gives \eqref{aes1}. 
Finally, \eqref{ellaes1} immediately follows from \eqref{nm0}, \eqref{nm3}, and  \eqref{ellineq3}.
The proof is complete.
\end{proof}

\section{A Priori Estimate on the Instability}\label{AE2}

We show Proposition \ref{apriori2} which leads to the instability theorem.

\begin{proof}[Proof of Proposition \ref{apriori2}]

We first show \eqref{aesin0}.
Let $\delta_{0}<\min\{R_{1}^{2}/8,r^{2}/8\}$ so that $\Phi_{b}=\delta<\min\{R_{1}^{2}/8,r^{2}/8\}$,
and set
\begin{gather*}
{{\Omega}}:=\left\{(x,\xi) \in {\mathbb R}_{+}\times{\mathbb R}^{3} \, \left| \, {\Xi}^{-}(x) < \xi_{1} < {\Xi}^{+}(x) \right.\right\},
\\
{\Xi}^{-}(x):= -\sqrt{2\Phi^{s}(x) + \frac{r^{2}}{2}\left(1-\frac{1}{2}e^{-\beta x/2}\right)}, \quad
{\Xi}^{+}(x):= \sqrt{2\Phi^{s}(x) + \frac{R_{1}^{2}}{4}(1+e^{-\beta x/2})},
\\
{\Gamma}^{\pm}:=\left\{(x,\xi) \in {\mathbb R}_{+} \times{\mathbb R}^{3} \, | \, \xi_{1}={\Xi}^{\pm}(x) \right\}, \quad
{N}^{\pm}(x):=\sqrt{1+|\partial_{x} {\Xi}^{\pm}(x)|^{2}}
\end{gather*}
for $r$, $R_{1}$, and $\beta$ being in \eqref{cutoff0}, \eqref{g0}, and the assumption of Proposition \ref{apriori2}, respectively.
We see from \eqref{decay0} and \eqref{g0} that 
\begin{gather}\label{instasupp2}
\supp g_{0} \cap \Omega =  \emptyset, \quad
{\Xi}^{-}(x) \leq -r/2, \quad 
{\Xi}^{+}(x) \geq R_{1}/2.
\end{gather}
Let us show that $f=0$ on $[0,T] \times \Omega$, which leads to \eqref{aesin0}.
It is clear that $F^{s}=f_{0}=\delta g_{0}=0$ on $\Omega$ owing to \eqref{r2} and \eqref{instasupp2}.
Integrating \eqref{suppes1} over $[0,t] \times \Omega$, and using the divergence theorem, we see that
\begin{align}
&\int\!\!\int_{{\Omega}} |f|^{2}(t) dxd\xi 
+ \left. \int_{0}^{t}\!\!\int_{(\Xi^{-}(0),\Xi^{+}(0)) \times \mathbb R^{2}}  \vert \xi_1\vert \vert f \vert^2  \chi(-\xi_{1}) d\tau d\xi \right|_{x=0}
\notag \\
&\quad +\int_{0}^{t} \!\! \int_{{\Gamma}^{+}} \frac{1}{N^{+}} \left\{(\pd_{x}\Phi^{s} +\pd_{x} \phi)   - \Xi^{+}\partial_{x} \Xi^{+} \right\} |f|^{2}  d\tau dS  
\notag \\
&\quad + \int_{0}^{t} \!\! \int_{{\Gamma}^{-}} \frac{1}{N^{-}} \left\{\Xi^{-} \partial_{x} \Xi^{-} - (\pd_{x}\Phi^{s} +\pd_{x} \phi)  \right\} |f|^{2}  d\tau dS 
\notag \\
& = \int_{0}^{t}\!\! \int\!\!\int_{{\Omega}} \frac{\xi_{1}-u_{\infty}}{\theta_{\infty}} (\partial_{x} \Phi^{s} + \partial_{x}\phi)   |f|^{2}d\tau dxd\xi
\notag \\
& \leq C \int_{0}^{t}\!\! \int\!\!\int_{{\Omega}} |f|^{2}(\tau) d\tau dxd\xi,
\label{instasupp3}
\end{align}
where we have used $|\partial_{x}\phi (t,x)| \leq C$, which follows from $\phi \in {\cal Y}_{\rm e}(T,\beta)$,  in deriving the inequality.
It is obvious that the second term on the left hand side is nonnegative.
We also see by using \eqref{nm3}, \eqref{ellineq0}, \eqref{N2}, $\Phi_{b}=\delta \leq \delta_{0}$, and $N_{0,\beta}(T) \leq \ve_{0}$ that the integrants of the third and fourth terms on the left hand side are nonnegative
for suitably small $\ve_{0}$ as follows:
\begin{align*}
&\frac{1}{N^{+}}  \left\{ (\pd_{x}\Phi^{s} +\pd_{x} \phi)  -\Xi^{+}\partial_{x} \Xi^{+}  \right\}  |f|^{2}
\\
&=\frac{1}{N^{+}} \left\{\pd_{x} \phi + \frac{\beta R_{1}^{2}}{16} e^{-\beta x/2}  \right\} |f|^{2}
\geq  \frac{1}{N^{+}} \left\{\frac{\beta R_{1}^{2}}{16} e^{-\beta x/2} -\ve_{0} Ce^{-\beta x/2}  \right\} |f|^{2}
\geq 0
\end{align*}
and
\begin{align*}
&\frac{1}{N^{-}}  \left\{\Xi^{-} \partial_{x} \Xi^{-} - (\pd_{x}\Phi^{s} +\pd_{x} \phi) \right\}  |f|^{2}
\\
&=\frac{1}{N^{-}} \left\{\frac{\beta r^{2}}{16} e^{-\beta x/2}  -\pd_{x} \phi  \right\} |f|^{2}
 \geq  \frac{1}{N^{-}} \left\{\frac{\beta r^{2}}{16} e^{-\beta x/2} -\ve_{0} Ce^{-\beta x/2} \right\} |f|^{2} 
\geq 0.
\end{align*}
Now applying Gronwall's inequality to \eqref{instasupp3}, we conclude that $f=0$ on $[0,T] \times \Omega$. 
Thus \eqref{aesin0} holds.

Next let us prove \eqref{aesin2}. Let $L >2 R_{2}$ and $\delta_{0}< R_{2}^{2}$, and set
\begin{gather*}
{{\Omega_{L}}}:=\left\{(x,\xi) \in {\mathbb R}_{+}\times{\mathbb R}^{3} \, \left| \, {\Xi}^{-}_{L}(x) < \xi_{1} < {\Xi}^{+}_{L}(x) \right.\right\},
\\
{\Xi}^{-}_{L}(x):= \sqrt{2\Phi^{s}(x) + 2R_{2}^{2}\left(1-\frac{1}{2}e^{-\beta x/2}\right)}, \quad
{\Xi}^{+}_{L}(x):= \sqrt{2\Phi^{s}(x) + L^{2}(1+e^{-\beta x/2})}
\end{gather*}
for $R_{2}$ and $\beta$ being in \eqref{g0} and the assumption of Proposition \ref{apriori2}, respectively.
We see from \eqref{decay0} and \eqref{g0} that 
\begin{gather*}
\supp g_{0} \cap \Omega_{L} =  \emptyset, \quad  {\Xi}^{-}_{L}(x) \leq 2R_{2}, \quad {\Xi}^{+}_{L}(x) \geq L.
\end{gather*}
In order to show \eqref{aesin2}, it is sufficient to show that $f=0$ on $[0,T] \times \Omega_{L}$, 
since $L$ is an arbitrary constant with $L >2 R_{2}$.
Similarly as the proof of \eqref{aesin0}, we integrate \eqref{suppes1} over $[0,t] \times \Omega_{L}$, use the divergence theorem and \eqref{r2}, and choose small $\ve_{0}$ independent of $L$.
Then, applying Gronwall's inequality, we can see that $f=0$ on $[0,T] \times \Omega_{L}$.
Thus \eqref{aesin2} holds.

We complete the proof by showing \eqref{aesin1}.
Let $\delta_{0}<R_{1}^{2}/8$ so that $\Phi_{b}=\delta<R_{1}^{2}/8$, and set
\begin{gather*}
{\tilde{\Omega}}:=\left\{(x,\xi) \in {\mathbb R}_{+}\times{\mathbb R}^{3} \, \left| \,   \tilde{\Xi}^{-}(x) < \xi_{1} < \tilde{\Xi}^{+}(x) \right.\right\},
\\
 \tilde{\Xi}^{-}(x):= \sqrt{2\Phi^{s}(x) + \frac{R_{1}^{2}}{4}(1+e^{-\beta x/2})}, \quad
\tilde{\Xi}^{+}(x):= \sqrt{2\Phi^{s}(x) + R_{2}^{2}(4-e^{-\beta x/2})},
 \\
 \tilde{\Gamma}^{\pm}:=\left\{(x,\xi) \in {\mathbb R}_{+} \times{\mathbb R}^{3} \, | \, \xi_{1}=\tilde{{\Xi}}^{\pm}(x) \right\},   \quad
\tilde{N}^{\pm}(x):=\sqrt{1+|\partial_{x} \tilde{\Xi}^{\pm}(x)|^{2}}.
\end{gather*}
We see from \eqref{decay0} and \eqref{g0} that 
\begin{gather}\label{instasupp1}
\supp g_{0} \cap {\tilde{\Omega}} = \supp g_{0}  \neq \emptyset.
\end{gather}
It is clear that $F^{s}=0$ in $\tilde{\Omega}$ owing to \eqref{r2}.
Let $\gamma$ be a positive constant to be determined later.
Multiplying \eqref{suppes1} by $e^{-\gamma t}e^{\beta x}$, integrating it over $[0,t] \times {\tilde{\Omega}}$, and using the divergence theorem, we see that
\begin{align}
&e^{-\gamma t} \int\!\!\int_{\tilde{\Omega}} e^{\beta x}|f|^{2}(t) dxd\xi 
\notag \\
& = \int\!\!\int_{\tilde{\Omega}} e^{\beta x}|f_{0}|^{2} dxd\xi 
+ \beta \int_{0}^{t}\!\!\int\!\!\int_{\tilde{\Omega}} e^{-\gamma \tau} e^{\beta x} \vert \xi_1\vert  \vert f \vert^2 d\tau dx d\xi 
\notag \\
& \quad - \gamma  \int_{0}^{t}\!\!\int\!\!\int_{\tilde{\Omega}}e^{-\gamma \tau} e^{\beta x} \vert f \vert^2 d\tau dx d\xi 
\notag \\
&\quad +  \int_{0}^{t} \!\! \int_{\tilde{\Gamma}^{+}} \frac{e^{-\gamma \tau} e^{\beta x}}{\tilde{N}^{+}}  \left\{\tilde{\Xi}^{+} \partial_{x} \tilde{\Xi}^{+} - (\pd_{x}\Phi^{s} +\pd_{x} \phi)  \right\} |f|^{2}  d\tau dS  
\notag \\
&\quad +  \int_{0}^{t}\!\!\int_{\tilde{\Gamma}^{-}}  \frac{e^{-\gamma \tau} e^{\beta x}}{\tilde{N}^{-}}  \left\{(\pd_{x}\Phi^{s} +\pd_{x} \phi) - \tilde{\Xi}^{-}  \partial_{x} \tilde{\Xi}^{-} \right\}  |f|^{2}  d\tau dS  
\notag \\
& \quad +  \int_{0}^{t}\!\! \int\!\!\int_{\tilde{\Omega}} e^{-\gamma \tau}e^{\beta x} \frac{\xi_{1}-u_{\infty}}{\theta_{\infty}}(\partial_{x} \Phi^{s} + \partial_{x}\phi) |f|^{2}d\tau dxd\xi.
\label{instaes1}
\end{align}
Let us estimate the terms on the right hand side from below. 
By the fact that $|\xi_{1}| \geq R_{1}/2>0 $ on $\tilde{\Omega}$,
the second term can be estimated as
\begin{align*}
\beta  \int_{0}^{t}\!\!\int\!\!\int_{\tilde{\Omega}} e^{-\gamma \tau} e^{\beta x} \vert \xi_1\vert  \vert f \vert^2 d\tau dx d\xi 
\geq  \frac{\beta R_{1}}{2}  \int_{0}^{t}\!\!\int\!\!\int_{\tilde{\Omega}} e^{-\gamma \tau} e^{\beta x} \vert f \vert^2 d\tau dx d\xi.
\end{align*}
Using \eqref{nm3}, \eqref{ellineq0},  $\Phi_{b}=\delta \leq \delta_{0}$, and $N_{0,\beta}(T) \leq \ve_{0}$ as well as taking $\delta_{0}$ and $\ve_{0}$ small enough,
we see that the integrants of the fourth and fifth terms on the right hand side are nonnegative as follows:
\begin{align*}
& \frac{e^{-\gamma \tau} e^{\beta x}}{\tilde{N}^{+}}  \left\{\tilde{\Xi}^{+} \partial_{x} \tilde{\Xi}^{+} - (\pd_{x}\Phi^{s} +\pd_{x} \phi)  \right\}  |f|^{2} 
\\
&= \frac{e^{-\gamma \tau} e^{\beta x}}{\tilde{N}^{+}}  \left\{\frac{\beta R_{2}^{2}}{4} e^{-\beta x/2} -\pd_{x} \phi  \right\} |f|^{2} 
 \geq \frac{e^{-\gamma \tau} e^{\beta x}}{\tilde{N}^{+}}  \left\{\frac{\beta R_{2}^{2}}{4} e^{-\beta x/2} -\ve_{0} Ce^{-\beta x/2}  \right\} |f|^{2} 
\geq 0
\end{align*}
and
\begin{align*}
& \frac{e^{-\gamma \tau} e^{\beta x}}{\tilde{N}^{-}}   \left\{(\pd_{x}\Phi^{s} +\pd_{x} \phi) - \tilde{\Xi}^{-} \partial_{x}  \tilde{\Xi}^{-} \right\}  |f|^{2}
\\
&= \frac{e^{-\gamma \tau} e^{\beta x}}{\tilde{N}^{-}}  \left\{\pd_{x} \phi + \frac{\beta R_{1}^{2}}{16} e^{-\beta x/2}   \right\} |f|^{2}
\geq  \frac{e^{-\gamma \tau} e^{\beta x}}{\tilde{N}^{-}}  \left\{\frac{\beta R_{1}^{2}}{16} e^{-\beta x/2} -\ve_{0} Ce^{-\beta x/2} \right\} |f|^{2}
\geq 0.
\end{align*}
Furthermore, we observe from \eqref{decay0} and $\Phi_{b}=\delta \leq \delta_{0}$ that
\begin{align*}
&  \int_{0}^{t}\!\! \int\!\!\int_{\tilde{\Omega}} e^{-\gamma \tau}e^{\beta x} \frac{\xi_{1}-u_{\infty}}{\theta_{\infty}}(\partial_{x} \Phi^{s} + \partial_{x}\phi) |f|^{2}d\tau dxd\xi
\\
&\geq -  (\delta_{0}+\ve_{0}) C \int_{0}^{t}\!\! \int\!\!\int_{\tilde{\Omega}} e^{-\gamma \tau} e^{\beta x}|f|^{2}d\tau dxd\xi.
\end{align*}
Substituting these into \eqref{instaes1} and taking $\delta_{0}$, $\ve_{0}$, and $\gamma$ small enough,
we arrive at 
\begin{align}\label{instaes5}
\int\!\!\int_{\tilde{\Omega}} e^{\beta x}  |f|^{2}(t) dxd\xi 
\geq e^{\gamma t}\int\!\!\int_{\tilde{\Omega}} e^{\beta x}  |f_{0}|^{2} dxd\xi \quad \text{for $t \in [0,T]$}.
\end{align}
The desired estimate \eqref{aesin1} immediately follows from \eqref{instasupp1}, \eqref{instaes5}, 
and $f_{0} = \delta g_{0}$.
The proof is complete.
\end{proof}

\medskip

\noindent
{\bf Acknowledgment.} 
M. Suzuki was supported by JSPS KAKENHI Grant Numbers 18K03 364 and 21K03308.

\begin{appendix}

\section{A Selection of Constants with \eqref{asp1}}\label{SA}
We find constants $\beta \in (0,1)$, $\rho_{\infty}>0$, $u_{\infty} < 0$,  $\theta_{\infty}>0$, and $\mu_{\infty}=\mu_{\infty}(\rho_{\infty},u_{\infty},\theta_{\infty})>0$ so that \eqref{asp1} holds, i.e. 
\begin{gather*}
\frac{|u_{\infty}|}{\theta_{\infty}}- 
\left( \frac{\rho_{\infty}^{1/2}}{\theta_{\infty}^{1/2}}(1+\eta(\beta))^{1/2} + \frac{\mu_{\infty}}{\beta}  \right)^{2} \frac{1+\eta(\beta)}{r - 2\ve}
-  \frac{\{1+\beta^{2}(1+\eta(\beta))^{2} \}\mu_{\infty}^{2}}{\beta^{2}{(r - 2\ve) }}>0
\end{gather*}
by assuming either the condition (i) or (ii) in Theorem \ref{mainThm}.
Note that 
\begin{gather}\label{etacon1}
\eta(\beta) \to 0 \ \ \text{as $\beta \to 0$.} 
\end{gather}

First suppose that the condition (i) holds, i.e. $r-2\ve\geq1$, $|u_{\infty}|>1$, and $\theta_{\infty} \ll 1$.
Let us fix $u_{\infty}$.
Owing to the assumption \eqref{netrual1}, the constant $\rho_{\infty}$ is determined uniquely for each $\theta_{\infty}$, and  so does $\mu_{\infty}$. Therefore, it is essential to find $\beta$ and $\theta_{\infty}$.
From \eqref{netrual1}, it is seen that
\begin{gather}\label{rhocon1}
\rho_{\infty} =\rho_{\infty}(\theta_{\infty}) \to 1 \ \  \text{as $\theta_{\infty} \to 0$.} 
\end{gather}
This gives a bound $\rho_{\infty} \in (1/2,2)$. 
By the direct computation with the aid of this bound, we see that 
\begin{gather}\label{mucon1}
\mu_{\infty}=\mu_{\infty}(\theta_{\infty}) \to 0 \ \  \text{as $\theta_{\infty} \to 0$.} 
\end{gather}
Using \eqref{etacon1}, \eqref{rhocon1}, and  the conditions $r-2\ve\geq1$ and $|u_{\infty}|>1$, 
we deduce the following for some $\beta=\beta(r,\ve,u_{\infty}) \in (0,1)$ and any $\theta_{\infty} \ll 1$:
\begin{gather*}
\frac{|u_{\infty}|}{\theta_{\infty}}- 
\left( \frac{\rho_{\infty}^{1/2}}{\theta_{\infty}^{1/2}}(1+\eta(\beta))^{1/2}  \right)^{2} \frac{1+\eta(\beta)}{r - 2\ve}>0.
\end{gather*}
Consequently, using \eqref{mucon1} and taking $\theta_{\infty} \ll 1$ according to $\beta$,
we have the desired constants with \eqref{asp1} for any $\theta_{\infty} \ll 1$.
In addition, it is worth pointing out that the constants allow that the condition \eqref{Bohm2} holds.
Indeed, $F^{\infty}=M_{\infty}\psi$ with \eqref{netrual1} converges to a delta function as $\theta_{\infty} \to 0$.
This fact and $|u_{\infty}|>1$ ensures \eqref{Bohm2} for $\theta_{\infty} \ll 1$.

Next suppose that the condition (ii) holds, i.e. $|u_{\infty}| \gg 1$. Let us set $\beta=1/2$ and fix $\theta_{\infty}$.
Due to \eqref{netrual1}, the constants $\rho_{\infty}$ and $\mu_{\infty}$ are
determined uniquely for each $u_{\infty}$. It is essential to find $u_{\infty}$.
It is seen by the direct computation that 
\begin{gather*}
\rho_{\infty} = \rho_{\infty} (u_{\infty}) \to 1, \ \ \mu_{\infty}=\mu_{\infty} (u_{\infty}) \to 0 \ \  \text{as $u_{\infty} \to \infty$}. 
\end{gather*}
This gives the bounds of $\rho_{\infty}$ and $\mu_{\infty}$, 
and hence we have the desired constants with \eqref{asp1} for any $|u_{\infty}| \gg 1$.
It is also clear that \eqref{Bohm2} holds for $|u_{\infty}| \gg 1$.

\section{The Solvability of Linearized Problem}\label{SB}

This section is devoted to the proof of Lemma \ref{sovl1} on the solvability of the linearized problem \eqref{l1}, i.e.
\begin{gather*}
{\cal L} f :=\pd_{t} f + \xi_{1}\partial_{x} f + (\partial_{x} \Phi^{s} + \partial_{x}\hat{\phi}) \partial_{\xi_{1}}f 
-  \frac{\xi_{1}-u_{\infty}}{2\theta_{\infty}}  (\partial_{x} \Phi^{s} + \partial_{x}\hat{\phi})   f = \hat{F},
\\
f(0,x,\xi) = f_{0}(x,\xi), \quad
f(t,0,\xi) = 0, \ \xi_{1}>0, \quad
\lim_{x \to\infty} f(t,x,\xi) = 0.
\end{gather*}

The problems \eqref{l1} has a characteristic boundary condition at $(x,\xi_{1},\xi')=(0,0,\xi')$, 
where a loss of regularity of solutions may occur. 
If nothing is done, we need to analyze carefully the regularity of solutions.
To avoid this analysis, we reduce the initial--boundary value problem \eqref{l1} 
to an initial value problem as follows. 
By using the assumption in Lemma \ref{sovl1}, we can have extensions 
$\tilde{\Phi}^{s}=\tilde{\Phi}^{s}(x)$ and $\tilde{\phi}=\tilde{\phi}(t,x)$
of the functions $\Phi^{s}$ and $\hat{\phi}$ such that
\begin{subequations}\label{extension1}
\begin{align}
\tilde{\Phi}^s(x)=\Phi(x) \quad & \text{if $x>0$}, 
\\
\tilde{\phi}(t,x)=\hat{\phi}(t,x) \quad &\text{if $(t,x)\in [0,\hat{T}] \times\R_+$}
\end{align}
and
\begin{gather}
\tilde{\Phi}^{s} \in {\cal B}^2(\mathbb R), \quad
\sum_{k=0}^2\sup_{x\in\overline{\mathbb R_{+}}}|\D_x^k\tilde{\Phi}^s(x)| \leq \tilde{C}\Phi_b,
\\
\tilde{\phi} \in C([0,\hat{T}];{L}^{2}(\mathbb R)) \cap L^{\infty}(0,\hat{T};{H}^{3}(\mathbb R)), \quad
\sup_{t \in [0,\hat{T}]} \| \tilde{\phi} \|_{{H}^{3}(\mathbb R)} \leq \tilde{C}\sup_{t \in [0,\hat{T}]} \| \hat{\phi} \|_{H^{3}(\mathbb R_{+})},
\end{gather}
\end{subequations}
where $\tilde{C}>0$ is a constant.
Similarly, we can also find some extensions $\tilde{F}$ and $\tilde{f}_{0}$ of 
the functions $\hat{F}$ and $f_{0}$ such that the domains of $\tilde{F}$ and $\tilde{f}_{0}$ are $[0,\hat{T}] \times\R\times\R^3$ and $\R\times\R^3$, respectively, and further
\begin{gather*}
\begin{aligned}
\tilde{F}(t,x,\xi) = \hat{F}(t,x,\xi) \quad &\text{if $(t,x,\xi)\in [0,\hat{T}] \times\R_+\times\R^3$}, 
\\
\tilde{f}_0(x,\xi) = f_0(x,\xi) \quad & \text{if $(x,\xi)\in\R_+\times\R^3$}.
\end{aligned}
\end{gather*}
We will mention some more properties of $\tilde{F}$ and $\tilde{f}_{0}$ in subsection \ref{SSB_Proof}.
Let us consider the following initial value problem:
\begin{gather*}
\pd_{t} f 
+ \xi_{1}\partial_{x} f 
+ \! (\partial_{x} \tilde{\Phi}^{s}+\partial_{x}\tilde{\phi}) \partial_{\xi_{1}}f 
- \frac{\xi_{1}-u_{\infty}}{2\theta_{\infty}}
(\partial_{x} \tilde{\Phi}^{s} + \partial_{x}\tilde{\phi}) f \!=\! \tilde{F}, 
\ \ \! (t, x, \xi) \! \in [0,\hat{T}] \! \times \R \times \R^3\!\!,
\\
f(0,x,\xi) = \tilde{f}_{0}(x,\xi), \ \ (x, \xi) \in  \R \times \R^3.
\end{gather*}
If we can find a solution $f$ of this initial value problem such that
\begin{gather*}
f(t, x, \xi)=0, \quad (t, x,\xi) \in [0,T_{0}]  \times  {\mathbb R}_{-}  \times  {\mathbb R}^{3}_{+},
\end{gather*}
which implies that $f(t) \in H^{1}_{0,\Gamma}$,
then the solution $f$ (restricted its domain into $ [0,T_{0}]  \times \R_{+} \times \R^3$) solves the problem \eqref{l1} as well.


However, there is still another difficulty that the coefficient $\frac{\xi_{1}-u_{\infty}}{2\theta_{\infty}}$ 
of the equation of the above initial value problem is unbounded.
To resolve this issue, we make use of another initial value problem modified by replacing 
the coefficient $\xi_{1}$ by the following smooth function $\Xi_{1}=\Xi_{1}(\xi_{1})$:
\begin{gather}\label{Xi1}
\Xi_1(\xi_1)=\Xi_1(\xi_1;K):=KX(K^{-1}\xi_1),
\end{gather}
where $K>0$ is a parameter and the function $X \in {\cal B}^2(\R)$ with 
\begin{gather*}
0 \leq X'(\xi_{1}) \leq 1, \quad
X(\xi_1)=\left\{
\begin{array}{ll}
2 & \hbox{if} \quad \xi_1\geq 3, \\
\xi_1 & \hbox{if} \quad |\xi_1|\leq 1, \\
-2 & \hbox{if} \quad \xi_1\leq -3.
\end{array}\right.
\end{gather*}
Note that
\begin{gather*}
{\rm sgn}\Xi_1(\xi_1)={\rm sgn}\xi_1, \quad
|\Xi_1(\xi_1)|\leq |\xi_1|, \quad
\Xi(\xi_1)=\left\{
\begin{array}{ll}
2K & \hbox{if} \quad \xi_1\geq 3K, \\
\xi_1 & \hbox{if} \quad |\xi_1|\leq K, \\
-2K & \hbox{if} \quad \xi_1\leq -3K.
\end{array}\right.
\end{gather*}
The modified problem is written as
\begin{subequations}\label{l1BB}
\begin{gather}
\pd_{t} f 
+ \Xi_{1}\partial_{x} f 
+ \! (\partial_{x} \tilde{\Phi}^{s}+\partial_{x}\tilde{\phi}) \partial_{\xi_{1}}f 
-  \frac{\Xi_{1}\!-\!u_{\infty}}{2\theta_{\infty}}
(\partial_{x} \tilde{\Phi}^{s} + \partial_{x}\tilde{\phi})f 
\! = \! \tilde{F}, 
\ \ \! (t, x, \xi) \! \in [0,\hat{T}] \! \times \R \times \R^3\!\!,
\label{leq1BB}\\
f(0,x,\xi) = \tilde{f}_{0}(x,\xi), \ \ (x, \xi) \in  \R \times \R^3.
\label{libc1BB}
\end{gather}
\end{subequations}
This can be solved easily by a standard theory of hyperbolic equations. For more details, see subsection \ref{SSB_Estimate}.

This section is organized as follows. 
First we discuss the properties of solutions of the modified problem \eqref{l1BB} in subsection \ref{SSB_Estimate}.
Subsection \ref{SSB_Proof} provides a priori estimates of the solution of the linearized problem \eqref{l1}.
In particular, one of the a priori estimates immediately gives the uniqueness of solutions of \eqref{l1}.
In subsection \ref{SSB_Proof}, 
we seek a solution $f$ of the modified problem \eqref{l1BB} by assuming some stronger conditions than the original conditions \eqref{lF1} and \eqref{lini1} in Lemma \ref{sovl1}, and show that the solution $f$ also solves the linearized problem \eqref{l1} and satisfies \eqref{localpro1}. Then we relax the stronger conditions into the original conditions \eqref{lF1} and \eqref{lini1}.

In our proof, we use the next fundamental lemma.

\begin{lem}\label{lem_BLem}
Let $\alpha\geq 0$. Suppose that $f$ and $\{f_j\}$ satisfies
\begin{gather*}
e^{\alpha x/2}f_j\to e^{\alpha x/2}f \quad
\hbox{in} \quad
C([0,T_0];H^1_{x,\xi}) \quad
\hbox{as} \quad
j\to\infty. 
\end{gather*}
For any $k=0,1$ and $t\in[0,T_0]$, there holds that
\begin{align*}
& \int_0^te^{C_0(t-\tau)} \left\||\xi_1|^{1/2} \nabla_{x,\xi}^{k} 
(e^{\alpha x /2} f)(\tau) \chi(-\xi_{1})\right\|_{L_{x,\xi}^2}^2 \,d\tau 
\notag\\
& \leq \varliminf_{j\to 0}\int_0^te^{C_0(t-\tau)} 
\left\||\xi_1|^{1/2} \nabla_{x,\xi}^{k} 
(e^{\alpha x /2} f_j)(\tau) \chi(-\xi_{1})\right\|_{L_{x,\xi}^2}^2 \,d\tau.
\end{align*}
\end{lem}

\begin{proof}[Proof of Lemma \ref{lem_BLem}]
We prove only the case $k=0$, since another case $k=1$ can be shown similarly.
We observe by Fatou's lemma that for any $\ell>0$,
\begin{align*}
& \int_0^te^{C_0(t-\tau)} \left\||\xi_1|^{1/2} (e^{\alpha x /2} f)(\tau) 
\chi(-\xi_{1})\chi(\xi_1+\ell)\right\|_{L_{x,\xi}^2}^2 \,d\tau 
\notag\\
& \leq \varliminf_{j\to 0}\int_0^te^{C_0(t-\tau)} 
\left\||\xi_1|^{1/2} (e^{\alpha x /2} f_j)(\tau) 
\chi(-\xi_{1})\chi(\xi_1+\ell)\right\|_{L_{x,\xi}^2}^2 \,d\tau 
\notag\\
& \leq \varliminf_{j\to 0}\int_0^te^{C_0(t-\tau)} 
\left\||\xi_1|^{1/2} (e^{\alpha x /2} f_j)(\tau) 
\chi(-\xi_{1}) \right\|_{L_{x,\xi}^2}^2 \,d\tau. 
\end{align*}
Letting $\ell \to \infty$ and applying the monotone convergence theorem,
we arrive at the desired inequality. 
\end{proof}

\subsection{The modified problem}\label{SSB_Support}

The subsection is devoted to the study of the modified problem \eqref{l1BB}. 
We abbreviate the spaces $L^{2}({\mathbb R}\times{\mathbb R}^{3})$, 
and $H^{k}({\mathbb R}\times{\mathbb R}^{3})$
by $\tilde{L}^{2}_{x,\xi}$ and $\tilde{H}^{k}_{x,\xi}$, respectively. 

First we show the solvability.

\begin{lem}\label{lem_BExistence}
Let $\alpha\geq 0$. 
Suppose that $\tilde{\Phi}^{s}$ and $\tilde{\phi}$ satisfies \eqref{extension1}
as well as $\tilde{F}$ and $\tilde{f}_{0}$ satisfies
\begin{gather}\label{Ff0}
e^{\alpha x /2} \tilde{F} \in C([0,\hat{T}];\tilde{H}^{1}_{x,\xi}), \quad
e^{\alpha x /2} \tilde{f}_{0} \in \tilde{H}^{1}_{x,\xi} 
\end{gather}
for some constant $\hat{T}>0$. 
Then the initial value problem \eqref{l1BB}
has a unique solution $f$ that satisfies 
\begin{gather}
e^{\alpha x /2} f \in C([0,\hat{T}];\tilde{H}^{1}_{x,\xi})
\cap C^1([0,\hat{T}];\tilde{L}^{2}_{x,\xi}). 
\label{BClass_solution}
\end{gather}
\end{lem}

\begin{proof}
Set $g=e^{\alpha x/2}f$. It is easy to see that $g$ solves the following initial value problem:
\begin{gather*}
\pd_{t} g 
+ \Xi_{1}\partial_{x} g 
+ (\partial_{x} \tilde{\Phi}^{s}+\partial_{x}\tilde{\phi}) \partial_{\xi_{1}}g 
-  \frac{\Xi_{1}-u_{\infty}}{2\theta_{\infty}}
(\partial_{x} \tilde{\Phi}^{s} + \partial_{x}\tilde{\phi})g 
-\frac{\alpha}{\,2\,}\Xi_1 g
= e^{\alpha x/2}\tilde{F}, 
\\
g(0,x,\xi) = e^{\alpha x/2}\tilde{f}_{0}(x,\xi).
\end{gather*}
This problem is obviously  equivalent to \eqref{l1BB}.
Note that $\Xi_1\in {\cal B}^2(\R)$ and $\tilde{\Phi}^s\in {\cal B}^2(\R)$.
Furthermore, we see from Gagliardo--Nirenberg's inequality that 
$\tilde{\phi} \in C([0,\hat{T}];{\cal B}^2(\R))$.
Therefore, all the coefficients in \eqref{l1BB} are bounded.
A standard theory of hyperbolic equations (for instance, see Chapter 6 in \cite{S.M.1}) ensures the unique existence of solutions 
$g\in C([0,\hat{T}]; \tilde{H}^1_{x,\xi}) \cap C^1([0,\hat{T}]; \tilde{L}^2_{x,\xi})$.
The proof is complete.
\end{proof}

Let us investigate the location of the support of the solution $f$ of \eqref{l1BB}.
The results are summarized in Lemmas \ref{lem_BSupport1} and \ref{lem_BSupport2}.
Here the set $D(\tilde{s})$ and constant $c_{0}$ are defined as
\begin{gather}
D(\tilde{s}):=\{(x,\xi)\in\R\times\R^3\, \left| \,
|x|^2+|\xi_1|^2< \tilde{s}^2\right\}
\cup (\R_-\times(-\tilde{s},\infty)\times\R^2) \quad \text{for $\tilde{s} \geq 0$},
\label{BDefinition_D0}
\\
c_0:=\sup_{x\in\mathbb R}|\D_x\tilde{\Phi}^s(x)|
+\sup_{(t,x)\in[0,\hat{T}]\times\mathbb R}|\D_x\tilde{\phi}(t,x)|.
\label{BDefinition_c0}
\end{gather}

\begin{lem}\label{lem_BSupport1}
Let $\alpha \geq 0$ and $\tilde{s}>0$. 
Suppose that $\tilde{\Phi}^{s}$ and $\tilde{\phi}$ satisfies \eqref{extension1}
as well as $\tilde{F}$ and $\tilde{f}_{0}$ satisfies \eqref{Ff0} and
\begin{gather*}
\tilde{F}(t, x,\xi)=0, \quad (t, x, \xi) \in [0,\hat{T}] \times D(\tilde{s}),
\\
\tilde{f}_{0}(x,\xi)=0, \quad (x, \xi) \in D(\tilde{s})
\end{gather*}
for some $\hat{T}>0$.
If $f$ is a solution of the modified problem \eqref{l1BB} and satisfies
\begin{gather}\label{reg_modified}
e^{\alpha x/2} f\in C([0,\hat{T}];\tilde{L}^2_{x,\xi})
\cap L^\infty(0,\hat{T};\tilde{H}^1_{x,\xi}), \ \
e^{\alpha x/2} (1+|\xi_1|)^{-1}f\in W^{1,\infty}(0,\hat{T};\tilde{L}^2_{x,\xi}),
\end{gather}
then there holds that
\begin{gather}\label{BSupport_claim}
f=0 \quad \text{on $\Omega(T_0;s(t))$},
\end{gather}
where
\begin{gather*}
\Omega(t_0;s(t)):=\{(t,x,\xi)\in[0,t_0]\times\R\times\R^3\,\left|\,
(x,\xi)\in D(s(t))\right.\} \quad
\text{for \ $t_0\in[0,T_0]$},
\\
T_{0}=T_0(\tilde{s}):=\min\Bigl\{\hat{T},\,
2\log\Bigl(\frac{2c_0+\tilde{s}}{2c_0}\Bigr)\Bigr\}, \qquad
s(t):=(2c_0+\tilde{s})e^{-t/2}-2c_0.
\end{gather*}
\end{lem}

\begin{proof}
For any $\ell>\max\{\tilde{s},c_0\hat{T}\}$, $k>\hat{T}\ell$, and $t_0\in[0,T_0]$, 
we define $\Omega_{k,\ell}(t_0)$ by
\begin{gather*}
\Omega_{k,\ell}(t_0):=\Omega(t_0;s(t))
\cap\{(t,x,\xi)\in[0,t_0]\times\R\times\R^3\,|\,
x\geq -k+\ell t,\,\xi_1\leq \ell-c_0t\}.
\end{gather*}
Note that $\Omega_{k,\ell}(t_0)$ can be expressed as 
\begin{gather*}
\Omega_{k,\ell}(t_0)
=\{(t,x,\xi)\in[0,t_0]\times\R\times\R^3\,|\,
(x,\xi)\in D_{k,\ell}(t)\}, 
\\
D_{k,\ell}(t)
:=D(s(t))\cap\{(x,\xi)\in\R\times\R^3\,|\,
x\geq -k+\ell t,\,\xi_1\leq\ell-c_0t\}.
\end{gather*}
Furthermore, it is easy to see that $\partial\Omega_{k,\ell}(t_0)
=\Gamma_1\cup\Gamma_2\cup\cdots\cup\Gamma_7$, 
where
\begin{align*}
& \Gamma_1:=\{(t,x,\xi)\in\partial\Omega_{k,\ell}(t_0)\,|\,
t=t_0\}
=\{t_0\}\times D_{k,\ell}(t_0), 
\\
& \Gamma_2:=\{(t,x,\xi)\in\partial\Omega_{k,\ell}(t_0)\,|\,
t=0\}
=\{0\}\times D_{k,\ell}(0), 
\\
& \Gamma_3:=\{(t,x,\xi)\in\partial\Omega_{k,\ell}(t_0)\,|\,
\xi_1=-s(t)\}, 
\\
& \Gamma_4:=\{(t,x,\xi)\in\partial\Omega_{k,\ell}(t_0)\,|\,
\xi_1=\ell-c_0t\}, 
\\
& \Gamma_5:=\{(t,x,\xi)\in\partial\Omega_{k,\ell}(t_0)\,|\,
x=-k+\ell t\}, 
\\
& \Gamma_6:=\{(t,x,\xi)\in\partial\Omega_{k,\ell}(t_0)\,|\,
x=0,\,\xi_1\geq 0\}, 
\\
& \Gamma_7:=\{(t,x,\xi)\in\partial\Omega_{k,\ell}(t_0)\,|\,
|x|^2+|\xi_1|^2=s(t)^2\}. 
\end{align*}

It is clear that $\tilde{F}=0$ holds on $\Omega_{k,\ell}(t_0) \subset[0,\hat{T}]\times D(\tilde{s})$. 
Multiply \eqref{leq1BB} by $2f$ and integrate it over $\Omega_{k,\ell}(t_0)$ to obtain
\begin{align}
& \int_{\Omega_{k,\ell}(t_0)}
\partial_{t} \vert  f \vert^2
+ \partial_{x} (\Xi_{1} \vert f \vert^2)
+\partial_{\xi_{1}} 
\{(\pd_{x}\tilde{\Phi}^{s} +\pd_{x} \tilde{\phi}) \vert f \vert^2\}
dtdxd\xi
\notag\\
& =-\int_{\Omega_{k,\ell}(t_0)}
(\pd_{x}\tilde{\Phi}^{s} +\pd_{x} \tilde{\phi})
\frac{\Xi_1-u_\infty}{\theta_\infty}\vert f \vert^2dtdxd\xi
\notag\\
& \leq C \int_0^{t_0}\Bigl(\int_{D_{k,\ell}(t)}|f(t,x,\xi)|^2dxd\xi\Bigr)dt,
\label{BEquation}
\end{align}
where $C>0$ is a constant.
Applying the divergence theorem to the left hand side, we have 
\begin{gather}
\int_{\Omega_{k,\ell}(t_0)}
\partial_{t} \vert  f \vert^2
+ \partial_{x} (\Xi_{1} \vert f \vert^2)
+\partial_{\xi_{1}} 
\{(\pd_{x}\tilde{\Phi}^{s} +\pd_{x} \tilde{\phi}) \vert f \vert^2\}
dtdxd\xi
=\sum_{i=1}^7\int_{\Gamma_i}\vec{\cal F}\cdot\vec{n}\,dS,
\label{BEquation2}
\end{gather}
where $\vec{n}=\vec{n}(t,x,\xi)\in\R^5$  denotes the outer normal vector of $\partial\Omega_{k,\ell}(t_0)$, 
and the vector function $\vec{\cal F}=\vec{\cal F}(t,x,\xi)\in\R^5$ is defined as
\begin{gather*}
\vec{\cal F}(t,x,\xi) :=\Bigl(|f|^2,\,\Xi_1|f|^2,\,
(\D_x\tilde{\Phi}^s+\D_x\tilde{\phi})|f|^2,0,0\Bigr)(t,x,\xi).
\end{gather*}
We claim that the integrals on $\Gamma_{i}$ are nonnegative for $i =2,\ldots,7$.
From this claim, \eqref{BEquation}, and \eqref{BEquation2}, we obtain
\begin{gather*}
\int_{D_{k,\ell}(t_0)}|f(t_0,x,\xi)|^2dxd\xi 
= \int_{\Gamma_1}\vec{\cal F}\cdot\vec{n}\,dS
\leq C\int_0^{t_0}\Bigl(\int_{D_{k,\ell}(t)}|f(t,x,\xi)|^2dxd\xi\Bigr)dt,
\end{gather*}
which together with Gronwall's inequality gives
\begin{gather*}
\int_{\Omega_{k,\ell}(t_0)}|f(t,x,\xi)|^2dtdxd\xi
=\int_0^{t_0}\Bigl(\int_{D_{k,\ell}(t)}|f(t,x,\xi)|^2dxd\xi\Bigr)dt
=0. 
\end{gather*}
Therefore, \eqref{BSupport_claim} follows from the arbitrariness of $\ell$ and $k$.

We complete the proof by showing the claim.
By the assumption that $\tilde{f}_0=0$ on $D_{k,\ell}(0) \subset D(\tilde{s})$,
the integral on $\Gamma_{2}$ is written as 
\begin{gather*}
\int_{\Gamma_2}\vec{\cal F}\cdot\vec{n}\,dS
=-\int_{D_{k,\ell}(0)}|f(0,x,\xi)|^2dxd\xi
=-\int_{D_{k,\ell}(0)}|\tilde{f}_0(x,\xi)|^2dxd\xi
=0.
\end{gather*}
The integrant $\vec{\cal F}\cdot\vec{n}$ on $\Gamma_{3}$ is estimated by using $-s'(t)=s(t)/2+c_0$, \eqref{BDefinition_c0}, and the fact $s(t)\geq0$ if $t \leq T_{0}$ as
\begin{gather*}
\vec{\cal F}\cdot\vec{n}
=\frac{-s'(t)-(\D_x\tilde{\Phi}^s+\D_x\tilde{\phi})}{\sqrt{s'(t)^2+1}}|f|^2
=\frac{s(t)/2+c_0-(\D_x\tilde{\Phi}^s+\D_x\tilde{\phi})}{\sqrt{s'(t)^2+1}}|f|^2 \geq 0.
\end{gather*}
Thus we see that $\int_{\Gamma_3}\vec{\cal F}\cdot\vec{n}\,dS\geq 0$.
Similarly, $\int_{\Gamma_4}\vec{\cal F}\cdot\vec{n}\,dS\geq 0$ holds.
Concerning the integral on $\Gamma_{5}$, we observe from
the fact  $|\Xi_1|\leq|\xi_1|\leq \ell -c_{0} t  \leq \ell$ that
\begin{gather*}
\int_{\Gamma_5} \vec{\cal F}\cdot\vec{n} \,dS
=\int_{\Gamma_5} \frac{\ell-\Xi_1}{\sqrt{\ell^2+1}}|f|^2 \,dS \geq 0.
\end{gather*}
It is clear that $\int_{\Gamma_6}\vec{\cal F}\cdot\vec{n}\,dS\geq 0$.
The integrant $\vec{\cal F}\cdot\vec{n}$ on $\Gamma_{7}$ can be estimated as
\begin{align}
\vec{\cal F}\cdot\vec{n}
&=\frac{1}{\sqrt{s'(t)^2+1}}\Bigl\{-s'(t)+\frac{x\Xi_1}{s(t)}
+\frac{\xi_1}{s(t)}(\D_x\tilde{\Phi}^s+\D_x\tilde{\phi})\Bigr\}|f|^2
\notag\\
&\geq \frac{1}{\sqrt{s'(t)^2+1}}\left\{-s'(t) -\left |\frac{x\Xi_1}{s(t)}\right|
- \left|\frac{\xi_1}{s(t)}(\D_x\tilde{\Phi}^s+\D_x\tilde{\phi})\right|\right\}|f|^2.
\label{Gamma7}
\end{align}
It is seen that for $(t,x,\xi)\in\Gamma_7$, 
\begin{gather*}
\Bigl|\frac{x\Xi_1}{s(t)}\Bigr|
\leq\frac{|x||\xi_1|}{s(t)}
\leq\frac{|x|^2+|\xi_1|^2}{2s(t)}
=s(t)/2, 
\\
\Bigl|\frac{\xi_1}{s(t)}(\D_x\tilde{\Phi}^s+\D_x\tilde{\phi})\Bigr|
\leq\frac{|\xi_1|}{s(t)}(|\D_x\tilde{\Phi}^s|+|\D_x\tilde{\phi}|)
\leq |\D_x\tilde{\Phi}^s|+|\D_x\tilde{\phi}|
\leq c_0. 
\end{gather*}
Plugging these and $-s'(t)=s(t)/2+c_0$ into \eqref{Gamma7}, 
we arrive at $\int_{\Gamma_7}\vec{\cal F}\cdot\vec{n}\,dS\geq 0$.
The proof is complete.
\end{proof}

\begin{lem}\label{lem_BSupport2}
Let $\alpha \geq 0$. 
Suppose that $\tilde{\Phi}^{s}$ and $\tilde{\phi}$ satisfies \eqref{extension1}
as well as $\tilde{F}$ and $\tilde{f}_{0}$ satisfies \eqref{Ff0} and
\begin{gather*}
{\rm supp}\tilde{F} 
\subset [0,\hat{T}]\times\R\times[-L,R]\times\R^2, \qquad
{\rm supp}\tilde{f}_0\subset \R\times[-L,R]\times\R^2,
\end{gather*}
where $R,L\in\R$ with $-L<R$.
If $f$ is a solution of the modified problem \eqref{l1BB} and satisfies \eqref{reg_modified},
then there holds that
\begin{gather*}
{\rm supp}f
\subset\{(t,x,\xi)\in [0,\hat{T}]\times\R\times\R^3 \,|\,
-L-c_0t\leq\xi_1\leq R+c_0t\},
\end{gather*}
where $c_{0}$ is defined in \eqref{BDefinition_c0}. 
\end{lem}
\begin{proof}
We can show this lemma similarly as the proof of Lemma \ref{lem_BSupport1}.
\end{proof}

\subsection{A priori estimates of solutions of the linearized problem}\label{SSB_Estimate}

In this subsection, we drive a priori estimates of solutions of the linearized problem \eqref{l1}.

\begin{lem}\label{lem_BSupport20}
Let $\alpha \geq 0$, $\hat{T}>0$, $\hat{\phi} \in C([0,\hat{T}];L^{2}_{x}) \cap L^{\infty}(0,\hat{T};H^{3}_x)$,
$e^{\alpha x /2} \hat{F} \in C([0,\hat{T}];H^{1}_{0,\Gamma})$, and
$e^{\alpha x /2} f_{0} \in H^{1}_{0,\Gamma}$.
Suppose that $\hat{F}$ and $f_{0}$ satisfies
\begin{gather*}
{\rm supp}\hat{F}
\subset \overline{ [0,\hat{T}]\times\R_+\times(-\infty,R]\times\R^2 }, \quad
{\rm supp}f_0\subset \overline{\R_+\times(-\infty,R]\times\R^2}
\end{gather*}
for some $R\in\R$.
If $f$ is a solution of the linearized problem \eqref{l1} and satisfies 
\begin{gather*}
e^{\alpha x /2} f\in C([0,\hat{T}];L^2_{x,\xi})
\cap L^\infty(0,\hat{T};H^1_{0,\Gamma}), \quad
e^{\alpha x /2} (1+|\xi_1|)^{-1}f\in W^{1,\infty}(0,\hat{T};L^2_{x,\xi}),
\end{gather*}
then there holds that
\begin{gather*}
{\rm supp}f
\subset\{(t,x,\xi)\in \overline{ [0,\hat{T}]\times\R_+\times\R^3} \,|\,
\xi_1\leq R+\hat{c}_0t\},
\end{gather*}
where the constant $\hat{c}_0$ is defined as
\begin{gather*}
\hat{c}_0:=\sup_{x\in\overline{\mathbb R_+}}|\D_x\Phi^s(x)|
+\sup_{(t,x)\in[0,\hat{T}]\times\overline{\mathbb R_+}}|\D_x\hat{\phi}(t,x)|. 
\end{gather*}
\end{lem}
\begin{proof}
We omit the proof, since it is easier than that of Lemma \ref{lem_BSupport1}.
\end{proof}

\begin{rem}\label{remB}
{\rm 
The uniqueness stated in Lemma \ref{sovl1} immediately follows from Lemma \ref{lem_BSupport20}.
}
\end{rem}

We derive another a priori estimate. 
To this end, we define the set $\Gamma$ and the function space $H^2_{0,\Gamma}$ as follows:
\begin{align}
\Gamma&:=\{0\}\times\overline{\R_+^3}
=\{(x,\xi)\in\R\times\R^3\,|\,
x=0,\,\xi_1\geq 0\},
\label{BDefinition_Gamma}\\
H^2_{0,\Gamma}&:=\overline{\ C^\infty_{0,\Gamma} \ }^{{}_{H^2_{x,\xi}}}, \quad
C^\infty_{0,\Gamma}
:=\{f\in C_0^\infty(\overline{\R_+\times\R^3})\,|\,
{\rm supp}f\cap\Gamma=\emptyset\}.
\notag
\end{align}

\begin{lem}\label{lem_BEstimate}
Let $\alpha>0$, $R_0>0$, and $L_0>0$. 
Suppose that $\hat{\phi} \in C([0,\hat{T}];L^{2}_{x}) \cap L^{\infty}(0,\hat{T};H^{3}_x)$,
$e^{\alpha x /2} \hat{F} \in C([0,\hat{T}];H^{1}_{0,\Gamma})$, and
$e^{\alpha x /2} f_{0} \in H^{1}_{0,\Gamma}$ for some $\hat{T}>0$.
Assume that $f$ is a solution of the problem \eqref{l1} and satisfies
\begin{gather}
e^{\alpha x/2}f\in C([0,T_0];H^2_{0,\Gamma})
\cap C^1([0,T_0];H^1_{x,\xi}), 
\notag\\
{\rm supp}f
\subset\{(t,x,\xi)
\in\overline{[0,T_0]\times\R_+\times\R^3}\,|\,
-L_0\leq\xi_1\leq R_0\}. 
\label{Baes0}
\end{gather}
Then there exists a constant $\delta_{0} \in (0,1]$ depending only on $\alpha$, $\theta_\infty$, and $C$ being in \eqref{decay0}
such that if $\Phi_{b} \leq \delta_{0}$ and $\sup_{t \in [0,\hat{T}]} \| \hat{\phi} (t) \|_{H^{3}_{x}} \leq \delta_{0}$, the solution $f$ satisfies \eqref{localpro1c} in which a constant $C_0$ is independent of $L_0$.
\end{lem}

\begin{proof}
First it is seen from \eqref{decay0} that
\begin{gather}\label{phismall}
\sum_{k=0}^2\Bigl(\sup_{x\in\mathbb R} |\D_x^k{\Phi}^s(x)|
+\sup_{(t,x)\in[0,T]\times\mathbb R}|\D_x^k\hat{\phi}(t,x)|\Bigr)
\leq 3(C+C_{s}) \delta_{0}=:C_{*}\delta_{0},
\end{gather}
where $C>0$ is the same constant being in \eqref{decay0} and $C_{s}$ is the best constant of Sobolev's inequality.
Multiplying \eqref{leq1} by $2e^{\alpha x} f$, integrating it over $\R_+\times\R^3$, and using \eqref{libc1},
we have
\begin{align}
&\dfrac{d}{dt} \Vert e^{\alpha x/2} f (t) \Vert^2_{L^2_{x,\xi}} 
+\left. \int_{\mathbb{R}^3_-} 
\vert \xi_1\vert \vert  f \vert^2 d\xi \right|_{x=0}
+\alpha\int_{\mathbb R_{+}}\int_{\mathbb R_-^3}|\xi_1|
e^{\alpha x}|f|^2 dxd\xi
\notag \\
& 
=\alpha\int_{\mathbb R_{+}}\int_{\mathbb R^3_{+}}\xi_1e^{\alpha x/2}|f|^2 dxd\xi
-\int_{\mathbb R_{+}}\int_{\mathbb R^3}e^{\alpha x}(\D_x\Phi^s+\D_x\hat{\phi})
\frac{\xi_1-u_\infty}{\theta_\infty}\vert f \vert^2 dxd\xi
\notag \\
& \quad +2\int_{\mathbb R_{+}}\int_{\mathbb R^3}e^{\alpha x}\hat{F}fdxd\xi
\notag \\
&=:I_1+I_2+I_3.
\label{Baes1}
\end{align}
Let us estimate the terms $I_{1}$, $I_{2}$, and $I_{3}$. From \eqref{Baes0}, it is seen that
\begin{align*}
I_1 
& \leq \alpha R_0\int_{\mathbb R_{+}}\int_{\mathbb R_+^3} e^{\alpha x}|f|^2 dxd\xi
\leq \alpha R_0\|e^{\alpha x/2}f\|_{L^2_{x,\xi}}^2.
\end{align*}
We observe from \eqref{Baes0} and \eqref{phismall} that
\begin{align*}
I_2
& \leq \frac{R_{0}C_{*}\delta_{0}}{\theta_\infty}\int_{\mathbb R_{+}}\int_{\mathbb R_+^3} e^{\alpha x}|f|^2 dxd\xi
+ \frac{C_{*}\delta_{0}}{\theta_\infty}\int_{\mathbb R_{+}}\int_{\mathbb R_-^3} |\xi_1| e^{\alpha x}|f|^2 dxd\xi
\\
& \quad +\frac{|u_\infty|C_{*}\delta_{0}}{\theta_\infty}\int_{\mathbb R_{+}}\int_{\mathbb R^3} e^{\alpha x}|f|^2 dxd\xi
\\
& \leq\frac{(|u_\infty|+R_0)C_{*}\delta_{0}}{\theta_\infty} \|e^{\alpha x/2}f\|_{L^2_{x,\xi}}^2
+\frac{C_{*}\delta_{0}}{\theta_\infty} \left\||\xi_1|^{1/2}(e^{\alpha x/2}f)\chi(-\xi_1)\right\|_{L^2_{x,\xi}}^2.
\end{align*}
Moreover, it is clear that 
$I_3\leq \|e^{\alpha x/2}f\|_{L^2_{x,\xi}}^2+\|e^{\alpha x/2}\hat{F}\|_{L^2_{x,\xi}}^2$.
Substituting these into \eqref{Baes1} gives
\begin{align}
&\dfrac{d}{dt} \Vert e^{\alpha x/2} f (t) \Vert^2_{L^2_{x,\xi}} 
+\alpha
\left \| |\xi_1|^{1/2}e^{\alpha x/2} f (t)\chi(-\xi_1) \right\|_{L^2_{x,\xi}}^2
\notag\\
& \leq \frac{C_{*}\delta_{0}}{\theta_\infty}
\left\| |\xi_1|^{1/2}(e^{\alpha x/2} f) (t)\chi(-\xi_1) \right\|_{L^2_{x,\xi}}^2
+C_0\|e^{\alpha x/2}f(t)\|_{L^2_{x,\xi}}^2
+\|e^{\alpha x/2}\hat{F}(t)\|_{L^2_{x,\xi}}^2,
\label{BEstimate0}
\end{align}
where $C_{0}>0$ is a constant depending on $R_{0}$ but independent of $L_{0}$.

We multiply \eqref{leq1} by $e^{\alpha x/2}$, apply the operator $\nabla_{x,\xi}$, 
and multiply it by $2\nabla_{x,\xi}(e^{\alpha x/2} f)$.
Furthermore, we integrate the resultant equality over $\R_+\times\R^3$, 
use the boundary condition $\nabla_{x,\xi}(e^{\alpha x/2}f)(t,0,\xi)\chi(\xi_{1})=0$ 
which comes from $f(t)\in H^2_{0,\Gamma}$,
and then estimate similarly as above to obtain
\begin{align}
&\dfrac{d}{dt} \Vert \nabla_{x,\xi}(e^{\alpha x/2} f)(t) \Vert^2_{L^2_{x,\xi}} 
+\alpha
\left\| |\xi_1|^{1/2}
\nabla_{x,\xi}(e^{\alpha x/2} f) (t)\chi(-\xi_1) \right\|_{L^2_{x,\xi}}^2
\notag\\
& \leq \frac{C_{*}\delta_{0}}{\theta_\infty} \sum_{k=0}^1
\left\| |\xi_1|^{1/2} \nabla_{x,\xi}^k(e^{\alpha x/2} f) (t)\chi(-\xi_1) \right\|_{L^2_{x,\xi}}^2
+C_0\|e^{\alpha x/2}f(t)\|_{H^1_{x,\xi}}^2
\notag\\
& 
\quad +\|\nabla_{x,\xi}(e^{\alpha x/2}\hat{F})(t)\|_{L^2_{x,\xi}}^2.
\label{BEstimate1}
\end{align}
Combining \eqref{BEstimate0} and \eqref{BEstimate1} leads to
\begin{align*}
&\dfrac{d}{dt} \Vert e^{\alpha x/2} f(t) \Vert^2_{H^1_{x,\xi}} 
+\Bigl(\alpha-\frac{2C_{*}\delta_{0}}{\theta_\infty}\Bigr)
\sum_{k=0}^1\bigl\| |\xi_1|^{1/2}
\nabla_{x,\xi}^k(e^{\alpha x/2} f) (t)\chi(-\xi_1) \bigr\|_{L^2_{x,\xi}}^2
\notag\\
& \leq C_0\|e^{\alpha x/2}f(t)\|_{H^1_{x,\xi}}^2
+\|e^{\alpha x/2}\hat{F}(t)\|_{H^1_{x,\xi}}^2.
\end{align*}
Letting $\delta_{0}$ be small enough and applying Gronwall's inequality, 
we conclude \eqref{localpro1c}.
\end{proof}

\subsection{The solvability of the linearized problem}\label{SSB_Proof}

In this section, we show Lemma \ref{l1} on the solvability of the linearized problem \eqref{l1}.
Lemma \ref{lem_BSupport20} immediately ensures the uniqueness of the solution.
What is left is to show the existence and the properties \eqref{localpro1}.

\begin{proof}[Proof of Lemma \ref{sovl1}]
Let $\alpha>0$, $s>0$, $R>0$, and \eqref{lphi1} holds.
We use the extensions $\tilde{\Phi}^{s}$ and $\tilde{\phi}$ in \eqref{extension1}.
Our proof is divided into two steps.
In the first step, we show the existence of solutions of the linearized problem \eqref{l1}
and the properties \eqref{localpro1} 
by assuming the following stronger conditions than \eqref{lF1} and \eqref{lini1}:
\begin{gather*}
e^{\alpha x/2}\hat{F}\in C([0,\hat{T}];H^1_{0,\Gamma}), 
\quad
{\rm supp}\hat{F}  \subset [0,\hat{T}]\times \Omega_{s,d,L,R},
\\
e^{\alpha x/2}f_0\in H^1_{0,\Gamma}, 
\quad
{\rm supp}f_0 \subset \Omega_{s,d,L,R},
\end{gather*}
where the domain $\Omega_{s,d,L,R}$ is defined for $d>0$, $L>0$, and $\Gamma$ in \eqref{BDefinition_Gamma} as
\begin{gather*}
\Omega_{s,d,L,R}:=\{(x,\xi)\in\overline{\R_+\times\R^3}\,|\,
|x|^2+|\xi_1|^2\geq s^2,\,
{\rm dist}(\{(x,\xi)\},\Gamma)\geq d,\,-L\leq\xi_1\leq R\}.
\end{gather*}
In the second step, we show the existence and properties \eqref{localpro1} 
under the original assumptions \eqref{lF1} and \eqref{lini1} in Lemma \ref{sovl1}
by making use of the conclusion of the first step.

\medskip

\noindent
\textsc{First step}: \ 
Suppose that $\hat{F}$ and $f_{0}$ satisfies the above stronger conditions.
We show that the linearized problem \eqref{l1} has a solution with \eqref{Class_solution} and \eqref{localpro1}.
First we choose extensions $\tilde{F}=\tilde{F}(t,x,\xi)$ and $\tilde{f}_0=\tilde{f}_0(x,\xi)$
of $\hat{F}$ and $f_{0}$ such that
\begin{gather*}
e^{\alpha x/2}\tilde{F}\in C([0,\hat{T}];\tilde{H}^1_{x,\xi}), \quad
\tilde{F}(t,x,\xi)=\hat{F}(t,x,\xi) \quad
\hbox{if $x>0$}, 
\\
{\rm supp}\tilde{F}
\subset\{(t,x,\xi)\in[0,\hat{T}]\times\R\times\R^3\,|\,
(x,\xi)\in D(s/2)^c,\,
-2L<\xi_1<2R\}, \\
e^{\alpha x/2}\tilde{f}_0\in \tilde{H}^1_{x,\xi}, \quad
\tilde{f}_0(x,\xi)=f_0(x,\xi) \quad
\hbox{if $x>0$}, 
\\
{\rm supp}\tilde{f}_0
\subset\{(x,\xi)\in\R\times\R^3\,|\,
(x,\xi)\in D(s/2)^c,\,
-2L<\xi_1<2R\},
\end{gather*}
where the set $D(\tilde{s})$ is defined in \eqref{BDefinition_D0}.
Furthermore, for the constant $c_0>0$ defined in \eqref{BDefinition_c0}, 
we take the parameter $K$ of the function $\Xi_1=\Xi_1(\xi_1;K)$ defined in \eqref{Xi1} as 
\begin{gather*}
K:=\max\{2R+c_0\hat{T},2L+c_0\hat{T}\}+1.
\end{gather*}
Now applying Lemma \ref{lem_BExistence}, 
we have a unique solution $f$ with \eqref{BClass_solution} of the modified problem \eqref{l1BB}.
Moreover, Lemma \ref{lem_BSupport1} with $\tilde{s}=s/2$ and Lemma \ref{lem_BSupport2} ensure that
\begin{gather}
f=0 \quad
\hbox{on} \quad
[0,T_0]\times\R_-\times\R_+^3,
\label{Support_solution}
\\
f=0 \quad
\hbox{on} \quad
[0,T_0]\times\{(x,\xi)\in\R\times\R^3\,|\,
|\xi_1|\geq K-1\}
\label{Support_solution2},
\end{gather} 
where $T_{0}=T_{0}(s/2)$ is defined in Lemma \ref{lem_BSupport1}.

From now on we prove that the solution $f$ (restricted its domain into $ [0,T_{0}]  \times \R_{+} \times \R^3$) solves the linearized problem \eqref{l1}, and satisfies \eqref{Class_solution} and \eqref{localpro1}. We see that $f$ satisfies the conditions $f(0,x,\xi)=f_0(x,\xi)$ and $\lim_{x\to\infty}f(t,x,\xi)=0$
and the regularity \eqref{Class_solution}, since $f$ is a solution of \eqref{l1BB} with \eqref{BClass_solution}, and $\tilde{f}_0(x,\xi)=f_0(x,\xi)$ if $x>0$. 
The boundary condition $f(t,0,\xi)=0$ for $\xi\in\R_+^3$ holds from \eqref{Support_solution}.
Let us show that $f$ solves \eqref{leq1}.
It follows from \eqref{Support_solution2} that $\Xi_1=\Xi_1(\xi_1;K)=\xi_1$ on $\supp f \cup \supp \hat{F}$.
Hence, $f$ solves the equation replaced $\Xi_{1}$ by $\xi_1$ in \eqref{leq1BB},
and the replaced equation is exactly the same as \eqref{leq1}. 

What is left is to show \eqref{localpro1}. 
First Lemma \ref{lem_BSupport20} ensures \eqref{localpro1b}.
It is obvious that \eqref{localpro1a} follows from 
\eqref{Class_solution}, \eqref{localpro1b}, and \eqref{localpro1c}.
Therefore, it suffices to show \eqref{localpro1c}.
Recall that $f$ satisfies  \eqref{localpro1b}, \eqref{BClass_solution}, \eqref{Support_solution}, and \eqref{Support_solution2}.
Using a standard modifier, we can find a sequence $\{f_j\}$ such that
$e^{\alpha x/2}f_j\in C([0,T_0];H^2_{0,\Gamma}) \cap C^1([0,T_0];H^1_{x,\xi})$,
\begin{gather*}
\supp f_{j}  \subset\{(t,x,\xi)
\in\overline{[0,T_0]\times\R_+\times\R^3}\,|\,
-K \leq\xi_1\leq R+C_{0}T_{0}+1\},
\end{gather*}
and as $j \to \infty$,
\begin{align*}
e^{\alpha x/2}f_j\to e^{\alpha x/2}f \quad
& \hbox{in} \quad
C([0,T_0];H^1_{x,\xi})
\cap C^1([0,T_0];L^2_{x,\xi}),
\\
e^{\alpha x/2} {\cal L} f_{j} \to e^{\alpha x/2}\hat{F}  \quad
& \hbox{in} \quad
C([0,T_0];H^1_{x,\xi}),
\end{align*}
where the operator $\cal L$ is defined in \eqref{leq1}.
Applying Lemma \ref{lem_BEstimate} to $f_{j}$,
and then taking $\varliminf_{j\to 0}$ with the aid of Lemma \ref{lem_BLem},
we conclude \eqref{localpro1c}.

\medskip

\noindent
\textsc{Second step}: \ 
Suppose that $\hat{F}$ and $f_{0}$ satisfies 
the original assumptions \eqref{lF1} and \eqref{lini1} in Lemma \ref{sovl1}.
We can find sequences $\{\hat{F}_j\}$ and $\{f_{0,j}\}$ such that
\begin{align*}
& e^{\alpha x/2}\hat{F}_j\in C([0,\hat{T}];H^1_{0,\Gamma}), \quad 
 {\rm supp}\hat{F}_j \subset [0,\hat{T}] \times \Omega_{s/2,d_{j},L_{j},2R}, \quad
\\
& e^{\alpha x/2}\hat{F}_j\to e^{\alpha x/2}\hat{F} \ \
\hbox{in} \ \
C([0,\hat{T}];H^1_{x,\xi}) \ \ 
\hbox{as} \ \ 
j\to\infty, 
\\
& e^{\alpha x/2}f_{0,j}\in H^1_{0,\Gamma}, \quad
 {\rm supp}f_{0,j} \subset \Omega_{s/2,d_{j},L_{j},2R}, \quad
e^{\alpha x/2}f_{0,j}\to e^{\alpha x/2}f_0 \ \
\hbox{in} \ \
H^1_{x,\xi} \ \
\hbox{as} \ \
j\to\infty, 
\end{align*}
where $d_j>0$ with $\lim_{j\to\infty} d_{j} =0$ and $L_j>0$ with $\lim_{j\to\infty} L_{j} = \infty$.
Now we see by the conclusion of the first step that the linearized problem \eqref{l1} with $\hat{F}=\hat{F}_j$ and $f_{0}=f_{0,j}$ has a solution $f_j$ with
$e^{\alpha x /2} f_j \in C([0,T_0];H^{1}_{x,\xi})
\cap C^1([0,T_0];L^{2}_{x,\xi})$
for $T_{0}=T_{0}(s/4)$.
Applying Lemma \ref{lem_BEstimate} to the difference $f_j-f_\ell$
with the aid of the modifier used in the end of the first step,
we see that $\{e^{\alpha x/2}f_j\}$ is a Cauchy sequence in $C([0,T_0];H^1_{x,\xi})$,
and hence it converges to some function $g\in C([0,T_0];H^1_{x,\xi})$.
Let $f:=e^{-\alpha x/2}g$. It is easy to know that $f$ is a solution with \eqref{Class_solution} 
of the linearized problem \eqref{l1}.

It remains to show \eqref{localpro1}. 
Lemma \ref{lem_BSupport20} ensures \eqref{localpro1b}.
It is obvious that \eqref{localpro1a} follows from 
\eqref{Class_solution}, \eqref{localpro1b}, and \eqref{localpro1c}.
Therefore, it suffices to show \eqref{localpro1c}.
By the conclusion of the first step, we have the estimate
\eqref{localpro1c} with $f=f_j$, $\hat{F}=\hat{F}_j$, and $f_0=f_{0,j}$.
Then taking $\varliminf_{j\to 0}$ with the aid of Lemma \ref{lem_BLem},
we conclude \eqref{localpro1c}.
The proof is complete.
\end{proof}

\end{appendix}

\providecommand{\bysame}{\leavevmode\hbox to3em{\hrulefill}\thinspace}
\providecommand{\MR}{\relax\ifhmode\unskip\space\fi MR }

\end{document}